\documentclass{amsart}

\usepackage{amssymb, amsfonts, amsmath, amsthm,bm, amscd}
\usepackage[dvipsnames]{xcolor}
\usepackage{mathrsfs}
\usepackage{graphicx}
\usepackage{caption}
\usepackage{subcaption}
\usepackage{wrapfig}
\usepackage[margin=1.5in]{geometry}
\usepackage{enumitem, multicol}
\usepackage{tikz}
\usetikzlibrary{cd}
\usepackage{nicefrac, bigints}
\usepackage{array}
\usepackage{mathtools} 
\usepackage{extarrows} 
\usepackage{amsbsy}
\usepackage{mathdots}

\usepackage[colorlinks,citecolor=cyan,linkcolor=teal]{hyperref}

\theoremstyle{definition}
\newtheorem{theorem}{Theorem}[section]
\newtheorem{maintheorem}{Theorem}

\newtheorem{maincor}[maintheorem]{Corollary}
\newtheorem{definition}[theorem]{Definition}
\newtheorem{lemma}[theorem]{Lemma}

\newtheorem{corollary}[theorem]{Corollary}
\newtheorem{proposition}[theorem]{Proposition}
\newtheorem{remark}[theorem]{Remark}

\newtheorem{claim}[theorem]{Claim}

\newtheorem{construction}[theorem]{Construction}
\newtheorem{example}[theorem]{Example}
\newtheorem*{heurthm}{Heuristic Theorem}

\makeatletter
\DeclareRobustCommand{\cev}[1]{%
  {\mathpalette\do@cev{#1}}%
}
\newcommand{\do@cev}[2]{%
  \vbox{\offinterlineskip
    \sbox\z@{$\m@th#1 x$}%
    \ialign{##\cr
      \hidewidth\reflectbox{$\m@th#1\vec{}\mkern4mu$}\hidewidth\cr
      \noalign{\kern-\ht\z@}
      $\m@th#1#2$\cr
    }%
  }%
}
\makeatother

\newcommand{\RR}{\mathbb{R}}

\newcommand{\HH}{\mathbb{H}}

\DeclareMathOperator{\Mod}{Mod}

\DeclareMathOperator{\sech}{sech}
\DeclareMathOperator{\Lip}{Lip}

\newcommand{\T}{\mathcal{T}}

\DeclareMathOperator{\Sp}{\mathsf{Sp}}

\newcommand{\cO}{\mathcal{O}}

\DeclareMathOperator{\Ol}{\mathcal{O}_\lambda}

\DeclareMathOperator{\rad}{rad}

\newcommand{\arc}{\underline{\alpha}}
\newcommand{\arcwt}{\underline{A}}
\newcommand{\arcb}{\underline{\smash{\beta}}}
\newcommand{\arcwtb}{\underline{B}}

\newcommand{\Base}{\mathscr{B}(S \setminus \lambda)}

\DeclareMathOperator{\res}{res}

\DeclareMathOperator{\reg}{reg}

\newcommand{\cT}{\mathcal T}

\DeclareMathOperator{\Out}{Out}
\DeclareMathOperator{\In}{In}

\newcommand{\rg}{\mathsf{Y}}
\newcommand{\GL}{\mathcal{GL}}
\DeclareMathOperator{\GLcr}{\mathcal{GL}^{cr}}
\newcommand{\TRG}{\mathcal{TRG}}

\usepackage{todonotes} 

\DeclareMathOperator{\sys}{sys}

\newcommand{\dfl}{\mathcal D}

\begin{document}

\title[Deflations and optimal Lipschitz maps]{Deflating hyperbolic surfaces and the shapes of optimal Lipschitz maps}

 \author[Calderon]{Aaron Calderon}
  \address{\parbox{.95\linewidth}{Aaron Calderon, Department of Mathematics, University of Chicago
  \vspace{1mm}}}
\email{aaroncalderon@uchicago.edu}

 \author[Tao]{Jing Tao}
  \address{\parbox{.95\linewidth}{Jing Tao, Department of Mathematics, University of Oklahoma
  \vspace{1mm}}}
\email{jing@ou.edu}

\setcounter{tocdepth}{1}

\begin{abstract}
Given two hyperbolic surfaces and a homotopy class of maps between them, Thurston proved that there always exists a representative minimizing the Lipschitz constant.
While not unique, these minimizers are rigid along a geodesic lamination.
In this paper, we investigate what happens in the complement of that lamination.
To do this, we introduce {\em deflations}, certain optimal maps to trees which can be used to obstruct optimal maps between surfaces.
Using a smooth version of the orthogeodesic foliation of the first author and Farre, we also construct many new families of optimal maps, showing that the obstructions coming from deflations are essentially the only ones.
\end{abstract}

\maketitle
\thispagestyle{empty}

\vspace{-5ex}

\section{Introduction}\label{sec:intro}
Understanding deformations of geometric structures on a Riemann surface $S$ is central to low-dimensional geometry and topology.
Deformations are organized by the Teichm{\"u}ller space $\T(S)$, the parameter space of marked complex/conformal/hyperbolic structures on $S$.
Given two points $X, X' \in \T(S)$, the change-of-marking map determines a homotopy class of maps $X \to X',$ and it is often important to find the ``best'' representative of this class.

In the complex setting, Teichm{\"u}ller proved there is a unique quasiconformal map between two marked Riemann surfaces minimizing the dilatation
\cite{Teich}.
In the Riemannian setting, work of Eells--Sampson and Hartman \cite{ESharmonic, Hartman} yields a unique harmonic representative (which minimizes Dirichlet energy).
In the hyperbolic setting, Thurston proved there always exists a map $X \to X'$ with minimal Lipschitz constant, but these representatives are not unique \cite{Th_stretch}. Any map realizing the minimal Lipschitz constant is called {\bf tight}.

Despite this flexibility, Thurston showed that tight maps all share a common structure (see also \cite{GK}).
The intersection of the stretch loci over all tight maps $X \to X'$ forms a chain-recurrent geodesic lamination $\lambda=\lambda(X,X')$ called the {\bf tension lamination}.

The present work develops Thurston’s picture further by analyzing tight maps on the complement of their tension laminations.
Given a tight map $X \to X'$ with tension lamination $\lambda$, cutting along $\lambda$ produces a map
$X \setminus \lambda \to X' \setminus \lambda$
that stretches each boundary component affinely by the Lipschitz constant.
We call such maps \textbf{boundary-tight}.
This notion isolates the part of the map where flexibility remains: the stretch along $\lambda$ is prescribed, while the complementary pieces may admit many families of shapes compatible with the same boundary behavior.

Our main results give a structural description of all boundary-tight maps.
The following captures the core phenomenon and motivates the technical statements below. 
\begin{heurthm}
For any $L$-Lipschitz boundary-tight map $X \setminus \lambda \to X' \setminus \lambda,$
\[X \setminus \lambda \approx (1/L) \cdot (X' \setminus \lambda).\]
Moreover, this is essentially a complete description: there are boundary-tight maps where the domain and target can have any possible shapes subject to this constraint.
\end{heurthm}

Precise statements appear as Theorems \ref{mainthm:deflate} (actually, Corollary \ref{cor:rescale}) and \ref{mainthm:itin}, and even more detailed results can be found in the body of the text.

\subsection{Regular polygons}
Let us first consider the special case of boundary-tight maps between ideal polygons. In this situation, our results can be phrased very sharply.

The prototypical example of a boundary-tight map goes back to Thurston in \cite{Th_stretch}, who constructed self-maps of an ideal triangle $\Delta$ with arbitrary Lipschitz constant $L \ge 1$.
Choose the unique triple of horocycles based at the ideal vertices of $\Delta$ that are pairwise tangent; the union forms a central horocyclic triangle $H$ invariant under the full dihedral symmetry group $D_6$ of $\Delta$.
The complement of $H$ is a union of three spikes, each equipped with an orthogonal foliation by horocycles and geodesic rays based at the vertices of $\Delta$.
Given $L\ge 1$, one can stretch the geodesic rays in each spike uniformly by factor $L$. This preserves the horocyclic foliation, sending the horocycle at distance $d$ from $\partial H$ to the horocycle at distance $Ld$ by an affine map.
Gluing by the identity map across $H$ yields a self-map of $\Delta$ that restricts to an $L$-homothety on $\partial \Delta$.
Since the lengths of horocycles do not grow under this map (in fact, they decay exponentially quickly), one verifies that this map expands maximally in the orthogonal direction to the horocyclic foliation, hence it is boundary-tight with Lipschitz constant $L$.

This construction clearly generalizes to other regular\footnote{By a {\bf regular} $n$-gon we mean one with the full dihedral group $D_{2n}$ of symmetries.} ideal $n$-gons (see \cite[\S15]{shshI} and \cite{HT:stretch}).
The central horocyclic triangle is now replaced by a regular horocyclic $n$-gon, and one obtains a boundary-tight map by fixing this central polygon and stretching in the spikes as before. 
More generally, for an ideal polygon $P$, we define an {\bf (immersed) horogon} to be a choice of $n$ horocycles, one based at each ideal vertex of $P$, such that for each geodesic side of $P$, the two horocycles based at its endpoints are tangent.
We observe that if $n$ is odd, every ideal $n$-gon contains a unique immersed horogon, while if $n$ is even, then an ideal $n$-gon contains either no immersed horogon or a 1-parameter family of them (Lemma \ref{lem:whenhorogon}).

A horogon is {\bf embedded} if these horocycles are pairwise disjoint except for the points of consecutive tangency, while it is {\bf self-tangent} if it is not embedded but there are no transverse intersections between the horocycles.
Note that an immersed horogon need not lie entirely inside $P$, while embedded or self-tangent ones must.
See Figure \ref{fig:horogons}.

\begin{figure}[ht]
    \centering
    \includegraphics[width=0.8\linewidth]{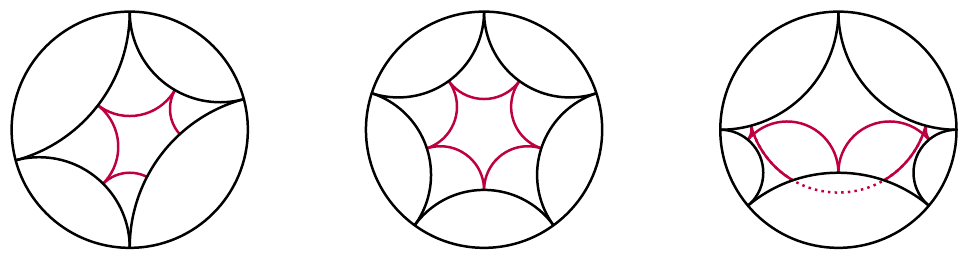}
    \caption{Left: a quadrilateral that admits no immersed horogon. Middle: an embedded (and regular) horogon in the regular ideal $5$-gon. Right: an immersed horogon in a nonregular ideal $5$--gon, with one side not entirely contained in the polygon.}
    \label{fig:horogons}
\end{figure}

The following is a concrete instance of the asymmetry inherent in optimal Lipschitz maps.

\begin{maintheorem}\label{mainthm:polygon}
Let $P$ be any ideal $n$-gon and let $P_{\text{reg}}$ denote the regular ideal $n$-gon.
\begin{enumerate}
    \item There is a boundary-tight $P \to P_{\text{reg}}$ if and only if $P$ has an embedded horogon.
    \item There is a boundary-tight $P_{\text{reg}} \to P$ if and only if $P$ has an immersed horogon.
\end{enumerate}
\end{maintheorem}

We do not compute the exact Lipschitz constants of these maps. However, by pre- or post-composing with boundary-tight self-maps of $P_{\text{reg}}$
(arising from stretching away from the embedded regular horogon as described above), 
one obtains a threshold  constant $L_0(P)$ such that for every $L \ge L_0$, there exists a boundary-tight map as in the Theorem with Lipschitz constant exactly $L$. Our constructions are very explicit, so $L_0(P)$ can in principle be estimated by tracking the parameters in the proofs.

\subsection{Spines and deflations}
For arbitrary surfaces $\Sigma$ with (possibly crowned) boundary, our results are more coarse, and must necessarily be so.
To state them precisely, we first need to clarify what we mean by rescaling the target surface.

Given any hyperbolic surface $Y$ with (possibly crowned) boundary, one can build a ribbon-graph spine $\Sp(Y)$ by taking the set of points in $Y$ which have non-unique closest point projections to $\partial Y$.
Equipping each edge of $\Sp(Y)$ with the length of (either of) its projection(s) to $\partial Y$ further specifies a metric on this ribbon graph.
This assignment induces a $\Mod(\Sigma)$-equivariant homeomorphism
\[\Sp: \cT(\Sigma) \to \TRG(\Sigma)\]
where $\TRG(\Sigma)$ is the space of marked metric ribbon graph structures on $\Sigma$ (see Section \ref{subsec:spinebasics} and Section \ref{subsec:RGs and arcs}, Theorem \ref{thm:spinecoords} in particular).

The closest-point projection map to $\partial Y$ determines a map $Y \to \Sp(Y)$ which
\begin{enumerate}
    \item is a 1-Lipschitz homotopy equivalence and
    \item restricts to an isometry on each geodesic of $\partial Y$.
\end{enumerate}
If $Y \in \cT(\Sigma)$ and $T \in \TRG(\Sigma)$, we call any map $Y \to T$ satisfying conditions (1) and (2) a {\bf deflation}.
One of the main new ideas of our work is to study these maps, and we believe the perspective we develop here will be useful in future work.
These share a similar flavor to the ``slack'' computations of \cite{FLM2}.

As already observed in Theorem \ref{mainthm:polygon}, the existence of a deflation $Y \to T$ does not imply that $\Sp(Y) = T$, or that they even have the same topological type.
Indeed, if $P$ contains an embedded or self-tangent horogon, then collapsing that horogon to a point and contracting the horocyclic foliation in each complementary spike yields a deflation $P \to \ast_n$, where $\ast_n$ is the ``$n$-pronged star'' ribbon graph resulting from gluing together $n$ copies of $\RR_{\ge 0}$ along $0$.
However, $\Sp(P) \neq \ast_n$ unless $P = P_{\text{reg}}$.

We therefore introduce a coarser notion of equivalence. Every metric ribbon graph $T$ is dual to a filling arc system $\arc(T)$ on $\Sigma$, and we can further assign to each arc $\alpha$ a weight $c_\alpha(T)$ recording the length of its dual edge in $T$.
Given $T_1, T_2 \in \TRG(\Sigma)$, say
\[T_1 = T_2 + O(1)\]
if there exists some constant $C$ such that:
\begin{itemize}
    \item For $\alpha \in \arc(T_1) \cap \arc(T_2),$ we have
$|c_\alpha(T_1) - c_\alpha(T_2)| \le C$.
    \item For $\alpha \in \arc(T_1) \Delta \arc(T_2),$ we have
    $c_\alpha(T_i) \le C$
($\Delta$ denotes the symmetric difference).
\end{itemize}
Equivalently, $T_1$ and $T_2$ are close in an appropriate bordification of $\TRG(\Sigma)$ (Section \ref{subsec:RGs and arcs}).

\begin{maintheorem}\label{mainthm:deflate}
Suppose that $Y \in \cT(\Sigma)$, $T \in \TRG(\Sigma)$, and $Y \to T$ is a deflation. Then,
\[\Sp(Y) = T + O(1)\]
where the implicit constant depends only on the systole and Euler characteristic of $Y$.
\end{maintheorem}

Deflations play well with boundary-tight maps (Lemma \ref{lem:composite_deflation}), so we conclude:

\begin{maincor}\label{cor:rescale}
Suppose that $Y \to Y'$ is an $L$-Lipschitz, boundary-tight map. Then,
\[\Sp(Y) = (1/L) \cdot \Sp(Y') + O(1).\]
\end{maincor}

Negative curvature implies that long edges of $\Sp(Y)$ correspond to bands of $Y$ where $\partial Y$ is very close to itself. With this viewpoint, Corollary \ref{cor:rescale} is stating that any boundary-tight map roughly stretches each one of these bands by its Lipschitz constant.

\begin{remark}\label{rmk:isolate}
If the set of arcs dual to long edges of $T$ does not fill $\Sigma$, then there are $T'$ with $T' = T + O(1)$ whose topological type varies wildly from that of $T$
(for example, apply a partial pseudo-Anosov supported on the subsurface of $\Sigma$ filled by short edges).
An analogous phenomenon also appears in the setting of boundary-tight maps: if $Y$ contains short enough essential curves, then the collars about those curves can be used to insulate a subsurface $U \subset Y$ from the affine stretch on $\partial Y$. One can then deform the geometry on $U$ at will so long as it is done at a slower rate than $\partial Y$ stretches.

Complementarily, if $Y$ does not contain short curves, then the arcs dual to long edges of $T$ must fill $\Sigma$ (Lemmas \ref{lem:big sys then fill} and \ref{lem:lengthcomparison T<Y}) and so the topological type of $T$ can vary only in a much more constrained range.
\end{remark}

Conversely, we prove that up to the constraint in Corollary \ref{cor:rescale}, there are boundary-tight maps (and tight maps) with essentially any prescribed behavior.

When $X$ is closed surface and $\lambda$ a lamination, it is not always the case that every topological type of ribbon graph appears as the spine of $X \setminus \lambda$; this is due to a generalization of the fact that if two simple closed curves are glued together they must have the same boundary length (see Theorem \ref{thm:imageofcut}).
Say that a marked ribbon graph $\Gamma$ is 
{\bf $(S \setminus \lambda)$-compatible} if there exists some $X \in \cT(S)$ such that $\Sp(X \setminus \lambda)$ is a metric structure on $\Gamma$.

\begin{maintheorem}\label{mainthm:itin}
Suppose that $\Gamma_0, \ldots, \Gamma_n$ are topological types of ribbon graph spines for $\Sigma$ and that $\Gamma_i$ is obtained from $\Gamma_{i+1}$ by collapsing edges.
Then there exist metrics $T_i \in \TRG(\Sigma)$ on $\Gamma_i$ with all long edges, hyperbolic structures $Y_i \in \cT(\Sigma)$ with $\Sp(Y_i) = T_i + O(1)$, and boundary-tight homeomorphisms
\[Y_0 \to Y_1 \to \cdots \to Y_n.\]
Moreover, if $\lambda$ is a measurable lamination on a closed surface $S$ and all $\Gamma_i$ are $(S \setminus \lambda)$-compatible,
then $T_i$ can be taken such that
there exist $X_i \in \cT(S)$ with $\Sp(X_i \setminus \lambda) = T_i + O(1)$ and tight homeomorphisms
\[X_0 \to X_1 \to \cdots \to X_n\]
whose tension laminations all contain $\lambda$.
\end{maintheorem}

Theorem \ref{mainthm:itin} is proved by gluing together (nearly-)regular ideal polygons and leveraging the maps from Theorem \ref{mainthm:polygon}.
While certain details are subtle, our construction is overall quite flexible and we actually prove a stronger statement which allows us to specify the metrics $T_i$ much more precisely. See Theorem \ref{thm:itineraries} and its proof.

\subsection{Walls, chambers, and envelopes}\label{subsec:envelopes}
The space of metric ribbon graphs $\TRG(\Sigma)$ carries a natural polyhedral structure. Each combinatorial type of spine $\Gamma$ for $\Sigma$ determines an orthant, with linear coordinates given by lengths of the edges.
If $\Gamma$ can be obtained from $\Gamma'$ by collapsing edges, then the orthant corresponding to $\Gamma$ appears in the boundary of the one corresponding to $\Gamma'$ by sending the lengths of the edges of $\Gamma' \setminus \Gamma$ to $0$.
When $\Sigma$ is the complement of a lamination $\lambda$ on $S$, the possible spines which occur must satisfy a gluing condition, which cuts out a subspace in each orthant of $\TRG(\Sigma)$.
When one orthant (or subspace) lies in the boundary of another, we refer to the boundary as a {\bf wall}, and to the interior as a {\bf chamber.}

For $X \in \T(S)$, the {\bf out-} (respectively {\bf in}-){\bf envelope} at $X$ with tension lamination $\lambda$ is 
\[ \Out(X,\lambda) = \{ X' \in \T(S) : \lambda = \lambda(X,X') \},
\qquad
\In(X,\lambda) = \{ X' \in \T(S) : \lambda = \lambda(X',X) \}. \] 
Equivalently, this is the union of all geodesics for the Lipschitz metric on $\cT(S)$ beginning (respectively, ending) at $X$ with tension lamination $\lambda$ (see Section \ref{subsec:Thmet_basics}).

Our theorems can also be interpreted as a description of the image of envelopes under
$\Sp_\lambda: \T(S) \to \TRG(S \setminus \lambda),$
the map that cuts along $\lambda$ and then takes the spine of the resulting surface.
Theorem \ref{mainthm:deflate} implies that any Thurston geodesic $X_t$ with tension lamination $\lambda$ must coarsely respect the polyhedral structure of $\cT(S \setminus \lambda)$.
That is, if $\Sp_\lambda(X_t)$ lies close to a wall, then it can wander within a uniform neighborhood of that wall.
However, as soon as $\Sp_\lambda(X_t)$ travels sufficiently far into the interior of an adjoining chamber, it must move deeper into that chamber as $t \to \infty$, never to return to the wall. 
See Theorem \ref{thm:lifecycle} for a more thorough discussion of this phenomenon.
Conversely, Theorem \ref{mainthm:itin} says that one can always move from a wall into an adjoining chamber.

In particular, if $X$ is sufficiently deep in a wall, then $\Sp_\lambda(\Out(X,\lambda))$ is (coarsely) contained in the union of the adjoining chambers, and (coarsely) contains an open cone in each. Similarly, if $X$ is sufficiently deep in a chamber, then $\Sp_\lambda(\In(X,\lambda))$ is (coarsely) contained in the closure of that chamber, and meets each of the adjoining walls.
See Figure \ref{fig:envelopes}.

\begin{figure}[ht!]
    \centering
    \includegraphics[width=0.6\linewidth]{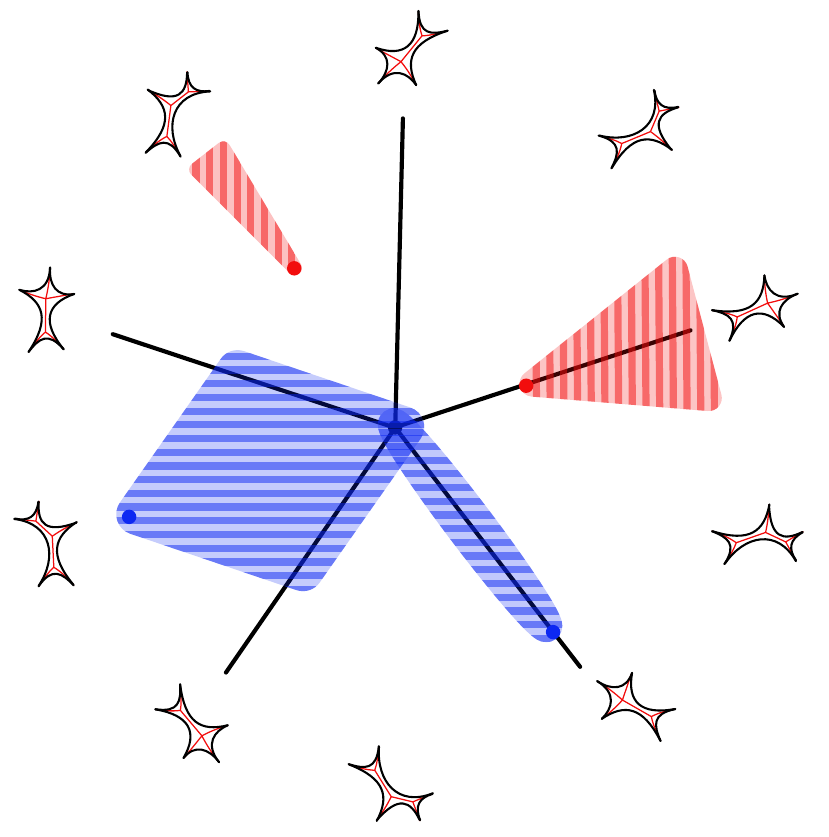}
    \caption{Envelopes in the Teichm{\"u}ller space of ideal pentagons. The unbounded, red, vertically-striped cones are the set of $P$ admitting a boundary-tight map from the tip of the cone (i.e., out-envelopes).
    The bounded, blue, horizontally-striped regions are those $P$ which map to the indicated points (i.e., in-envelopes).
    The center point is the regular pentagon; its out-envelope is the entire Teichm{\"u}ller space, while its in-envelope (not pictured) is a small open neighborhood of the center.
    }
    \label{fig:envelopes}
\end{figure}

\subsection{Context}
In a widely-circulated unpublished 1986 preprint \cite{Th_stretch}, Thurston proved that for any $X, X' \in \T(S)$, the Lipschitz constant of any tight map between them is the supremum that any curve is stretched:
\[ \inf_{f:X \to X'} \text{Lip}(f) = \sup_{\alpha} \frac {\ell_{X'}(\alpha)}{\ell_X(\alpha)}\]
where the supremum is taken over all closed curves and $\ell_X(\alpha)$ is the hyperbolic length of the geodesic representative of $\alpha$ on $X$.
By taking the log of these expressions, he defined an asymmetric metric on $\T(S)$ now called the {\bf Thurston metric}:
\[d(X,X') = \inf_{f:X \to X'} \log \text{Lip}(f).\]
The infinitesimal form of this description also yields an (asymmetric) Finsler structure on $\T(S)$ inducing the metric. The metric is geodesic, though typically not uniquely so, reflecting the non-uniqueness of tight maps.

After Thurston’s work, the geometry of the Thurston metric remained largely unexplored for over a decade, and much of the attention has focused on constructing new examples of tight maps \cite{PTmaps, Pap_Yam, HPtori} and analogies with the Teichm{\"u}ller metric.
Its asymptotic geometry was first analyzed by \cite{Pap:extension}, then later by \cite{walsh:HB}, who also identified the isometry group of $(\T(S),d)$ for most $S$. 
In \cite{Li:LengthSpectrums2003,ChoiRafi07}, it was shown that the Teichm{\"u}ller and Thurston metrics roughly agree on the thick part of $\T(S)$, but can diverge sharply in the thin part, and in \cite{Telp_Masur} it was shown that Masur's criterion (linking recurrence properties and unique ergodicity) fails for the Thurston metric.
The large-scale behavior of Thurston geodesics has been studied in detail in \cite{LRT12, LRT15, DLRT20, LMRT24}, and the Finsler structure has been elucidated in recent work \cite{Pan20,HOP21, BNOP25}.
We direct the reader to the surveys \cite{PapadopTheret:Teich_Thurston, PSsurvey}.


There has been a recent surge of interest in the theory, driven by a number of newly-discovered connections with other fields.
The structure of tight maps to circles is related to the horocycle flow on abelian covers of hyperbolic surfaces \cite{FLM2}, while maps which shrink or grow the lengths of all geodesics can be used to study Margulis spacetimes \cite{DGK} and self-joinings of hyperbolic manifolds \cite{KMO1}.
Thurston geodesics renormalize the earthquake flow on the bundle of measured laminations over moduli space, so are intricately linked to counting simple closed curves, and are often mapped to Teichm{\"u}ller geodesics by the bridge linking the two worlds \cite{MirzEQ, shshI, shshII}.
By looking at $p$-harmonic maps as $p \to \infty$, Daskalopoulos and Uhlenbeck laid an analytic framework for the theory of the Thurston metric, recovering many of Thurston's theorems and exhibiting a duality between tight maps and earthquakes \cite{DU1, DU2, DU3}.

We mention in particular recent progress of Pan and Wolf,
who use limits of harmonic maps with growing energy to build canonical ``harmonic stretch rays'' between any two $X, X' \in \T(S)$ \cite{PW_rays}.
As an application of this technique, they analyzed the topological structure of envelopes of Thurston geodesics and proved a number of regularity properties \cite{PW_envelopes}.
The geometry of harmonic stretch rays is governed by that of optimal Lipschitz maps to trees (\cite[Section 13]{PW_rays} and \cite[Section 6]{PW_envelopes}), so the current paper can be seen as complementary to their approach, bringing in new hyperbolic-geometric arguments to bear on these problems.

\subsection{Outline of paper}
In the introductory Section \ref{sec:background}, we review a number of basic results about hyperbolic geometry and geodesic laminations, and also serves to set notation.
This section also contains a discussion of some of the main objects of the paper, including the construction of the spine of a surface and its dual arc system.
Section \ref{sec:tightmaps} collects basic facts about tight maps between hyperbolic surfaces and establishes some elementary properties of boundary-tight maps.

With these preliminaries taken care of, in Section \ref{sec:obstructpoly} we introduce deflations and use them to prove the necessity part of Theorem \ref{mainthm:polygon}.
The purpose of doing this first is that our proof will contain many of the main ideas of our general estimates on deflations, but they are stronger for ideal polygons and the key features are more apparent.
We prove Theorem \ref{mainthm:deflate} in Section \ref{sec:generalobstruct} by coarsening these arguments and tracking error terms. This Section also contains Theorem \ref{thm:lifecycle}, which elaborates on our discussion from Section \ref{subsec:envelopes}.

Having established our results obstructing boundary-tight maps, we turn in Section \ref{sec:construct horogon} to constructing them.
Using a smooth realization of the orthogeodesic foliation by horocycles and hypercycles, we build the maps guaranteed in Theorem \ref{mainthm:polygon} by specifying them on certain geometric pieces coming from $\Sp(Y)$.
This is the most technical part of the paper, because we are explicitly constructing maps and proving uniform bounds on their Lipschitz constants.
Our arguments are quite flexible and we derive a number of other results which we believe are of independent interest (see in particular Theorem \ref{thm:scalespineup}).
In Section \ref{sec:itin}, we bootstrap up to arbitrary surfaces by gluing together (nearly-)regular polygons with large shears. The resulting surfaces look almost like a union of (nearly-)regular polygons, so we can glue the maps from Theorem \ref{mainthm:polygon} together to exhibit any desired behavior. This proves Theorem \ref{mainthm:itin}.

\subsection*{Acknowledgments}
The first author would like to express his gratitude to James Farre and Alex Nolte for a number of useful conversations about tight maps. 
We thank Alex Nolte and Huiping Pan for helpful comments on a draft of this paper.
AC was supported by NSF grants
DMS-2202703 and DMS-2506934. JT was supported by NSF grant DMS-2304920.

\tableofcontents

\section{Hyperbolic surfaces and geodesic laminations}\label{sec:background}

After recalling a few facts from basic hyperbolic geometry, in this section we discuss geodesic laminations and their complements (Section \ref{subsec:crownbasics}), orthogeodesic foliations and spines (Section \ref{subsec:spinebasics}), and the duality between ribbon graphs and arc systems (Section \ref{subsec:RGs and arcs}).

\subsection{Basic hyperbolic geometry}
Consider the Poincar\'e disk model of the hyperbolic plane $\mathbb{H}^2$. Let $C$ be a Euclidean circle that meets $\mathbb{H}^2$. The intersection $C \cap \mathbb{H}^2$ comes in four flavors. If $C \subset \mathbb{H}^2$, then $C$ is also a hyperbolic circle with appropriate center and radius. If $C$ is tangent to $\partial \mathbb{H}^2$, then $C$ is a horocycle. If $C$ meets $\partial \mathbb{H}^2$ in two points orthogonally, then $C \cap \mathbb{H}^2$ is a geodesic. If $C$ meets $\partial \mathbb{H}^2$ in two points but not orthogonally, then $C \cap \mathbb{H}^2$ is a \textbf{hypercycle}.
A hypercycle may equivalently be defined as the boundary of a half-regular neighborhood of the geodesic passing through the same endpoints.
If a hypercycle lies distance $d$ from the geodesic, then it has constant curvature $\tanh d \in (0,1)$. A horocycle has constant curvature 1, and a circle of hyperbolic radius $r$ has constant curvature $1 /\tanh r \in (1,\infty)$.

\subsection{Laminations, crowns, residues}\label{subsec:crownbasics}
Let $X \in \cT(S)$ be a marked, closed, hyperbolic surface of genus $g$. A {\bf geodesic lamination} $\lambda$ on $X$ is a closed union of pairwise disjoint complete geodesics.
We let $\GL(X)$ denote the space of all geodesic laminations on $X$, equipped with the Hausdorff topology on compact subsets. A lamination is said to be {\bf chain-recurrent} if it can be approximated by a sequence of multicurves; the (closed) subspace consisting of chain-recurrent laminations is denoted by $\GLcr(X)$.

Given any other $X' \in \cT(S)$, the marking maps yield a $(\pi_1(X), \pi_1(X'))$-equivariant identification of the ideal boundaries of $\widetilde{X}$ and $\widetilde{X}'$, and hence a natural identification of $\GL(X)$ and $\GL(X')$.
We will henceforth drop the dependence on the base surface and just write $\GL$ and $\GLcr$, leaving the topological type of the surface implicit.

If $X \in \cT(S)$ and $\lambda \in \GL$, the complement of $\lambda$ in $X$ is a (possibly disconnected) open hyperbolic surface.
We use $X \setminus \lambda$ to denote its metric completion.
Each component of $X \setminus \lambda$ is a finite-volume hyperbolic surface with (possibly) {\bf crowned boundary}; that is, its boundary is geodesic, but may consist of multiple asymptotic geodesics.
An ideal hyperbolic polygon is also considered a surface with a single crowned boundary.
We will reserve $S$ for the topological type of a closed surface and use $\Sigma$ to denote the topological type of a surface with possibly crowned geodesic boundary.

If $\Sigma$ is a crowned surface, then we say that an {\bf $n$-crown} $C$ is a chain of boundary geodesics $g_1, \ldots, g_n$ such that $g_i$ is asymptotic to $g_{i+1}$, with indices taken mod $n$.
When $n$ is even, an $n$-crown $C$ can be oriented (in one of two ways) by picking an arbitrary orientation on $g_1$, then propagating that orientation to the other $g_i$ such that asymptotic geodesics always point the same way.
The following definition will play an important role later in the paper (see Proposition \ref{prop:residues}).

\begin{definition}[Definition 2.9 of \cite{Gupta_wild}]
Let $S$ be a surface with crowned boundary, $C$ an $n$-crown of $S$ with $n$ even, and pick an orientation $\vec{C}$ of $C$.
Given $Y \in \cT(S)$, we define its {\bf metric residue} as follows.
Pick small horocycles based at each of the vertices of $C$ and let $\hat g_i$ denote the subsegment of $g_i$ running between the truncating horocycles. Then set 
\[\res_Y(\vec C) = \sum_{i=1}^n \varepsilon_i \ell(\hat g_i),\]
where $\varepsilon_i$ is $+1$ if $S$ lies on the left of $g_i$ and $-1$ otherwise.
\end{definition}

Denote the complementary components of $S \setminus \lambda$ (considered now as topological surfaces) by $\{\Sigma_j\}_{j=1}^m$.
Cutting $X\in \cT(S)$ along $\lambda$ induces a map 
\[cut_{\lambda}: \T(S) \to \prod_{j=1}^m \T(\Sigma_j)\]
and we let $\cT(S \setminus \lambda)$ denote the image of this map.\footnote{
In \cite{shshI}, the subspace $\cT(S \setminus \lambda)$ was denoted by $\Base$, and the product of Teichm{\"u}ller spaces of subsurfaces was denoted by $\cT(S \setminus \lambda)$. Our choice of notation in this paper aligns with \cite{spine}, in that $\cT(S \setminus \lambda)$ remembers that structures on the complementary components must be able to be glued together, and we have reserved $\Base$ for the corresponding subspace of the arc complex of $X \setminus \lambda$.}

When $\lambda$ supports a measure, we can give an exact description of $\cT(S \setminus \lambda)$ in terms of metric residues.
For each orientable component $\mu$ of $\lambda$, let $C_1, \ldots, C_{K(\mu)}$ denote the crowns of $S \setminus \lambda$ adjacent to $\mu$. Pick an arbitrary orientation on $\mu$ and use that to induce orientations on each $C_k$. Theorem 12.1 of \cite{shshI} then states the following:

\begin{theorem}\label{thm:imageofcut}
Suppose that $\lambda$ admits a measure of full support. Then the image of $cut_\lambda$ in $\prod_{j=1}^m \T(\Sigma_j)$ is the intersection of the vanishing loci
\[\sum_{k = 1}^{K(\mu)} \res_{Y_j}(\vec{ C_k}) = 0\]
as $\mu$ ranges over all orientable components $\mu$ of $\lambda$.
\end{theorem}

Note that vanishing does not depend on sign, hence our choice of orientation on $\mu$ was unimportant. The condition in the theorem is a generalization to arbitrary laminations of the fact that if one component of $\lambda$ is a simple closed curve then the corresponding boundaries of $X \setminus \lambda$ must have the same length.

\begin{remark}
Theorem \ref{thm:imageofcut} is false for general geodesic laminations.
The issue is that when $\lambda$ contains spiraling leaves, the proof of Theorem \ref{thm:imageofcut} gives a linear condition on the residues of the pieces of $S \setminus \lambda$ together with shears along spiraling leaves. See Figure \ref{fig:badex}, and compare also with the discussion of residues in Section \ref{subsec:glue_reg}.
\end{remark}

\begin{figure}[ht]
    \centering
    \includegraphics[width=0.5\linewidth]{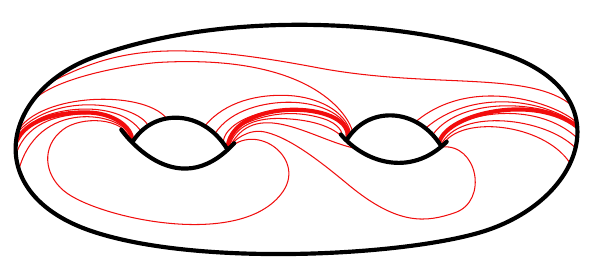}
    \caption{A lamination containing 3 simple closed curves and 4 spiraling leaves. Its complement is a union of two ideal 4-gons. There is no hyperbolic structure $X$ such that $X \setminus \lambda$ is a union of regular 4-gons, despite this satisfying the residue condition of Theorem \ref{thm:imageofcut}.    
    If one were to force the complementary components to be regular, the rightmost curve would have to have length 0 (compare Proposition \ref{prop:reg_glue_res}).}
    \label{fig:badex}
\end{figure}

\subsection{The spine of a surface}\label{subsec:spinebasics}
Let $\Sigma$ be a hyperbolic surface with boundary. 
To describe a hyperbolic structure on $\Sigma$, it is useful to have a system of coordinates that reflects the geometry of that structure.

Let $Y \in \cT(\Sigma)$. The {\bf orthogeodesic foliation} $\cO_{\partial Y}(Y)$ of $Y$ with respect to its boundary is the piecewise-smooth foliation whose leaves are the fibers of the closest-point projection map to $\partial Y$. It is equipped with an invariant transverse measure given by the pullback of the length element from $\partial Y$. 

The set of points which are closest to at least 2 points of $\partial Y$ form a piecewise-geodesic {\bf spine} $\Sp(Y)$ for $Y$ which inherits a ribbon structure (that is, a cyclic ordering of half-edges incident to each vertex) from the orientation of $Y$.
The transverse measure on $\cO_{\partial Y}(Y)$ equips this ribbon graph with a metric structure, in which the length of each edge of $\Sp(Y)$ is equal to the length of either of its projections to $\partial Y$. 
Equivalently, each vertex of an edge $e$ of $\Sp(Y)$ corresponds to singular leaves of $\cO_{\partial Y}(Y)$, and the length of $e$ is equal to the measure of any arc transverse to $\cO_{\partial Y}(Y)$ and connecting those leaves.
Note that this is {\em not} the same as the geodesic length of the edge.

Recall that any ribbon graph has a natural thickening to a surface with boundary.
A {\bf marking} of a ribbon graph $T$ is a choice of embedding $T \hookrightarrow \Sigma$ together with a deformation retract $\Sigma \to T$.
Let $\TRG(\Sigma)$ denote the space of marked metric ribbon graphs with the same topological type as $\Sigma$ whose vertices all have valence at least 3, with markings taken up to homotopy; note that the mapping class group of $\Sigma$ acts naturally by changing markings.
In the case where $\Sigma$ has crowned boundary, the ribbon graphs in $\TRG(\Sigma)$ must also have infinite rays emanating from vertices, corresponding to the spikes of $\Sigma$.
This space has a natural polyhedral structure determined by the combinatorics of the ribbon graph (see Section \ref{subsec:RGs and arcs} immediately below).

Performing the spine construction in families yields a map
\[\Sp: \cT(\Sigma) \to \TRG(\Sigma).\]
Generalizing a result of Luo for surfaces with compact boundary \cite{Luo} (as well as work of \cite{BowdEpst,Do,Ushijima}), Farre and the first author proved the following:

\begin{theorem}[Theorem 6.4 of \cite{shshI}]\label{thm:spinecoords}
If $\Sigma$ is a finite-volume surface with crowned boundary, then $\Sp$ is a $\Mod(\Sigma)$-equivariant (and stratified real-analytic) homeomorphism.
\end{theorem}

The metric residue conditions mentioned at the end of the previous subsection define linear equations on the spine coordinates.
Thus, if $S$ is a closed hyperbolic surface and $\lambda$ is a chain-recurrent lamination on $S$, then $\Sp$ identifies $\cT(S \setminus \lambda)$ with a polyhedral subcomplex of $\prod \TRG(\Sigma_j)$.
We note that not every topological type of ribbon graph on $S \setminus \lambda$ may occur in $\Sp(\cT(S\setminus \lambda))$,
because the combinatorics of the graph may not be compatible with the residue conditions (see, e.g., \cite[Figure 3]{twisttorus}).

\subsection{Ribbon graphs and arc systems}\label{subsec:RGs and arcs}
There is a natural duality between ribbon graphs of a given topological type and $\RR_{>0}$-weighted, filling arc systems on a surface.
This formalism will make our bookkeeping slightly nicer and will allow us to compare the geometry of different hyperbolic surfaces even when their spines do not have the same topological type.

Given a ribbon graph $\rg$ whose thickening has topological type $\Sigma$, one obtains a collection of arcs on $\Sigma$ by connecting points in $\partial \Sigma$ that come from thickening the same edge. These arcs break up into finitely many isotopy classes, and one can pick pairwise disjoint representatives of these classes to yield a(n isotopy class of) {\bf dual arc system} $\arc(\rg)$. By definition, the complement of $\arc(\rg)$ in $\Sigma$ is a union of polygons. If $\rg$ is also equipped with a metric, then one can weight each arc of $\arc(\rg)$ by the length of the edge dual to that arc.

When a ribbon graph arises as the spine of a hyperbolic surface $Y \in \cT(\Sigma)$, the dual arc system $\arc(Y) := \arc(\Sp(Y))$ can equivalently be obtained as the set of isotopy classes of leaves of the orthogeodesic foliation. Each isotopy class contains a unique geodesic representative orthogonal to $\partial Y$ at its endpoints, and we often use $\arc(Y)$ to refer to these representatives. It is important to note that these representatives may {\em not} be leaves of $\cO_{\partial Y}(Y)$, though they will be if they are short enough.

The discussion above implies that there is a natural piecewise-linear homeomorphism 
\[\TRG(\Sigma) \cong |\mathscr{A}_{\text{fill}}(\Sigma, \partial \Sigma)|_{\RR},\]
where the second term is the space of all positively weighted, filling arc systems on $\Sigma$ such that no arc is homotopic into a spike.\footnote{Observe that the correspondence we have described always results in an arc with weight $\infty$ truncating each spike, so it is not necessary to record those data.}
When $\Sigma$ is an ideal polygon the empty arc system is also filling.
See \cite[\S2.1]{Mond_handbook} or \cite[\S6]{shshI} for more detailed descriptions of this complex.

The space $|\mathscr{A}_{\text{fill}}(\Sigma, \partial \Sigma)|_{\RR}$ naturally sits inside the space of all weighted arc systems, and this viewpoint also allows one to take limits of weighted arc systems with weights going to $0$.
This allows us to circumvent many of the subtleties encountered when attempting to compactify the moduli space of ribbon graphs {\`a} la \cite{Kont_Witten}.
We caution the reader that the correct topology to take on the space of weighted arc systems is the metric, not CW, topology, and so some residual subtlety yet remains. See, e.g., \cite[\S8]{BowdEpst}.

\section{Optimal maps}\label{sec:tightmaps}

In this section, we collect statements about tight and boundary-tight maps (Section \ref{subsec:Thmet_basics}) and then prove some preliminary results using the structure of spikes of $X \setminus \lambda$ (Section \ref{subsec:preserve spikes}).

\subsection{Thurston metric}\label{subsec:Thmet_basics}
Given $X,X' \in \cT(S)$, a \textbf{change-of-marking} is a map $f: X \to X'$ in the homotopy class of $ \rho_{X'} \circ \rho_X^{-1}$, where $\rho_X : S \to X$ is the marking on $X$. Following Thurston \cite{Th_stretch}, we define the distance from $X$ to $X'$ to be
\[ d(X,X') = \inf_f \log \text{Lip}(f),\]
where $f: X \to X'$ is a change-of-marking map.
A change-of-marking map $f: X \to X'$ is \textbf{tight} if $\log \text{Lip}(f) =d (X,X')$.

Thurston showed that $d$ is an asymmetric metric on $\cT(S)$, meaning it satisfies the reflexive property and the triangle inequality but $d(X,X')$ may not coincide with $d(X',X)$.
Let us recall some results which we have already mentioned in the Introduction.

\begin{theorem}[\cite{Th_stretch}]
Given $X,X' \in \cT(S)$, the following statements hold.
\begin{enumerate} 
\item For simple closed curve $\alpha$ on $S$, denote by $\ell_X(\alpha)$ the length of the geodesic representative of $\alpha$ on $X$. Then
\[ d(X,X') =  \sup_\alpha \log \frac{\ell_{X'}(\alpha)}{\ell_X(\alpha)}.\]
\item There exists a tight homeomorphism $f: X \to X'$.
\item There exists a chain-recurrent lamination $\lambda=\lambda(X,X')$ such that every optimal map $f: X \to X'$ maps leaves of $\lambda$ affinely to leaves of $\lambda$ with slope Lip$(f)$, and $\lambda$ is maximal with respect to this property.
\item There exists a geodesic $\{X_t\}$ from $X$ to $X'$ such that for all $t$, $\lambda(X,X') \subseteq \lambda(X,X_t) \cap \lambda(X_t,X')$.  
\end{enumerate}
\end{theorem}

We call $\lambda(X,X')$ the \textbf{tension lamination} from $X$ to $X'$. A bi-infinite geodesic or a geodesic ray $\{X_t\}$ likewise has a tension lamination $\lambda$, which is a chain-recurrent lamination with the property that for all $s < t$, $\lambda \subseteq \lambda(X_s,X_t)$, and such that $\lambda$ is maximal with respect to this property.

Given a tight map $X \to X'$, cutting along the tension lamination $\lambda(X, X')$ yields a Lipschitz map between the complements that sends boundaries to boundaries affinely.
We abstract this in the following:

\begin{definition}
Let $Y, Y' \in \cT(\Sigma)$ be hyperbolic surfaces with (possibly crowned) boundary. A map $f: Y \to Y'$ is said to be {\bf boundary-tight} if it is $L$-Lipschitz, takes $\partial Y$ onto $\partial Y'$, and is affine with stretch factor $L$ when restricted to each geodesic of $\partial Y$.
\end{definition}

Conversely, we can always glue together boundary-tight maps.
In particular, in order to describe a tight map $f:X \to X'$ with a given tension lamination $\lambda$, it is equivalent to understand the induced maps on the complementary subsurface $X \setminus \lambda$.
The following is a consequence of \cite[Proposition 4.1]{Th_stretch}.
\footnote{As pointed to us by Huiping Pan, to apply that Proposition one also needs to know that boundary-tight maps preserve the (germ of the) partial horocyclic foliation in a neighborhood of $\lambda$, which is the content of Lemma \ref{lem:preserve horocycle endpts} below.}

\begin{theorem}[Gluing boundary-tight maps]\label{thm:glueoptimal}
Suppose that $\lambda$ is a chain-recurrent geodesic lamination on a hyperbolic surface $X \in \cT(S)$ and let $Y_1, \ldots, Y_m$ denote the components of (the metric completion of) $X \setminus \lambda$.
Suppose that $L > 1$ and for each $j = 1, \ldots, m$, we are given a $Y_j' \in \cT(\Sigma_j)$ and an $L$-Lipschitz, boundary-tight map
\[f_j: Y_j \to Y_j'.\]
Then there is an $X' \in \cT(S)$ and an $L$-Lipschitz map
\[f: X \to X'\]
such that $X' \setminus \lambda = \bigcup Y_j'$ and $f|_{Y_j} = f_j$. In particular, $f |_\lambda$ is affine.
\end{theorem}


\subsection{Spikes and horocycles}\label{subsec:preserve spikes}
Let us demonstrate one immediate consequence of our shift in perspective from closed surfaces to complementary subsurfaces.
When $\lambda$ is not a geodesic multicurve, it contains asymptotic geodesics, so the components of $S \setminus \lambda$ have spikes.
These spikes greatly constrain the structure of tight maps with tension lamination $\lambda$.

Let $Y$ denote a simply-connected hyperbolic surface with geodesic boundary (the reader should think of $Y$ either as an ideal polygon or the universal cover of a finite-area surface).
Identifying $Y$ with a subset of $\HH^2$, a {\bf horocycle} in $Y$ is the intersection of $Y$ with a horocycle based at some point of its ideal boundary. Note that with this definition, horocycles in $Y$ may not necessarily be connected.
Horocycles can also be intrinsically defined as the level sets of Busemann functions based at the spikes.

\begin{lemma}[Horocycle endpoints preserved]\label{lem:preserve horocycle endpts}
Suppose that $f:Y \to Y'$ is a boundary-tight map between simply-connected hyperbolic surfaces with geodesic boundary and that $g_1, g_2 \subset \partial Y$ are asymptotic geodesics.
Then for any horocycle $H$ based at the ideal endpoint of $g_1$ and $g_2$, the points $f(g_1 \cap H)$ and $f(g_2 \cap H)$ must also lie on a horocycle based at the common endpoint of $f(g_1)$ and $f(g_2)$.
\end{lemma}

\begin{proof}
Throughout the proof, let us set
\[p_1 = g_1 \cap H
\text{ and }
p_2 = g_2 \cap H.\]
We will decorate with primes for the corresponding objects on $Y'$, so that $p_1' := f(g_1 \cap H)$ and so on.
Let $L$ to be the Lipschitz constant of $f$.
Consider the horocycle $H'$ on $Y'$ through $p_1'$ and based at the common endpoint of $g_1'$ and $g_2'$, and set $q' := H' \cap g_2'$. See Figure \ref{fig:spike}.

\begin{figure}[ht]
    \centering
    \includegraphics[width=0.8\linewidth]{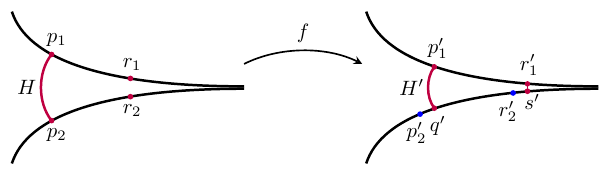}
    \caption{Endpoints of horocycles as in Lemma \ref{lem:preserve horocycle endpts}.}
    \label{fig:spike}
\end{figure}

Fix $\varepsilon > 0$ and choose points $r_i \in g_i$ deep in the spike formed by the $g_i$ such that 
\[d(r_1, r_2) < \varepsilon/L
\text{ and }
d(p_1, r_1) = d(p_2, r_2).\]
Now by the Lipschitz property of $f$, we know that 
\[d(r_1', r_2') < \varepsilon\]
and in particular the horocycle through $r_1'$ meets $g_2'$ at a point $s'$ at distance at most $\varepsilon$ from $r_2'$ (this follows from the fact that projection along the leaves of the horocyclic foliation is 1-Lipschitz).

Now by boundary-tightness, $d(p_1', r_1') = d(p_2', r_2')$,
and $d(p_1', r_1') = d(q', s')$ by construction, so
\[d(p_2', q') = d(s', r_2') < \varepsilon.\]
Since $\varepsilon$ was arbitrary, we conclude that $p_2' = q'$, which is what we wanted to show.
\end{proof}

In particular, this gives our first restriction on the structure of boundary-tight maps.

\begin{proposition}\label{prop:residues}
Let $\vec C$ be an oriented $n$-crown of $\Sigma$.
Then for any $Y, Y' \in \cT(\Sigma)$ and any $L$-Lipschitz, boundary-tight map $f:Y \to Y'$,
\[L \cdot \res_Y(\vec C) = \res_{Y'}(\vec C).\]
\end{proposition}
\begin{proof}
Recall that the metric residue of $\vec C$ (in either $Y$ or $Y'$) is computed by fixing horocycles in the spikes and then taking a signed combination of lengths of segments $\hat g_i$ running between the endpoints of the truncating horocycles.
By Lemma \ref{lem:preserve horocycle endpts}, the endpoints of horocycles are sent to the endpoints of horocycles under $f$, so the metric residue of $\vec C$ in $Y'$ can be computed using the segments $f(\hat g_i)$ of $\partial Y'$.
Now $f$ stretches the boundary of $Y$ by an affine factor of $L$, so the length of $f(\hat g_i)$ is $L$ times the length of $\hat g_i$.
\end{proof}

Using a very similar argument to Lemma \ref{lem:preserve horocycle endpts}, we can also show that $f$ maps every point of a horocycle at least as deep into the spike as $f$ maps its endpoints.

\begin{lemma}[Spikes into spikes]\label{lem:cusps into cusps}
Let all notation be as in Lemma \ref{lem:preserve horocycle endpts}.
Then for every point $q$ on the horocycle $H$, $f(q)$ lies within the (closed) region bounded by $\partial Y'$ and the horocycle through $p_1'$ and $p_2'$.
\end{lemma}
\begin{proof}
As before, take some point $r_1$ deep in the spike such that the horocycle through $r_1$ has length at most $\varepsilon/L$. Let $s$ be the unique point of that horocycle such that 
\[d(q,s) = d(p_1, r_1).\]
Set $q' = f(q)$ and let $q''$ denote the intersection of $g_1'$ with the horocycle through $q'$.

By the fact that projection along horocycles is 1-Lipschitz, the triangle inequality, and the fact that $f$ is $L$-Lipschitz,
\[d(q'',r_1') \le d(q', r_1') \le d(q', s') + \varepsilon\]
Using the Lipschitz property again together with boundary tightness, we get
\[d(q', s') \le L d(q,s) 
 = L d(p_1, r_1) = d(p_1', r_1').\]
Putting these together, we see that $q'$ lies at most $\varepsilon$ outside the horocycle through $p_1'$. Sending $\varepsilon$ to $0$ completes the proof.
\end{proof}

\section{Deflations and regular polygons}\label{sec:obstructpoly}
In this section we prove one direction of Theorem \ref{mainthm:polygon}.
Our proof rests on the theory of deflations, which are introduced in Section \ref{subsec:deflations}.
After established basic properties of these maps, we show in Section \ref{subsec:regpoly_deflate} how the necessity part of Theorem \ref{mainthm:polygon} follows from a statement about deflations (Proposition \ref{prop:deflate to star}).
We deduce this Proposition in Section \ref{subsec:obstruct star} as a consequence of broader restrictions on deflations, which will be further generalized in the following Section \ref{sec:generalobstruct}.

\subsection{Deflations}\label{subsec:deflations}
Given a hyperbolic surface $Y$ with (crowned) boundary, there is a natural 1-Lipschitz map $Y \to \Sp(Y)$ obtained by collapsing the leaves of the orthogeodesic foliation. By definition of the metric along $\Sp(Y)$, this map restricts to an isometry along each geodesic of $\partial Y$.
Similarly, if one takes an ideal triangulation of $Y$ and considers the horocyclic foliation with respect to that triangulation, then collapsing the leaves results in another map with the same properties (1-Lipschitz and isometric along boundary geodesics).

We abstract these phenomena in the following definition:

\begin{definition}\label{def:deflation}
Let $\Sigma$ be a surface with (possibly crowned) boundary, and let $Y \in \cT(\Sigma)$ and $T \in \TRG(\Sigma)$.
A {\bf deflation} $\dfl:Y \to T$ is a homotopy equivalence such that:
\begin{enumerate}
    \item $\dfl$ is 1-Lipschitz. 
    \item $\dfl$ restricts to an isometric embedding on each geodesic of $\partial Y$.
    \item If $\varphi: \Sigma \to Y$ and $r: \Sigma \to T$ denote the markings, then $\dfl \circ \varphi$ is homotopic to $r$.
\end{enumerate}
\end{definition}

The key property of deflations is that they play well with boundary-tight maps. The proof of the following lemma is immediate by unpacking definitions.

\begin{lemma}\label{lem:composite_deflation}
Suppose that $Y, Y' \in \cT(\Sigma)$ and that $f:Y \to Y'$ is a boundary-tight map with Lipschitz constant $L$.
Suppose also that $T \in \TRG(\Sigma)$ and $\dfl:Y' \to T$ is a deflation.
Then the composition of $\dfl \circ f$ and the rescaling action yields a deflation
$Y \to 1/L \cdot T.$
\end{lemma}

In particular, obstructions to the existence of deflations $Y \to 1/L \cdot \Sp(Y')$ also obstruct the existence of boundary-tight maps $Y \to Y'$ (compare Proposition \ref{prop:deflate to star} below).
\medskip

In our motivating examples of deflations, it was easy to understand how the surface $Y$ deflated to the graph $T$. In general, the preimages of points of $T$ may have interior and can even be disconnected --- consider a small perturbation of the map that collapses the horocyclic foliation of an ideal triangle.
However, just from first principles we can prove that the map $f:\partial Y \to T$ is generically 2-to-1.

\begin{lemma}\label{lem:deflation_doubleback}
Let $\dfl:Y \to T$ be a deflation. Then for any open edge $e$ of $T$, the preimage $\dfl^{-1}(e) \cap \partial Y$ consists of two connected intervals.
\end{lemma}
\begin{proof}
Lift everything to universal covers, so that we have an equivariant map $\tilde f$ from a simply-connected hyperbolic surface $\widetilde Y$ to a metric ribbon tree $\widetilde T$.
Fix one lift $\tilde e$ of the (open) edge $e$.
Since we are in a tree, every geodesic in $\widetilde T$ that meets $\tilde e$ must cross it, and there are exactly two such geodesics which are also compatible with the ribbon structure.
Since $\tilde f$ is an isometry on the boundary geodesics of $\partial \widetilde Y$, we see that exactly two of them map to $\tilde e$, and the preimage of $\tilde e$ on each is a connected (open) interval.
\end{proof}

\subsection{Maps to regular polygons}\label{subsec:regpoly_deflate}
For any $n \ge 3$, recall that $P_{\text{reg}}$ denotes the regular ideal $n$-gon, i.e., the unique ideal $n$-gon with a dihedral symmetry group of order $2n$.
Given an ideal $n$-gon $P$, recall from the \hyperref[sec:intro]{Introduction} that an (immersed) {\bf horogon} is a choice of horocycles, one based at each vertex of $P$, such that the two horocycles based at the endpoints of each edge of $P$ are tangent.

\begin{lemma}\label{lem:whenhorogon}
An ideal $n$-gon admits a unique horogon if $n$ is odd, a 1-parameter family if $n$ is even and $P$ has zero residue, and no horogon if $n$ is even and $P$ has nonzero residue.
\end{lemma}
\begin{proof}
Pick an arbitrary point $x_1 \in \partial P$. Choose one of the horocycles $h_1$ through $x_1$ and based at one of the endpoints of the edge $g_1$ of $P$ containing $x_1$.
Let $g_2$ denote the other geodesic of $\partial P$ with endpoint the basepoint of $h_1$, and set $x_2 = g_2 \cap h_1$.
Let $h_2$ denote the horocycle through $x_2$ based at the other endpoint of $g_2$, and continue on in this manner, defining points $x_i$ on consecutive sides of $P$ until returning back to $g_1$; let $x_{n+1} = g_1 \cap h_n$.

If $x_{n+1} = x_1$, then the union of the $h_i$ a horogon with vertices $x_1, \ldots, x_n$.
If $n$ is odd then moving $x_1$ along $g_1$ moves $x_{n+1}$ in the opposite direction, so there is a unique point at which they meet.
If $n$ is even, then we use the $h_i$ in our computation of the metric residue of $P$, deducing that the distance between $x_1$ and $x_{n+1}$ equals $|\res(P)|$. In the case where $\res(P)=0$, moving $x_1$ along $g_1$ does not change the fact that $x_1 = x_{n+1}$, so any placement of our initial point defines a horogon for $P$. 
\end{proof}

Recall that we say a horogon is {\bf embedded} if horocycles are all pairwise disjoint except for the specified tangencies, and is {\bf self-tangent} if any pair of horocycles intersects at most once but it is not embedded.

Let $\ast_n$ denote the ``$n$-valent star'' graph, obtained by gluing together $n$ copies of $\RR_{\ge 0}$ along $0$.
Let $P$ be any ideal $n$-gon containing either an embedded or self-tangent horogon.
Collapsing along the leaves of the horocyclic foliation in each spike and crushing the horogon to a point yields a deflation $P \to \ast_n$.
We prove that such $P$ are the {\em only} polygons that deflate to $\ast_n$.

\begin{proposition}\label{prop:deflate to star}
Let $P$ be an ideal $n$-gon. There exists a deflation from $P$ to $\ast_n$ if and only $P$ contains either an embedded or self-tangent horogon.
\end{proposition}

Assuming the Proposition (which we prove in the subsection immediately below), let us prove one direction of Theorem \ref{mainthm:polygon}.

\begin{proof}[Proof of necessity in Theorem \ref{mainthm:polygon}]
Suppose there is a boundary-tight map $f:P \to P_{\text{reg}}$. Now $P_{\text{reg}}$ deflates to the $n$-valent star $\ast_n$ by collapsing the regular horocyclic foliation, so by Lemma \ref{lem:composite_deflation}, there exists a deflation $P \to \ast_n$ (note that $\ast_n$ is fixed by rescaling the metric since each leaf has infinite length).
Proposition \ref{prop:deflate to star} then guarantees that $P$ must contain either an embedded or self-tangent horogon.

Suppose that $P$ does not contain an embedded horogon. Take the self-tangent horogon inside $P$ and choose $p \in P$ to be one of the self-tangencies; then $p$ lies in two non-adjacent spikes of $P$.
Now Lemma \ref{lem:cusps into cusps} ensures that $f(p)$ must lie inside the closures of each of the corresponding spikes of $P_{\text{reg}}$ complementary to the regular horogon, but these regions do not intersect. This is a contradiction.
\end{proof}

\subsection{Obstructing deflations}\label{subsec:obstruct star}
As we have already seen in Section \ref{subsec:preserve spikes}, spikes place strong constraints on the structure of boundary-tight maps. It turns out similar constraints are true for deflations.
The following is analogous to Lemma \ref{lem:cusps into cusps}.

\begin{lemma}[Spikes into infinite edges]\label{lem:deflations to leaves}
Suppose that $\dfl:Y \to T$ is a deflation. Let $e$ be any (open) half-infinite edge of $T$.
Let $p \in \partial Y$ be such that $\dfl(p) \in e$.
Then for any point $q$ on the horocycle through $p$ based at the spike of $Y$ corresponding to $e$,
\[\dfl(q) \ge \dfl(p),\]
where we have identified $e$ with $(0, \infty)$.
\end{lemma}

\begin{proof}
Fix $\varepsilon > 0$ and pick $r$ living on the same geodesic as $p$ but much deeper in the spike, such that the horocyclic segment through $r$ has length $\varepsilon$. Let $s$ on that horocycle be such that $d(r,p) = d(s,q).$
Then using the properties of deflations,
\[\dfl(r)-\dfl(p) 
= d(r,p) = d(s,q) 
\ge \dfl(s) -\dfl(q),\]
and so
\[\dfl(q) \ge \dfl(p) + \left(\dfl(s) - \dfl(r)\right)
\ge \dfl(p)-\varepsilon.\]
Sending $\varepsilon \to 0$ completes the proof.
\end{proof}

This immediately yields the following useful corollary:

\begin{corollary}\label{cor:preimages of leaves}
Let all notation be as in Lemma \ref{lem:deflations to leaves}. Then for any $t >0$, the points of $\dfl^{-1}(t) \cap \partial Y$ are connected by a horocyclic segment $H_t$ that is completely contained in $Y$.
\end{corollary}

Note that Lemma \ref{lem:deflation_doubleback} ensures that $\dfl^{-1}(t) \cap \partial Y$ is a pair of points.

\begin{proof}
To see the first statement, pick $p \in \dfl^{-1}(t) \cap \partial Y$ and let $q$ be the other point of intersection of the horocycle through $p$ with $\partial Y$; then Lemma \ref{lem:deflations to leaves} ensures $\dfl(q) \ge \dfl(p)$.
Applying the Lemma with the roles reversed gives equality.
The second statement is a consequence of the fact that deflations are isometries on $\partial Y$ (else, there would be two distinct points of the same boundary geodesic mapping to the same point of $\ell$).
\end{proof}

\begin{remark}\label{rmk:limit_H0}
The limiting horocycle $H_0 = \lim_{t \to 0} H_t$ connects points of $\dfl^{-1}(0)$ and is contained in $Y$, but may a priori be tangent to $\partial Y$.
Note that $\dfl^{-1}(0) \cap \partial Y$ may contain points {\em not} on $H_0$ (consider the deflation $P_{\text{reg}} \to \ast_n$ obtained by collapsing the orthogeodesic foliation).
\end{remark}

\begin{remark}
The above statements share more than just a similarity with the arguments in Section \ref{subsec:preserve spikes}.
Indeed, Corollary \ref{cor:preimages of leaves} can be used to give another, morally equivalent, proof of Lemma \ref{lem:preserve horocycle endpts}. 
Given a boundary-tight map $Y \to Y'$, deflating $Y'$ along its orthogeodesic foliation and rescaling yields a deflation of $Y$ to a ribbon graph with a half-infinite edge.
Preimages in $\partial Y$ of points in that edge must be connected by horocyclic segments, and since the orthogeodesic foliation is equivalent to the horocyclic foliation in spikes, the corresponding points in $\partial Y'$ must also lie on a horocycle.
\end{remark}

Given a deflation $\dfl: Y \to T$ as in Lemma \ref{lem:deflations to leaves} and a half-infinite edge $e$ of $T$, let $V_{e}$ denote the (open) spike of $Y$ foliated by the horocycles $H_t$ guaranteed by Corollary \ref{cor:preimages of leaves}:
\[V_{e}:=
\bigcup_{t > 0} H_t.\]
We observe the following, which is strictly stronger than Proposition \ref{prop:deflate to star}.

\begin{corollary}\label{cor:2 leaves 2 spikes}
If $e_1$ and $e_2$ are distinct half-infinite edges, then $V_{e_1}$ and $V_{e_2}$ are disjoint.
\end{corollary}
\begin{proof}
By Lemma \ref{lem:deflations to leaves},
$\dfl(V_{e}) = e$, and the open half-infinite edges $e_1$ and $e_2$ are disjoint.
\end{proof}

\begin{proof}[Proof of Proposition \ref{prop:deflate to star}]
Suppose that $P$ is an ideal $n$-gon and that $\dfl:P \to \ast_n$ is a deflation.
Let $0 \in \ast_n$ denote the unique vertex.
By Corollary \ref{cor:preimages of leaves} and Remark \ref{rmk:limit_H0}, for each pair of asymptotic geodesics $g_1, g_2$ of $P$, the points
\[\dfl^{-1}(0) \cap g_1 \text{ and } \dfl^{-1}(0) \cap g_2\]
lie on a common horocycle.
By Corollary \ref{cor:2 leaves 2 spikes}, these horocycles bound spikes with disjoint interiors and so intersect at most tangentially.
Thus, the points of $\dfl^{-1}(0) \cap \partial P$ are the vertices of a horogon with at worst self-tangencies.
\end{proof}

\section{Spines and deflations}\label{sec:generalobstruct}
In the previous sections, we saw that spikes place strong restrictions on the structure of both boundary-tight maps and deflations.
In this one, we use similar arguments to constrain the structure of deflations for arbitrary hyperbolic surfaces $Y$ with (crowned) boundary.
We can still apply the lemmas from Section \ref{sec:obstructpoly} to any spikes of $Y$, but for those geodesics of $\partial Y$ that are not asymptotic, our estimates must necessarily be coarser.

In Section \ref{subsec:edgesource}, we show that if $Y \to T$ is a deflation, then long edges of $T$ must all come from long edges of $\Sp(Y)$. This follows from basic properties of deflations.
For the converse, that long edges of $\Sp(Y)$ must always persist under deflation, we generalize our arguments from the previous section. This requires more detailed hyperbolic-geometric estimates and is accomplished in Section \ref{subsec:bands_deflate}.
Together, these allow us to prove Theorem \ref{mainthm:deflate} in Section \ref{subsec:generalspinedeflate}.
These results have consequences for the structure of boundary-tight maps $Y \to Y'$ and hence the geometry of Thurston geodesics; these corollaries are collected in Section \ref{subsec:deflconseq_geos}.

Throughout this section, we fix a surface with (possibly crowned) boundary $\Sigma$ together with $Y \in \cT(\Sigma)$, $T \in \TRG(\Sigma)$, and a deflation $\dfl:Y \to T$.
We will use $\arc(Y)$ and $\arc(T)$ to denote the (unweighted) arc systems dual to $\Sp(Y)$ and $T$, respectively.
For each arc $\alpha$ in either arc system we will use $c_\alpha(Y)$ and $c_\alpha(T)$ to denote its weight, equivalently, the length of the corresponding edge in $\Sp(Y)$ or $T$.

\subsection{Long edges come from thin bands}\label{subsec:edgesource}
First, let us show that if $T$ has a sufficiently long edge, then $\Sp(Y)$ must have a corresponding edge of at least roughly that length.

We begin by defining a useful geometry quantity. Set 
\[\rad(Y):= \sup_{y \in Y}d(y, \partial Y). \]
Note $\rad(Y)$ is also equal to the radius of the largest circle embedded in the universal cover of $Y$, equivalently, it is half the length of the longest (singular) leaf of $\cO_{\partial Y}(Y)$.
This quantity is not uniformly bounded over all of $\cT(\Sigma)$ \cite[Figure 1]{shshII}, but is over the thick part. The following Lemma will be useful in allowing us to ensure that the rest of our arguments in this section have uniform constants so long as $Y$ is chosen to be thick.

\begin{lemma}\label{lem:rad unif bdd}
For every $s >0$ there is an $R$ depending only on $\chi(\Sigma)$ and $s$ such that for every $Y \in \cT(\Sigma)$ with $\sys(Y) \ge s$, we have 
$\rad(Y) \le R$.
\end{lemma}
\begin{proof}
This statement is clear when $\Sigma$ only has closed boundary components, as the (universal curve over) the $s$-thick part of moduli space is compact.
It is also immediate when $\Sigma$ is an ideal $n$-gon, as the diameter of the circle inscribed in an ideal triangle is $\log 3$, so $d(y, \partial Y)$ is bounded by $n \log 3$.

It therefore suffices to bound, for any hyperbolic crown fitting a head with circumference at least $s$, the distance between the closed boundary geodesic $\gamma$ and the non-closed ones.
This can be done by observing that this distance is at most a sum of geometric series (one for each spike spiraling around $\gamma$) and an appropriate multiple of $\log 3$ (traveling through non-spiraling spikes).
But now each individual geometric series has decay rate $\approx e^{-s}$, thus so as long as $s$ is bounded below we get an upper bound for the sum.
\end{proof}

We have already noted in Remark \ref{rmk:isolate} that having very short systole allows for great flexibility in both deflations and boundary-tight maps. 
The previous Lemma implies that if $Y$ has sufficiently large systole, then the arcs with large weight must fill $Y$, which (combined with Theorem \ref{mainthm:deflate}) restricts deflations and boundary-tight maps from $Y$.

\begin{lemma}[Large arcs fill thick surfaces]\label{lem:big sys then fill}
For every $C > 0$, there is some $s_0$ such that for every $Y$ with $\sys(Y) \ge s_0$, the sub-arc system $\{\alpha \mid c_\alpha(Y) \ge C\}$ fills $Y$.
\end{lemma}
\begin{proof}
By definition, the constant $R(s)$ from Lemma \ref{lem:rad unif bdd} does not increase as $s$ increases.

Suppose that $Y$ has systole $s$. Lemma \ref{lem:rad unif bdd} states that $\rad(Y) \ge R(s)$, and since closest-point projection distorts distances at worst exponentially, the systole of $\Sp(Y)$ is at least $e^{-R}s$.
There is a uniform upper bound $\chi$ on the cardinality of $\arc(Y)$ depending only on the topology of $Y$, so every loop in $\Sp(Y)$ must cross an edge with length at least
$e^{-R}s/ \chi$, which is in particular larger than $C$ for $s$ sufficiently large.
\end{proof}

The number $\rad(Y)$ in a sense quantifies the difference of $Y$ from being a tree; as such, it is natural that it appears as an error term when making estimates on deflations.

\begin{lemma}[Long edges come from bands]\label{lem:long edges come from bands}
Let $\dfl:Y \to T$ be a deflation. 
If there is an arc $\alpha \in \arc(T)$ with 
$c_\alpha(T) > 2 \rad(Y),$
then $\alpha \in \arc(Y)$.
\end{lemma}

\begin{proof}
We begin by lifting everything to the universal cover and choosing some lift of $\alpha$ (which we will call by the same name).
Let $g_1$ and $g_2$ be the geodesics of $\partial \smash{\widetilde Y}$ connected by $\alpha$. 
Suppose towards contradiction the statement is not true; then either $\arc(Y)$ contains an arc $\beta$ crossing $\alpha$ or $\alpha$ is realized inside a singular leaf of $\cO_{\partial Y}(Y)$, which must have prongs emanating from both sides.
In either case, we can pick geodesics $g_3, g_4 \subset \partial \smash{\widetilde Y}$ which are on different sides of $\alpha$ and which are joined by a (possibly singular) leaf of $\cO_{\partial Y}(Y)$.
In particular, $d(g_3, g_4) \le 2 \rad(Y)$.

Since $g_3$ and $g_4$ live on different sides of $\alpha$, their images under $\dfl$ must live in different components of $\smash{\widetilde{T}} \setminus e$, where $e$ is the edge dual to $\alpha$.
In particular, since $\smash{\widetilde T}$ is a tree, this implies
\[c_\alpha(T) \le d_T(\dfl(g_3), \dfl(g_4)).\]
On the other hand, deflations are 1-Lipschitz, so putting everything together we get
\[c_\alpha(T) \le d_T(\dfl(g_3), \dfl(g_4))
\le d(g_3, g_4) \le 2 \rad(Y).\]
Thus so long as $c_\alpha(T) >2\rad(Y)$ we get a contradiction, ensuring $\alpha \in \arc(Y)$.
\end{proof}

Slightly tweaking the above proof gives an estimate comparing the weights of arcs.

\begin{lemma}\label{lem:lengthcomparison T<Y}
Let $\dfl:Y \to T$ be a deflation and suppose that $\alpha \in \arc(Y) \cap \arc(T)$.
Then
\[c_\alpha(T) \le c_\alpha(Y) + 4 \rad(Y).\]
\end{lemma}
\begin{proof}
As in the previous Lemma, lift to the universal cover and let $g_1$ and $g_2$ denote the geodesics of $\partial Y$ that $\alpha$ connects.
Let $p,q \in g_1$ be the endpoints of the projection of the edge of $\Sp(Y)$ dual to $\alpha$; equivalently, the leaves of the orthogeodesic foliation emanating from any point between $p$ and $q$ are isotopic to $\alpha$. Note $d(p,q) = c_\alpha(Y)$ by definition.

Pick some other geodesic $g_p \neq g_1, g_2$ of $\partial \smash{\widetilde Y}$ meeting the singular leaf of $\cO_{\partial Y}(Y)$ through $p$ and similarly pick $g_q$.
Then we have
\[d(g_p, g_q) \le c_\alpha(Y) + 4\rad(Y)\]
by tracing the path from $g_p$ to $p$ along a singular leaf, then to $q$ along $g_1$, then to $g_q$ along another singular leaf.
However, $g_p$ and $g_q$ are separated by $\alpha$, so 
\[c_\alpha(T) \le d_T(\dfl(g_p), \dfl(g_q)).\]
Invoking the 1-Lipschitz property of $\dfl$ finishes the proof.
\end{proof}

\subsection{Thin bands map to long edges}\label{subsec:bands_deflate}
We now show converses to the results of the previous subsection.
First, let us show that if two geodesics of $\partial Y$ are close (equivalently, if $\Sp(Y)$ has a long enough edge) then they map to the same edge of $T$.

\begin{lemma}[Share edge when close]\label{lem:thickthin_deflation}
There is a universal $\varepsilon_0$ ($\approx .83$) such that the following holds. Let $\dfl: Y \to T$ be any deflation.
Then if two geodesics $g_1$ and $g_2$ of $\partial Y$ are at distance less than $\varepsilon_0$ from each other, then the geodesics $\dfl(g_1)$ and $\dfl(g_2)$ must share an edge in $T$.
\end{lemma}

We emphasize that $\varepsilon_0$ is independent of the thickness of $Y$ and the topological type of $\Sigma$. The result holds for any hyperbolic surface with boundary.

\begin{proof}
As always, lift to the universal cover.
Consider the orthogeodesic arc $\alpha$ connecting $g_1$ and $g_2$ and let $\varepsilon$ be its length. Let $a_i \in g_i$ denote the endpoints of $\alpha$.
Suppose $\dfl(g_1)$ and $\dfl(g_2)$ do not share an edge. Then since $T$ is a tree, there is a unique geodesic path $\gamma$ connecting $\dfl(g_1)$ and $\dfl(g_2)$.
Since $\dfl$ is 1-Lipschitz, the points $\dfl(a_1)$ and $\dfl(a_2)$ are at most $\varepsilon$ apart in $T$, and thus are each at most $\varepsilon$ away from the endpoints of $\gamma$. See Figure \ref{fig:deflate_quad}.

\begin{figure}[ht]
    \centering
    \includegraphics[width=.8\linewidth]{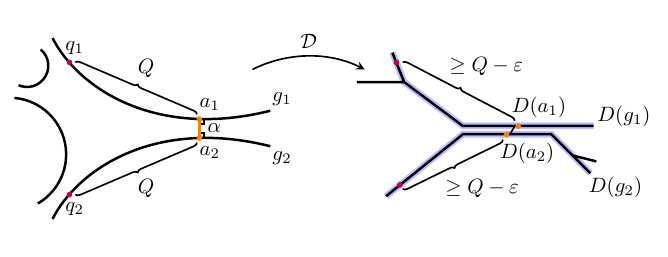}
    \caption{Points and geodesics in the proof of Lemma \ref{lem:thickthin_deflation}}
    \label{fig:deflate_quad}
\end{figure}

Now consider a pair of points $q_i \in g_i$ that are both some large distance $Q$ away from the $a_i$ (on the same side of $\alpha$).
By the hyperbolic trigonometry of Saccheri quadrilaterals,
\begin{equation}\label{eqn:hypdist}
d_{Y}(q_1, q_2) = 2 \sinh^{-1}\left(\cosh(Q) \sinh(\varepsilon/2) \right).
\end{equation}
On the other hand, since $T$ is a tree and $\dfl$ is an isometry on each leaf of $\partial Y$,
\begin{equation}\label{eqn:treedist}
d_T(\dfl(q_1), \dfl(q_2)) \ge 2Q-2\varepsilon.
\end{equation}
The 1-Lipschitz property of $\dfl$ enforces that \eqref{eqn:hypdist} $\ge$ \eqref{eqn:treedist} for all $Q$: but this only occurs so long as $\varepsilon$ is sufficiently large. Taking $\varepsilon_0 = .83$, for example, ensures that \eqref{eqn:treedist} $>$ \eqref{eqn:hypdist} for sufficiently large $Q$, yielding a contradiction. Thus $\dfl(g_1)$ and $\dfl(g_2)$ must share an edge.
\end{proof}

Now that we know that nearby geodesics $g_1$ and $g_2$ deflate to a common edge of $T$, we can further estimate how much of the $g_i$ must map to that edge.

\begin{proposition}\label{prop:short leaves deflate to same edge}
Let $\dfl:Y \to T$ be a deflation.
Suppose $\ell$ is a leaf of of $\cO_{\partial Y}(Y)$ and:
\begin{itemize}
    \item The orthogeodesic arc $\alpha$ homotopic to $\ell$ has length $\varepsilon < \varepsilon_0$.
    \item The length of $\ell$ is at most $\log 3$.
    \item The endpoints of $\ell$ are at least $\varepsilon$ away from those of the $\log 3$ leaf.
\end{itemize}
Then the endpoints of $\ell$ must map to the same edge of $T$.
\end{proposition}

As a consequence, we see that the endpoints of {\em all} sufficiently short leaves of $\cO_{\partial Y}(Y)$ must map to the same edge under any deflation.

\begin{corollary}\label{cor:uniform shortness}
With all notation as above, there is an $\ell_0$ ($\approx .425$) such that the endpoints of any leaf of $\cO_{\partial Y}(Y)$ with length at most $\ell_0$ must map to the same edge of $T$.
\end{corollary}

The exact value of the threshold $\ell_0$ will not be important to our arguments.

\begin{proof}[Proof (assuming Proposition \ref{prop:short leaves deflate to same edge})]
Given $\alpha \in \arc(Y)$, let $\ell(\alpha,Y)$ denote the maximal length of a leaf of $\cO_{\partial Y}(Y)$ isotopic to $\alpha$ and satisfying the hypotheses of the Proposition.
Our goal is to give a lower bound on $\ell(\alpha,Y)$ that is uniform in $\alpha$ and $Y$.
As one travels from the $\log 3$ leaf to the shortest representative of $\alpha$, the length of the leaves of $\cO_{\partial Y}(Y)$ decays no faster than the length of leaves in a spike.
Thus, it suffices to compute the length of the orthogeodesic leaf in a spike that is $\varepsilon_0$ away from the leaf of the length $\log 3$, which can be done explicitly (in the upper half-plane model, for example).
\end{proof}

Our proof of Proposition \ref{prop:short leaves deflate to same edge} combines ideas from Sections \ref{subsec:preserve spikes} and \ref{subsec:obstruct star} with those of Lemma \ref{lem:thickthin_deflation}.
The Proposition can also be proven by explicit hyperbolic trigonometry, directly generalizing of the argument of Lemma \ref{lem:thickthin_deflation}; we leave this alternative proof to the reader as a fun (but rather involved) exercise.

Let us first prove an analogue of Lemma \ref{lem:preserve horocycle endpts} for arbitrary deflations, showing that endpoints of some leaves of the orthogeodesic foliation must map near each other.

Given two disjoint geodesics $g_1$ and $g_2$ in $\HH^2$ and a point $p_1 \in g_1$, define $p_1^{opp} \in g_2$ to be the image of $p_1$ under the unique reflection interchanging $g_1$ and $g_2$.
Observe that if $g_1, g_2 \subset \partial Y$ and the leaf of $\cO_{\partial Y}(Y)$ through $p_1$ hits $g_2$, then it does so at $p_1^{opp}$.

\begin{lemma}[Almost symmetric]\label{lem:opp vs p2}
Suppose that $\dfl:Y \to T$ is a deflation. Let $t \in T$ and suppose that $p_1, p_2 \in \dfl^{-1}(t) \cap \partial Y$ lie on non-asymptotic geodesics $g_1, g_2$ of $\partial Y$.
Let $\alpha$ denote the unique orthogeodesic arc between $g_1$ and $g_2$ and set $\varepsilon$ to be its length.
Then $d_Y(p_2, p_1^{opp}) \le 2 \varepsilon$.
\end{lemma}

\begin{figure}[ht]
    \centering
    \includegraphics[width=.85\linewidth]{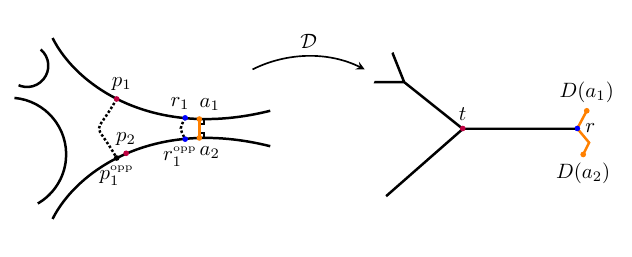}
    \caption{Points and their opposites as in the proof of Lemma \ref{lem:opp vs p2}.}
    \label{fig:opp_p2}
\end{figure}

\begin{proof}
Reference to Figure \ref{fig:opp_p2} will be useful throughout this proof.

Let $a_i$ denote the endpoints of $\alpha$ on $g_i$ and consider the geodesic segment $\gamma$ in $T$ between $\dfl(a_1)$ and $\dfl(a_2)$.
Let $r \in T$ denote the closest point projection of $t$ to $\gamma$ (it is possible that $t=r$, or that $\gamma = r$ is a point).
Observe that the geodesics from $t$ to each $\dfl(p_i)$ run from $t$ to $r$ then along $\gamma$. Thus, there is a point of $g_i$ mapping to $r$; set $r_i = \dfl^{-1}(r) \cap g_i$ and define $r_1^{opp} \in g_2$ as above (if $t=r$, then $r_i = p_i$ and the same for opposite points).

By definition of the opposite points $p_1^{opp}$ and $r_1^{opp}$, and because $\dfl$ is an isometry restricted to $g_1$ and $g_2$,
\[d_Y(p_1^{opp}, r_1^{opp}) = d_Y(p_1, r_1)
= d_T(t,r) = d_Y(p_2, r_2).\]
Thus
\[d_Y(p_1^{opp}, p_2) = d_Y(r_1^{opp}, r_2) \le d_Y(r_1, r_2),\]
where the inequality is a consequence of the fact that projecting to the boundary along the leaves of $\cO_{\partial Y}(Y)$ is $1$-Lipschitz.
Applying the triangle inequality and the fact that $\dfl$ is an isometry on the boundary again,
\begin{align*}
d_Y(r_1, r_2) 
&\le d_Y(r_1, a_1) + d_Y(a_1, a_2) + d_Y(a_2, r_2) \\
& = d_T(r, \dfl(a_1)) + \varepsilon + d_T(r, \dfl(a_2))\\
& = \varepsilon + d_T(\dfl(a_1), \dfl(a_2))
\le 2\varepsilon.
\end{align*}
where the last equality comes from the fact that $r$ is on the geodesic from $\dfl(a_1)$ to $\dfl(a_2)$.
\end{proof}

\begin{proof}[Proof of Proposition \ref{prop:short leaves deflate to same edge}]
So long as $C < \varepsilon_0$, Lemma \ref{lem:thickthin_deflation} ensures that the geodesics $g_1$ and $g_2$ are connected by an arc $\alpha$ of $\cO_{\partial Y}(Y)$ and deflate to a common edge of $T$. 
Since $\varepsilon_0 < \log 3$, there are leaves of $\cO_{\partial Y}(Y)$ homotopic to $\alpha$ (on either side of $\alpha$) of length exactly $\log 3$ (compare \cite[proof of Lemma 6.6]{shshI}).

Let $\varepsilon$ denote the length of $\alpha$, as before.
Let $t \in T$ denote either of the vertices of the common edge of $\dfl(g_1)$ and $\dfl(g_2)$ and let $p_i \in g_i$ be its preimages as in Lemma \ref{lem:opp vs p2}. That Lemma shows that $p_1$ and $p_2$ are roughly symmetric; the main content of this proof is to show that we can estimate their position using horocycles based at the endpoints of the $g_i$.

Let $\xi_1$ denote the ideal endpoint of $g_1$ obtained by traveling from $\alpha \cap g_1$ through $p_1$ and continuing on (and similarly for $\xi_2$). Consider the geodesic from $\xi_1$ to $\xi_2$ and let $m$ be its ``midpoint,'' i.e., the unique point fixed by the reflection swapping $g_1$ and $g_2$. See Figure \ref{fig:busemann}.

\begin{figure}[ht]
    \centering
    \includegraphics[width=.85\linewidth]{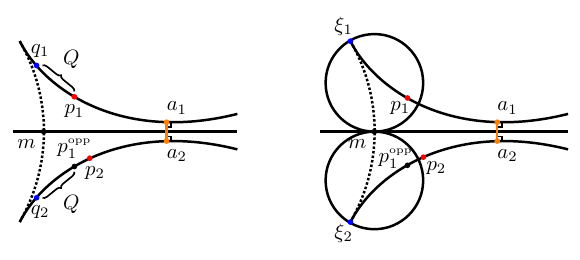}
    \caption{The horocycles used to estimate the $p_i$ in Claim \ref{claim:busemann}.}
    \label{fig:busemann}
\end{figure}

\begin{claim}\label{claim:busemann}
For $i=1,2$, we have
\[\beta_{\xi_i}(m, p_i) \ge -\varepsilon\]
where $\beta_\xi$ is the Busemann cocycle based at $\xi \in \partial \HH^2$.
\end{claim}
\begin{proof}[Proof of Claim]
As in the proof of Lemma \ref{lem:thickthin_deflation}, pick points $q_i \in g_i$ that are both some large distance $Q$ farther away from the endpoints of $\alpha$ than $p_1$ and $p_1^{opp}$.
By our choice of $p_i$ and fact that $\dfl$ is an isometry on the geodesics of $\partial Y$,
\[d_T(\dfl(q_1), \dfl(q_2)) 
= d_T(\dfl(q_1), \dfl(p_1)) + d(\dfl(q_2), \dfl(p_2))
= d_Y(q_1, p_1) + d(q_2, p_2).\]
By definition of the point $m$, we have
\[d_Y(q_1, q_2) = d_Y(q_1,m) + d_Y(q_2, m)\]
(where the two terms in the sum are actually equal).

Now since $\dfl$ is 1-Lipschitz, it must be that $d_Y(q_1, q_2) \ge d_T(\dfl(q_1), \dfl(q_2)).$
Using the equations from the previous paragraph and rearranging, we get
\begin{equation}\label{eqn:busemann_prelimit}
\left(d_Y(q_1,m) - d_Y(q_1, p_1)\right) + \left(d_Y(q_2,m) - d_Y(q_2, p_2)\right) \ge 0.
\end{equation}
Taking $Q \to \infty$, the points $q_i$ converge to the ideal endpoints $\xi_i$ and the two terms in \eqref{eqn:busemann_prelimit} converge to the Busemann cocycles at these endpoints. Thus, in the limit we get
\[\beta_{\xi_1}(m, p_1) + \beta_{\xi_2}(m, p_2) \ge 0.\]

By symmetry, we have that $\beta_{\xi_1}(m, p_1) = \beta_{\xi_2}(m, p_1^{opp})$, and Lemma \ref{lem:opp vs p2} allows us to replace $p_2$ with $p_1^{opp}$ at the cost of $2\varepsilon$.
Putting this all together, we get that
\[\beta_{\xi_1}(m, p_1) = \beta_{\xi_2}(m, p_1^{opp}) \ge -\varepsilon\]
as desired. The statement for $p_2$ follows by interchanging the labels.
\end{proof}

Consider now the points $b_i \in g_i$ defined such that $\beta_{\xi_i}(m, b_i) = 0.$ Let $B$ denote the bisector of $g_1$ and $g_2$, that is, the fixed set of the reflection interchanging them.
The geodesic segment emanating from $b_i$ orthogonal to $g_i$ and running until it hits $B$ has some length depending only on $\varepsilon$.
This length is monotonically decreasing with the length of $\alpha$, with a minimum of $\log \sqrt 3$ when $g_1$ and $g_2$ are asymptotic (this is the radius of the inscribed circle inside of an ideal triangle).
Thus, if $s_i$ denote the endpoints of the $\log 3$ leaf homotopic to $\alpha$ (on the same side as $b_i$), we have that
\[\beta_{\xi_i}(s_i, b_i) \ge 0.\]

Putting this all together, we know by Lemma \ref{lem:thickthin_deflation} that $\dfl(g_1)$ and $\dfl(g_2)$ must share an edge $e$.
The subsegment of each $g_i$ that maps to $e$ ends at the point $p_i$, and by Claim \ref{claim:busemann} and the inequality above,
\[\beta_{\xi_i}(s_i, p_i) = 
\beta_{\xi_i}(s_i, b_i) + \beta_{\xi_i}(b_i, p_i) \ge -\varepsilon.\]
Hence, the endpoints of any leaf of $\cO_{\partial Y}(Y)$ satisfying the hypotheses of the Proposition are closer to $\alpha$ than the $p_i$ and therefore must map to the same edge $e$.
\end{proof}

\subsection{Proof of Theorem \ref{mainthm:deflate}}\label{subsec:generalspinedeflate}
Using Proposition \ref{prop:short leaves deflate to same edge}, we can deduce that all sufficiently long edges of the spine of $Y$ have a corresponding edge in $T$ of comparable length.

\begin{corollary}[Thin bands deflate to long edges]\label{cor:thin bands long edges}
Set $C_0 = 2\log(6 + 3\sqrt 3)$ $(\approx 4.83)$.
If there is an arc $\alpha \in \arc(Y)$ with length $\ell(\alpha) < \varepsilon_0$ and weight $c_{\alpha}(Y) > C_0$, then
\[c_\alpha(Y)- C_0 \le c_\alpha(T).\]
\end{corollary}

If one increases $C_0$ to $2\pi/\varepsilon_0 (\approx 7.57)$, then \cite[Lemma 6.6]{shshI} implies that $\alpha$ has length at most $\varepsilon_0$. Similarly, using the fact that lengths of leaves of $\cO_{\partial Y}(Y)$ grow uniformly exponentially as one travels out from the geodesic representative, there is a slightly smaller $\varepsilon_0'$ such that if one stipulates that the length of $\alpha$ is at most $\varepsilon_0'$, then $c_\alpha(Y) \ge C_0$ (an exact formula for $\varepsilon_0'$ can be deduced from \cite[Lemma 6.6]{shshI}).
We have chosen not to phrase the result in this way to keep the source of our bounds more transparent.

\begin{proof}
In the following proof, whenever we talk about the distance between leaves of $\cO_{\partial Y}(Y)$, we mean as measured along $\partial Y$. Equivalently, this is the same as the distance between the corresponding points in the spine after collapsing the leaves of $\cO_{\partial Y}(Y)$.

As recorded in \cite[Lemma 7.10]{shshII}, the leaves homotopic to $\alpha$ of length $\log 3$ are at most $\log(2 + \sqrt 3)$ away from the singular leaves of $\cO_{\partial Y}(Y)$.
Proposition \ref{prop:short leaves deflate to same edge} ensures that all leaves homotopic to $\alpha$ that are at least $\ell(\alpha) < \log 3$ farther from the singular leaves than the $\log 3$ leaves must map to the same edge of $T$ under deflation.
Thus, all leaves homotopic to $\alpha$ that are at least
\[\log(2 + \sqrt 3) + \log 3 = C_0/2\]
away from the singular leaves must map to the same edge $e$ of $T$. 
Combining these estimates on each side of the band of parallel leaves homotopic to $\alpha$ finishes the proof.
\end{proof}

When $\arc(Y)$ is maximal, these bounds get better as the weights get larger.

\begin{corollary}\label{cor:thin bands long edges max}
Suppose that $\arc(Y)$ is maximal and $c_\alpha(Y) > C_0$ for every $\alpha \in \arc(Y)$.
Set $\varepsilon$ to denote the maximum length of any arc of $\arc(Y)$ and set $E$ to be the maximum distance between the singular leaves of $\cO_{\partial Y}(Y)$ and the leaves of length $\log 3$. Then for $\alpha \in \arc(Y)$,
\[| c_\alpha(Y) - c_\alpha(T)| \le 2(E + \varepsilon).\]
\end{corollary}

We note that $E$ goes to 0 as $\varepsilon$ does and vice versa.

\begin{proof}
The lower bound can be extracted immediately from the proof of Corollary \ref{cor:thin bands long edges}: simply replace $\log(2+\sqrt 3)$ by $E$ and $\log 3$ by $\varepsilon$ and repeat the proof {\em mutatis mutandis.}

To get the upper bound, we need to use the maximality hypothesis. As in the proof above, let $u$ be the endpoint of a singular leaf on some geodesic $g_1 \subset \partial Y$.
Let $p_2, p_3\in g_1$ denote the points distance $E + \varepsilon$ on either side of $u$; the orthogeodesic leaves through these points are nonsingular and meet some other geodesics $g_2 \neq g_3$ of $\partial Y$.
Observe that $d(p_2, p_3) = 2(E+\varepsilon).$

By the first statement of the Corollary and maximality, we know that $p_2$ and $p_3$ must map to the interiors of different edges $e_2$ and $e_3$ of $T$. However, by Lemma \ref{lem:deflation_doubleback}, we know $g_2$ does not map to $e_3$ and $g_3$ does not map to $e_2$, so $p_2$ and $p_3$ must each be within $2(E+\varepsilon)$ of the endpoints of their respective edges.
Combining estimates on each side as before, we get that the length of the edge of $T$ corresponding to an arc $\alpha$ of $\arc(Y)$ is at most
\[(c_\alpha(Y) - 2(E+\varepsilon)) + 2(2(E+\varepsilon)),\]
completing the proof.
\end{proof}

Theorem \ref{mainthm:deflate} now follows by assembling the results we have proven to this point.

\begin{proof}[Proof of Theorem \ref{mainthm:deflate}]
Recall that we are given a deflation $\dfl:Y \to T$. 
Set 
\[C = 2\pi/\varepsilon_0 + 4\rad(Y).\]
By Lemma \ref{lem:rad unif bdd}, this is uniformly bounded depending only on the systole and topological type of $Y$.
We need to prove:
\begin{enumerate}
    \item For $\alpha \in \arc(Y) \cap \arc(T),$ we have
$|c_\alpha(Y) - c_\alpha(T)| \le C$.
    \item For $\alpha \in \arc(Y) \setminus \arc(T),$ we have
    $c_\alpha(Y) \le C$.
    \item For $\alpha \in \arc(T) \setminus \arc(Y),$ we have
    $c_\alpha(T) \le C$.
\end{enumerate}
The combination of Lemma \ref{lem:lengthcomparison T<Y} and Corollary \ref{cor:thin bands long edges} immediately yields (1).
To get (2), if $\alpha \in \arc(Y)$ has weight at least $C$, then Lemma 6.6 of \cite{shshI} implies that the orthogeodesic representative of $\alpha$ has length at most $\varepsilon_0$, so Lemma \ref{lem:thickthin_deflation} ensures $\alpha \in \arc(T)$.
Claim (3) is just the statement of Lemma \ref{lem:long edges come from bands}.
\end{proof}

\subsection{The asymptotic shape of a Lipschitz geodesic}\label{subsec:deflconseq_geos}
As an immediate consequence of our work in the previous subsection, any short enough arc of $\cO_{\partial Y}(Y)$ must persist after applying a boundary-tight $f:Y \to Y'$ and we can even estimate its weight (equivalently, the length of the dual edge of $\Sp(Y')$).
The following quantifies Corollary \ref{cor:rescale}.

\begin{proposition}\label{prop:shortarcs_persist_optimal}
Suppose that $f:Y \to Y'$ is a boundary-tight Lipschitz map with Lipschitz constant $L>1$. If $\arc(Y)$ contains an arc $\alpha$ of length less than $\varepsilon_0$, then $\arc(Y')$ also contains $\alpha$.
For any such arc,
\[c_\alpha(Y)-C_0
\le \frac{c_\alpha(Y')}{L}
\le c_\alpha(Y) + 4\rad(Y).\]
\end{proposition}

\begin{proof}
By Lemma \ref{lem:composite_deflation}, the composite map
\[Y \xrightarrow{f} Y' \to \Sp(Y') \xrightarrow{1/L} 1/L \cdot \Sp(Y'),\]
where the second map is projection along the leaves of $\cO_{\partial Y'}(Y')$, is a deflation. Lemma \ref{lem:thickthin_deflation} ensures that geodesics of $\partial Y$ connected by an arc of length $< \varepsilon_0$ map to a common edge of $1/L \Sp(Y')$, hence the corresponding geodesics of $Y'$ must as well. Thus there is an arc of $\arc(Y')$ connecting these geodesics.
The second statement follows by applying Corollary \ref{cor:thin bands long edges} and Lemma \ref{lem:lengthcomparison T<Y} to this composite deflation.
\end{proof}

This allows us to give a description of the life cycle of an arbitrary Thurston geodesic $(X_t)$ with tension lamination $\lambda$: either it always has small systole, in which case it looks like a Thurston geodesic on a lower-complexity subsurface, or the following must occur.
Every thin band of $X_0 \setminus \lambda$ persists in $X_t \setminus \lambda$ and is (roughly) exponentially stretched. As time goes on, new regions of $X_t \setminus \lambda$ may become thin, at which point those bands persist and are stretched in all $X_{t+s} \setminus \lambda$.
New bands occur at only finitely many critical times; after the last of these, the final arc system is stretched exponentially while the geometry of the complement can only vary in a bounded way.

\begin{theorem}[Thick-thin along geodesics]\label{thm:lifecycle}
Let $(X_t)_{t=0}^\infty$ be an arbitrary Thurston geodesic ray with tension lamination $\lambda$ and whose systole is uniformly bounded below.
Then there is a sequence of times 
\[0 < t_1 < \ldots < t_n \text{ where } n \le \dim_{\RR} \T(S \setminus \lambda) \]
and a sequence of nested arc systems 
$\arc_0 \subset \arc_1 \subset \ldots \subset \arc_n$ on $S \setminus \lambda$ such that:
\begin{enumerate}
    \item For all $t > t_i$, the arc system $\arc(X_t)$ contains $\arc_i$.
    \item For each arc $\alpha$ of $\arc_n$ and any $s >0$,
\[\frac{c_\alpha(X_{t+s})}{c_\alpha(X_t)} \xrightarrow{t\to \infty} e^s.\]
    \item Set $i(t) := \max \{i \mid t > t_i\}$. Then the weights of the arcs of $\arc(X_t) \setminus \arc_{i(t)}$ are universally bounded.
\end{enumerate}
Moreover, if the final arc system $\arc_n$ is maximal, then (2) can be upgraded to
\[{c_\alpha(X_{t+s})} - e^s{c_\alpha(X_t)} \xrightarrow{t\to \infty} 0.\]
\end{theorem}
\begin{proof}
The statements in this theorem are just an encapsulation of our discussion to this point.
Set $\arc_0$ to be the set of arcs on $X_0 \setminus \lambda$ with length at most $\varepsilon_0$ and weight greater than $C_0$.
Since there is a boundary-tight map $X_0 \setminus \lambda \to X_t \setminus \lambda$ for all $t >0$, Lemma \ref{lem:thickthin_deflation} ensures that $\arc(X_t)$ must contain $\arc_0$ and Proposition \ref{prop:shortarcs_persist_optimal} ensures that the weight of these arcs grows roughly exponentially.
Applying the Proposition to the complement of $\lambda$ in $X_t$ and $X_{t+s}$ for $t \gg 0$ gives the desired growth rate of the weights (note that $\rad(X_t \setminus \lambda)$ is uniformly bounded by assumption and Lemma \ref{lem:rad unif bdd}).

If any other arc of $\arc(X_t)$ becomes $\varepsilon_0$-short and gains weight greater than $C_0$, then we record the first time it does so as $t_1$ and repeat the above argument. This process eventually terminates because the maximal size of an arc system on $X \setminus \lambda$ is the same as the dimension of its Teichm{\"u}ller space.
If an arc of $X_t$ does not appear then either its weight is not $C_0$-large or it is not $\varepsilon_0$-short.
As in the remark after Corollary \ref{cor:thin bands long edges}, either possibility implies a universal bound (slightly larger than $C_0$) on the weight \cite[Lemma 6.6]{shshI}.

The final statement of the Theorem follows from Corollary \ref{cor:thin bands long edges max}. Namely, since the weights of all arcs of $\arc(X_t)$ are going to $\infty$, their lengths are going to $0$, and the complementary hexagons of $X_t \setminus (\lambda \cup \arc(X_t))$ are converging to ideal triangles.
Thus, the Corollary implies
\[
|c_{\alpha}(X_t) - e^{-s}c_{\alpha}(X_{t+s})|
\xrightarrow{t\to \infty} 0\]
which for fixed $s$ gives the desired statement.
\end{proof}

\begin{remark}
It is {\em not} necessarily true that for fixed $t$,
\[\frac{c_\alpha(X_{t+s})}{c_\alpha(X_t)} \sim e^s\]
as $s \to \infty$.
This is because the endpoints of the singular leaves of $\cO_{\lambda}(X_t \setminus \lambda)$ (which are used to define the arc weights) do not have to be mapped precisely to the endpoints of the singular leaves of $\Ol(X_{t+s} \setminus \lambda)$, only nearby. An explicit example can be constructed by taking concatenations of stretch lines in different completions of the same lamination.
This also explains why we get a stronger statement in Theorem \ref{thm:lifecycle} for maximal arc systems: as $t \to \infty$, any maps from $X_t$ to $X_{t+s}$ must take endpoints closer and closer to endpoints.
\end{remark}

\begin{remark}
The final arc system $\arc_n$ may be empty. For example, if $\lambda$ is maximal then the Teichm{\"u}ller space of $S \setminus \lambda$ is a single point. 
More generally, whenever $\lambda$ fills $S$ (so that $S \setminus \lambda$ consists of a union of ideal polygons) then one can construct Thurston geodesics with empty $\arc_n$ using ideal polygons containing embedded horogons (see the \hyperref[sec:intro]{Introduction}).

If the systoles of any $X_t$ is at least $s_0(\varepsilon_0)$, then Lemma \ref{lem:big sys then fill} ensures that all $\arc_i$ for $i \ge i(t)$ are filling. However, in general $\arc_n$ does not have to fill $S \setminus \lambda$; compare Remark \ref{rmk:isolate}.
\end{remark}

We highlight one immediate consequence of this theorem, which essentially says that along every Thurston geodesic, the complement has a well-defined limit.

\begin{corollary}
Given any Thurston geodesic $(X_t)_{t=0}^\infty$ with tension lamination $\lambda$, the rescaled arc system $e^{-t} \cdot \arc(X_t\setminus \lambda)$ has a well-defined limit in the (geometric realization of the) arc complex $\mathscr{A}(S \setminus \lambda)$.
\end{corollary}

We note that this limit must be taken in the arc complex, not in the moduli space of ribbon graphs; this is because of the possibility that $\arc_n$ may not be filling.
One cannot even take a limit in Kontsevich's compactification (see \cite{Mond_handbook}) --- which can be described as a certain blowup of the arc complex \cite{BowdEpst, Loo_arccx} --- because the (bounded weight) arcs of $\arc(X_t) \setminus \arc_n$ may not have a stable combinatorial type.

\section{Constructing Lipschitz maps via h-gons}\label{sec:construct horogon}
We now complete the proof of Theorem \ref{mainthm:polygon} by constructing new boundary-tight maps.
Given an ideal polygon $P$, define $P_L$ to be the polygon such that 
\[L \cdot \Sp(P) = \Sp(P_L).\]
For any $P$, define its {\bf waist} to be the length of the shortest (possibly singular) leaf of its orthogeodesic foliation $\cO_{\partial P}(P)$ that is not isotopic into a spike.

Our main goal for this section is to prove the following:

\begin{theorem}\label{thm:scalespineup}
For every $n \ge 3$ and $w >0$, there is an $L_0$ such that for any ideal $n$-gon $P$ with waist at least $w$ and any $L \ge L_0$, there is an $L$-Lipschitz, boundary-tight map $P \to P_L$.
\end{theorem}

The methods we use to prove this Theorem will also yield (a stronger version of) the sufficiency part of Theorem \ref{mainthm:polygon}.(1) --- see Corollary \ref{cor:Horo to Horo}. The proof of Theorem \ref{mainthm:polygon}.(2) will be deduced from Theorem \ref{thm:scalespineup} at the end of the section.

The main idea of the proof of Theorem \ref{thm:scalespineup} is to decompose $P$ into nice pieces, build maps on those pieces, and then glue them back together.
To perform this gluing, our maps must agree on their overlaps, which in turn requires us to ensure that our constructions are compatible with each other.
One of the key points of the proof is therefore to get uniform control on the geometry of the pieces we use (Lemma \ref{lem:Hgonbounds}, and compare the hypotheses in Proposition \ref{Prop:HLip}).

In Section \ref{sec:decomp} we describe a decomposition of $P$ using its spine; this also results in a version of the orthogeodesic foliation $\cO_{\partial P}(P)$ with smooth leaves.
From there, in Sections \ref{subsec:hgon vs}--\ref{subsec:rectmaps} we build maps on each of the pieces and bound their Lipschitz constants. There are two important subtleties here: the first is that the na{\"i}ve maps between some pieces are generally {\em not} Lipschitz (Example \ref{example:nonLip}), and the other is that we need to ensure that our maps realize their Lipschitz constants along $\partial P$, which must be done by explicit construction (Proposition \ref{Prop:RectLip} and Lemma \ref{Lem:CuspLip}). 
Finally, in Section \ref{subsec:glue maps} we glue these maps together.

\subsection{Decompositions of polygons}\label{sec:decomp}
In order to prove Theorem \ref{thm:scalespineup}, we introduce new geometric objects called \textbf{h-gons}.
An h-gon is a compact polygon $H \subset \mathbb{H}^2$ with at least $3$ sides where each side is either a segment of a horocycle or a hypercycle and all vertices have angle 0.
If all sides have constant curvature $1$, {\em i.e.}, they are horocyclic, then we still call $H$ a horogon. For simplicity, we also refer to a segment of a horo- or hypercycle as an \textbf{h-cycle}. An h-gon is said to be \textbf{embedded} if no two non-adjacent sides touch. We also consider \textbf{rectangles}, which throughout this section are right-angled quadrilaterals bounded on two opposite sides by hypercycles, which we refer to as {\bf horizontal}, and on the other two sides by geodesics of the same length, which we refer to as {\bf vertical}. Finally, a \textbf{spike} is a closed non-compact region bounded by two asymptotic geodesic rays and a segment of a horocycle based at their shared point at infinity. 

We can decompose any ideal polygon into a union of these pieces.

\begin{lemma}\label{Lem:HgonDecomp}
Let $P$ be an ideal polygon. Dual to its spine $\Sp(P)$, there is a decomposition of $P$ into embedded h-gons, rectangles, and spikes, such that each vertex corresponds to an embedded h-gon, each compact edge to a rectangle, and each ray to a spike.
\end{lemma}

\begin{figure}[ht]
    \centering
    \includegraphics[width=0.7\linewidth]{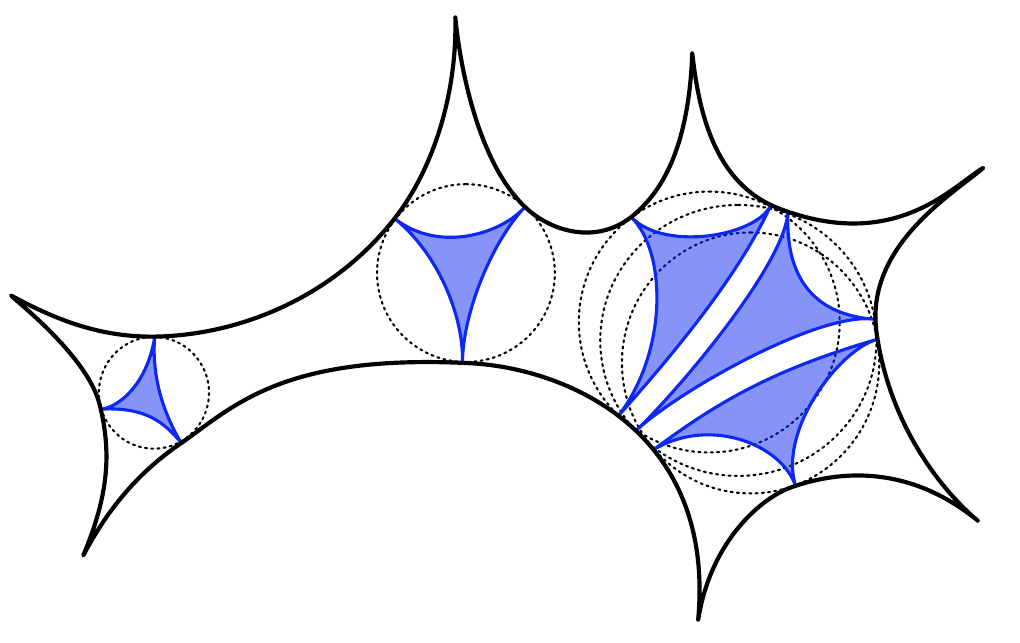}
    \caption{The decomposition of an ideal polygon dual to its spine.}
    \label{fig:decomp}
\end{figure}

\begin{proof}
Recall that a vertex $x$ of $\Sp(P)$ of valence $k$ means there is a closed ball $B \subset P$ centered at $x$ such that $\partial B$ is tangent to $k$ sides of $P$. Label the sides to which $B$ is tangent by $s_0$, $s_1$, \ldots, $s_{k-1}$ in cyclic order, and define $x_i = s_i \cap B$. Then $x_i$ and $x_{i+1}$, where $i$ is taken mod $k$, can be joined by a horocyclic segment if $s_i$ and $s_{i+1}$ are asymptotic, and by a hypercyclic segment orthogonal to both $s_i$ and $s_{i+1}$ otherwise. See Figure \ref{fig:decomp}. Let $H$ denote the union of these h-cycles.

To see that $H$ is an embedded h-gon, we observe that each h-cycle of $H$ is a circular arc orthogonal to $P$ at its endpoints, hence is orthogonal to the circle $\partial B$. We therefore have a collection of circular arcs connecting cyclically ordered points on $\partial B$ and all orthogonal to $\partial B$.
These must be disjoint except for their endpoints.\footnote{For an amusing proof of this fact, equip $B$ with its Poincar{\'e} metric. Then the h-cycles become a chain of consecutively asymptotic geodesics, which clearly cannot be tangent or intersect. \label{footnote:hyphgon}}

By construction, the complementary regions of $\bigcup_x H_x$ are comprised of spikes, one for each ray of $\Sp(P)$, and rectangles, one for each compact edge of $\Sp(P)$.
The rectangles may be foliated by hypercycles orthogonal to $\partial P$ and connecting the endpoints of leaves of the orthogeodesic foliation $\cO_{\partial P}(P)$.
\end{proof}

We say an embedded h-gon $H$ is \textbf{inscribable} if there is an ideal polygon such that $H$ appears as one of the pieces of the decomposition of Lemma \ref{Lem:HgonDecomp}.
Equivalently, $H$ is inscribable if there exists a circle $C$ containing each of the vertices of $H$ and perpendicular to every side, such that the geodesics tangent to $C$ at the points of $H \cap C$ are all disjoint, except perhaps at their ideal endpoints.\footnote{
Observe that one can always connect the endpoints of these geodesics by a chain of consecutively tangent geodesics very far away from $C$, hence an inscribable h-gon always appears as a decomposition piece of {\em some} polygon, perhaps with many sides.}
The {\bf radius} of an inscribable h-gon is the hyperbolic radius of its circumscribing circle.

\begin{remark}
Not every h-gon which can be drawn inside an ideal polygon $P$ is necessarily inscribable (for example, embedded horogons are not always).
\end{remark}

The following is immediate from the fact that the area of an ideal $n$-gon is $(n-2)\pi$.

\begin{lemma}\label{lem:Hgon rad bd}
For every $n \ge 3$, there is an $r_n$ such that any h-gon arising from the decomposition of an ideal $n$-gon (via Lemma \ref{Lem:HgonDecomp}) has radius at most $r_n$.
In particular, its diameter and the length of its longest side are both at most $2r_n$.
\end{lemma}

This bound gives us uniform control on the geometry of such h-gons. Observe that the following bounds do not hold for the family of embedded horogons inside ideal $n$-gons (which are usually not inscribable).

\begin{lemma}[Uniform bounds on inscribable h-gons] \label{lem:Hgonbounds}
The shortest side of any inscribable h-gon has length at least $1$.
Consequently, for every $r>0$, there are $\theta_r, c_r, d_r$ such that for every inscribable h-gon $H$ of radius $\le r$:
\begin{enumerate}
    \item The angle formed between the radii connecting consecutive vertices of $H$ to the center of the circumscribing circle is at least $\theta_r$.
    \item At each vertex of $H$, the (signed) geodesic curvatures of the incoming sides differ by at least $c_r$.
    \item The non-consecutive sides of $H$ are at least $w_r$ apart.
\end{enumerate}
\end{lemma}
\begin{proof}
We will actually prove a more detailed version of (1), from which the lower bound on side length will follow.
Consider any hypercycle $\gamma$ of any inscribable h-gon, and let $x$ and $y$ denote its endpoints on the circumscribing circle $C$.
The geodesics $g_x$ and $g_y$ tangent to $C$ at $x$ and $y$ are disjoint by assumption, which in particular means that the angle formed by the radii through $x$ and $y$ is at least
$2 \sin^{-1}(\sech(r))$, where $r$ is the radius of $C$ 
\cite[Theorem 2.2.2 (v)]{Buser}.
See Figure \ref{fig:inscribable}.

\begin{figure}
    \centering
    \includegraphics[width=0.5\linewidth]{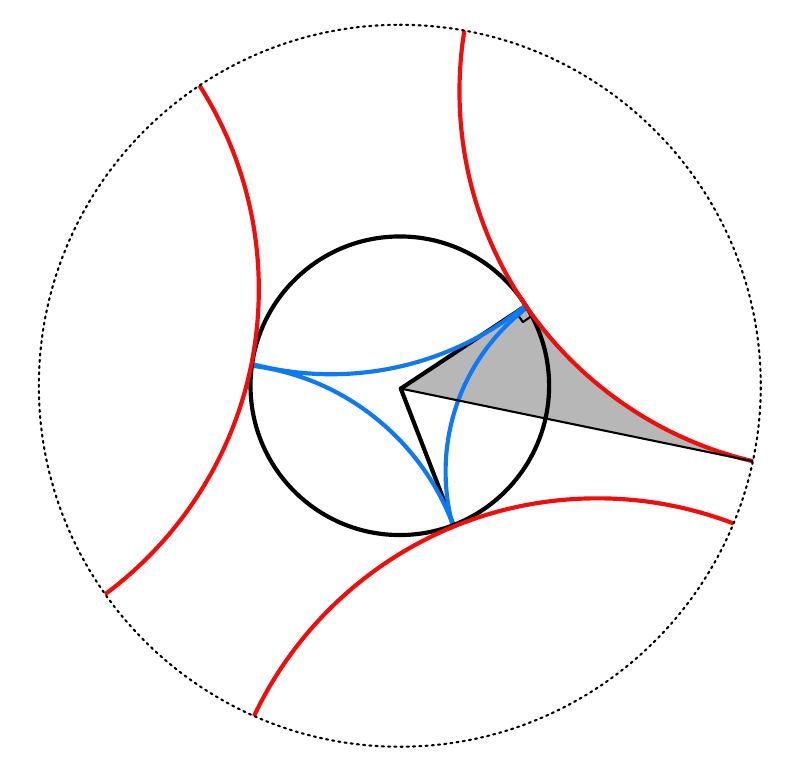}
    \caption{An inscribable h-gon and geodesics tangent to its circumscribing circle. The shaded region is a triangle with angles $0, \pi/2,$ and $\sin^{-1}(\sech(r))$.}
    \label{fig:inscribable}
\end{figure}

To establish the lower bound on side length, we observe that an h-gon must have at least 3 sides, so the radii connecting vertices of $H$ to the center of $C$ subtend some angle of size at most $2\pi/3$. This puts a lower bound on $r$ of $\log \sqrt 3$. Since the length of the h-cycle connecting two points is monotonic in radius and in the subtended angle, the absolute minimum is achieved when $r = \log \sqrt 3$ and the angle is $2\pi/3$, i.e., when $\gamma$ is one of the horocycles of the horogon in an ideal triangle.
These horocycles each have unit length.

To get the bound in (2), we observe that if the curvature of two h-cycles emanating from the same vertex is very small, then their other points of intersection with $C$ are very close, contradicting (1).
For (3), we similarly note that if two non-consecutive sides of $H$ are close, then since their geodesic curvature is bounded by $1$ their endpoints on $C$ must also be close, leading to the same contradiction.
\end{proof}

Similarly, we say a rectangle $R$ is \textbf{inscribable} if there exists an ideal polygon $P$ such that $R$ arises a decomposition piece determined by $\Sp(P)$.
Note that every inscribed rectangle always buttresses some inscribed h-gon. In particular, the following statement holds.

\begin{corollary} \label{cor:rectbounds}
If a rectangle $R$ is inscribable in an $n$-sided ideal polygon, then the arc lengths of  horizontal sides of $R$ are bounded above by $2r_n$ and below by 1.
\end{corollary}

\subsection{The vertices of h-gons}\label{subsec:hgon vs}
In the next three subsections, we build Lipschitz maps between the pieces of the decomposition of Lemma \ref{Lem:HgonDecomp} which we will then glue to yield Lipschitz maps between the ambient ideal polygons.

To aid this, we first recall Lang and Schroeder's extension of the Kirszbraun–Valentine theorem to hyperbolic space \cite{LS_KV} (see also \cite{GK}). 

\begin{theorem}\label{Thm:KV}
Let $D \subset \mathbb{H}^n$ be a set and $f : D \to  \mathbb{H}^n$ a Lipschitz map. Then $f$ can be extended to a map from $\mathbb{H}^n$ to itself with the same Lipschitz constant.
\end{theorem}

This theorem will allow us to build Lipschitz maps between h-gons $H_1$ and $H_2$ by specifying a Lipschitz map $\partial H_1 \to \partial H_2$.
Unfortunately, it turns out that the map which stretches each h-cycle linearly with respect to arc length is generally {\em not} Lipschitz if tangent h-cycles are stretched at different rates 
(compare the proof of Lemma \ref{lem:preserve horocycle endpts}).

\begin{example}\label{example:nonLip}
Suppose two h-cycles $h_1$ and $h_2$ are tangent at a point $v$ and we are given a map $f$ that affinely stretches them by rates $L_1 \neq L_2$.
Take points $p_i$ on $h_i$ which are $\varepsilon$ away from $v$; then $d(p_1, p_2) \approx O(\varepsilon^2)$.
However, $d(v,f(p_i)) \approx L_i \varepsilon$, and
\[d(f(p_1), f(p_2)) \gtrapprox (L_1 - L_2) \varepsilon - O(\varepsilon^2)\]
Taking $\varepsilon$ sufficiently small makes the Lipschitz constant blow up.
\end{example}

Thus, we will have to work harder and specify maps which uniformly expand/contract near the vertices of $H_1$.
This is the content of the current subsection.
\medskip

In the following, we set up some notation to describe the local picture at a vertex of an h-gon.
For each $0 \neq r \in [-2,2]$, consider the Euclidean circle 
\[x^2 + \left(y-\frac{1}{r}\right)^2 = \frac{1}{r^2}.\]
Consider the Poincar{\'e} disk model $\mathbb{D}^2$ of the hyperbolic plane; then the intersection $h_r$ of this circle with $\mathbb{D}^2$ is an h-cycle through $0$. Set $h_0$ to be the intersection of the $x$-axis with the disk.
Observe that $h_0$ is a geodesic, $h_{\pm 2}$ are horocycles, and the rest are hypercycles.
Denote by $\mathbb{D}^2_+$ the region of $\mathbb{H}^2$ with $x \ge 0$, and set $h_r^+ = h_r \cap \mathbb{D}^2_+$. 

Since we will be concerned with the local picture at a point of tangency, we once and for all fix a small $\delta_0$ so that the hyperbolic and Euclidean metrics are uniformly bi-Lipschitz in the $\delta_0$ neighborhood of $0 \in \mathbb{D}^2$.
All computations in this section will take place in this neighborhood.
We write $A \approx B$ to mean $A$ and $B$ are equal up to a uniform multiplicative error depending only on our choice of $\delta_0$, and likewise for the notation $A \lessapprox B$.

Given $\varepsilon \ll 1$, let $C_\varepsilon$ be the (Euclidean) circle $(x-\varepsilon)^2 + y^2 = \varepsilon^2$.
By construction, both $C_\varepsilon$ and $h_r$ pass through $0$, and $C_\varepsilon$ is tangent to the $y$-axis while $h_r$ is tangent to the $x$-axis. Therefore, $C_\varepsilon$ always meets $h_r$ (twice) orthogonally.
Given $r \in [-2,2]$ and $\varepsilon >0$, let $h_r^+(\varepsilon)$ be the segment of $h_r^+$ within $C_\varepsilon$. The \textbf{radial parametrization}\footnote{
In the vein of footnote \ref{footnote:hyphgon}, this parametrization is akin to taking ``horocyclic coordinates'' near each vertex of an h-gon \cite[Figure 1.1.2]{Buser}.}
of $h_r^+(\varepsilon)$ is the map $[0,\varepsilon] \to h_r^+(\varepsilon)$ given by $u \mapsto C_u \cap h_r^+ $.

\begin{lemma}[Radial parametrizations uniformly bi-Lipschitz] \label{lem:radcontract}
There is $M > 1$ such that
\[ \frac{1}{M} \varepsilon \le \ell(h_r^+(\varepsilon)) \le M \varepsilon\]
for all $\varepsilon \le \delta_0$ and all $r \in [-2,2]$, where $\ell(h_r^+(\varepsilon))$ denotes (hyperbolic) arc length.
\end{lemma}
\begin{proof}
The corresponding statement for Euclidean arc length follows because the arcs $h_r$ have bounded (Euclidean) curvature.
Our choice of $\delta_0$ gives a uniform comparison between the hyperbolic and Euclidean metrics, so the desired statement follows.
\end{proof}

Any map $g: [0,\varepsilon] \to [0,\varepsilon']$ induces a map $g_{r,r'}$ from $h_r^+(\varepsilon)$ to $h_{r'}^+(\varepsilon')$, simply by using $g$ to interlace their radial parametrizations. That is, $g_{r,r'}$ is the composition
\[h_r^+(\varepsilon) \to [0,\varepsilon] \xrightarrow{g}  [0,\varepsilon'] \to h_{r'}^+(\varepsilon').\]
Conversely, from any map $g_{r,r'}$ between small segments of the h-cycles $h_r^+$ and $h_{r'}^+$, we can recover a map  
$g: [0,\varepsilon] \to [0,\varepsilon']$ using their radial parametrizations, defined by
\[g(u) = u' \text{ where } g_{r,r'}(C_u \cap h_r^+) = C_{u'} \cap h_{r'}^+.\]

We now upgrade our radial parametrizations from an h-cycle to an entire neighborhood of the vertex of an h-gon.
Given $-2 \le r < s \le 2$ and $\varepsilon > 0$, let $N(\varepsilon, r,s)$ be the (closed) region bounded by $C_\varepsilon$, $h_r^+$, and $h_s^+$.
Its {\bf corner divergence} is defined to be $s-r$.
There is a pair of orthogonal foliations of $N(\varepsilon, r,s)$ given by the family of circles $C_u$, with $u \in (0,\varepsilon]$, and h-cycles $h_v^+$, with $v \in [r,s]$.

Given two subintervals $[r,s]$ and $[r',s']$ of $[-2,2]$, let
$l :[r,s] \to [r',s']$ be the unique affine bijection.
Then for any $g: [0,\varepsilon] \to [0,\varepsilon']$, the map 
\[\begin{array}{rccc}
f: & N(\varepsilon, r,s) & \to & N(\varepsilon', r',s') \\
& (u,v) & \mapsto & (g(u),l(v))
\end{array}\]
is called the \textbf{standard map} induced by $g$.
In particular, $f$ takes leaves of the foliations on $N(\varepsilon,r,s)$ to those of $N(\varepsilon',r',s')$. 

\begin{lemma}[Lipschitz maps near vertices] \label{lem:HLip}
Fix $\delta_0$ as above and $c>0$.
Given any $\varepsilon, \varepsilon' \le \delta_0$, any $K$-Lipschitz map $g:[0,\varepsilon] \to [0,\varepsilon']$, and any pair of intervals $[r,s]$ and $[r',s']$ with corner divergence $\ge c$,
there exists $L$ such that the standard map 
\[f: N(\varepsilon, r,s)  \to  N(\varepsilon', r',s')\]
induced by $g$ is $L$-Lipschitz.
\end{lemma}
\begin{proof}
To show $f$ is Lipschitz, we bound its Lipschitz constant along the leaves of each foliation.
In the $u$ direction, we observe that the restriction of $f$ to any h-cycle $h_v$ is just the map $g_{v, \ell(v)}$ described after Lemma \ref{lem:radcontract}, which by that Lemma is at worst $KM^2$-Lipschitz.

In the $v$ direction, fix $u$ and let $\sigma_u$ denote the arc length element along the circle $C_u$. Our goal is to bound $d \sigma_{g(u)} / d \sigma_u$.
One explicitly computes that
\[C_u \cap h_v^+ = \left\{0, \frac{2u}{1+u^2v^2} + i \frac{2u^2v}{1+u^2v^2}\right\}\]
and so
\[\frac{\partial x}{\partial v} = 
\frac{-4u^3v}{(1+u^2v^2 )^2}
\qquad
\text{and}
\qquad
\frac{\partial y}{\partial v} = 
\frac{2u^2(1-u^2v^2)}{(1+u^2v^2)^2}.
\]
In particular, the change in the arc length of $C_u$ is given by 
\[\frac{d\sigma_u}{dv} = 
\frac{2u^2}{1+u^2v^2}.\]
We therefore have
\[\frac{d\sigma_{g(u)}}{d\sigma_{u}} = 
\frac{d\sigma_{g(u)}}{d l(v)} \cdot \frac{d l(v)}{dv} \cdot \frac{dv}{d\sigma_{u}} =
l'(v) \cdot \frac{g(u)^2}{u^2}\cdot 
\frac{1+u^2v^2}{1+g(u)^2 l(v)^2}
.\]
The last term is uniformly bounded once $\varepsilon$ and $\varepsilon'$ are small enough, so since $g$ is $K$-Lipschitz, $c \le s-r$, and $s'-r' \le 4$, we get
\[\frac{d\sigma_{g(u)}}{d\sigma_{u}} \lessapprox 4K^2/c.\]
Combining this with our prior bound for fixed $v$ finishes the proof.
\end{proof}

\subsection{Maps on h-gons}\label{subsec:hgonmaps}
Now that we have constructed Lipschitz maps near the vertices of h-gons, we can use Theorem \ref{Thm:KV} to extend these to maps between the h-gons themselves. First, we must set up a bit more notation to describe the neighborhoods of vertices in h-gons.

Let $H$ be an h-gon and $v \in H$ be any of its vertices. 
A \textbf{standard neighborhood} of $v$ in $H$ is the region in $H$ bounded by a circle through $v$ and orthogonal to the sides of $H$ meeting at $v$.
For $\varepsilon > 0$, we denote by $N_\varepsilon(v)$ the standard neighborhood of $v$ bounded by the circle with the same curvature as $C_\varepsilon$;
by abuse of notation, we refer to $\varepsilon$ as the \textbf{radius} of $N_\varepsilon(v)$.
A \textbf{truncation} of $H$ is a choice of a standard neighborhood at each vertex of $h$ such that no two pairwise intersect; $H$ is \textbf{$\varepsilon$-truncatable} if the standard neigborhoods of its vertices of radius $\varepsilon$ are all disjoint.
The complement of the truncating neighborhoods is a right-angled polygon $\overline{H}$ with sides alternating between h-cycles and segments of circles.
The \textbf{injectivity radius} of $\overline{H}$ is the quantity $\inf_{p,q} d(p,q)$, where $p, q$ lie on non-adjacent sides of $\overline{H}$; note that this is $O(\varepsilon^2)$ once $\varepsilon$ is sufficiently small.

Likewise, given an h-cycle $\gamma$ (thought of as a side of $H$) with vertices $v$ and $w$, the $\varepsilon$-truncation $\bar{\gamma}$ of $\gamma$ is the removal of $N_\varepsilon(v) \cup N_{\varepsilon}(w)$ from $\gamma$.
Let $b$ be the arc length of $\bar{\gamma}$; then we can parametrize $\gamma$ by $[0,b+2\varepsilon]$ as follows:
   \begin{itemize}
      \item $[\varepsilon,b+\varepsilon] \to \bar{\gamma}$ is the arc length parametrization.
      \item $\gamma \setminus \bar{\gamma}$ is radially parametrized by $[0,\varepsilon] \cup [b+\varepsilon,b+2\varepsilon]$.
  \end{itemize}
By Lemma \ref{lem:radcontract}, this parametrization is biLipschitz with respect to arc length.
  
Given two h-cycles $\gamma, \gamma'$ which are truncatable by $\varepsilon$ and $\varepsilon'$, respectively, then any $g: [0,\varepsilon] \to [0,\varepsilon']$ induces a \textbf{standard map} $g_{\gamma,\gamma'}: \gamma \to \gamma'$ via
\[
   g_{\gamma,\gamma'}(t) =
   \begin{cases}
       g(t), & t \in [0,\varepsilon] \\
       \dfrac{b'}{b} (t-\varepsilon)+\varepsilon, & t \in [\varepsilon,b+\varepsilon] \\
       g(-t+b+2\varepsilon), & t \in [b+\varepsilon,b+2\varepsilon],
   \end{cases} 
  \]
where $[0,b+2\varepsilon]$ is the above parametrization of $\gamma$, and similarly for $\gamma'$.
In other words, $g_{\gamma,\gamma'}$ is one of the maps $g_{r,r'}$ described after Lemma \ref{lem:radcontract} on each component of $\gamma \setminus \bar{\gamma}$, and maps $\bar{\gamma}$ onto $\bar{\gamma}'$ by constant slope $b'/b$ with respect to arc length. 

We can now extend the standard maps on (standard neighborhoods of) vertices and edges of an h-gon to a Lipschitz map on the entire h-gon.
Moreover, given a bound on the geometric quantities described above, the family of such maps have uniform quality. For maps between two $k$-sided h-gons, we always assume that their sides are cyclically ordered so the $j$-th side maps to the $j$-th side. 
 
\begin{proposition}[Lipschitz maps between h-gons] \label{Prop:HLip}
Let $\{H_\alpha\}_{\alpha \in I}$ be a family of embedded $k$-sided h-gons such that either all of the following hold:
\begin{enumerate}
    \item[(a)] There are uniform, positive upper and lower bounds to their side lengths.
    \item[(b)] There is a uniform, positive lower bound to the corner divergence of each vertex.
    \item[(c)] There is a uniform, positive lower bound to the distance between any pair of non-consecutive sides.
\end{enumerate}
Or:
\begin{enumerate}
    \item[(a')] There exists $n\ge 3$ such that each $H_\alpha$ is inscribable in an ideal $n$-gon.
\end{enumerate}
Then there exists $\varepsilon_I$ such that the following statement holds. 
Given any $\varepsilon, \varepsilon' \le \varepsilon_I$ and any $K$-Lipschitz bijection $g: [0,\varepsilon] \to [0,\varepsilon']$,
there exists $L$ such that for all $\alpha, \beta \in I$,
there is an $L$-Lipschitz map $f_{\alpha\beta}: H_\alpha \to H_{\beta}$ such that:
    \begin{itemize}
        \item[(1)] $f_{\alpha,\beta}$ restricts to the standard map $N_\varepsilon(v) \to N_{\varepsilon'}(f_{\alpha,\beta}(v))$ induced by $g$ near each vertex $v$ of $H_\alpha$.
        \item[(2)] $f_{\alpha \beta}$ restricts to the standard map $\gamma \to f_{\alpha,\beta}(\gamma)$ induced by $g$ on each side $\gamma$ of $H_\alpha$.
    \end{itemize}
\end{proposition}
\begin{proof}
Lemmas \ref{lem:Hgon rad bd} and \ref{lem:Hgonbounds} imply that condition (a') implies (a), (b), and (c), so we will just prove the proposition under these three assumptions.

Pick $\varepsilon_I$ such that it is smaller than the constant of Lemma \ref{lem:HLip} and such that every $H_\alpha$ is $2\varepsilon_I$-truncatable.
We will further shrink $\varepsilon_I$ at the end of the proof.
For each $j=1, \ldots, k$ and $\alpha \in I$, let $v_\alpha^j$ be the $j$-th vertex of $H_\alpha$.
As above, we consider the standard neighborhoods $N_\varepsilon(v_\alpha^j)$ of each vertex $v_\alpha^j$.
Let $\overline{H}_\alpha^\varepsilon$ denote the truncation of $H_\alpha$ with radius $\varepsilon$ at each vertex. 

Now let $\alpha, \beta \in I$ be given. 
Assumption (b) ensures that the family of standard neighborhoods satisfies the hypotheses of Lemma \ref{lem:HLip}, thus $g$ induces a standard map
\[N_\varepsilon(v_\alpha^j) \to N_{\varepsilon'}(v_\beta^j)\] 
which is $L$-Lipschitz for all $j=1, \ldots, k$, where $L$ depends only on $K$.

Set $\partial_\varepsilon \overline{H}_\alpha := \partial H_\alpha \cap \overline{H}_\alpha^\varepsilon,$
i.e., it is the union of the $\varepsilon$-truncations of the sides of $H_\alpha$.
Since each side is $2 \varepsilon_I$-truncatable, the length of the shortest $\varepsilon$-truncated side is still comparable to the lower bound from assumption (a).
Observe that 
\[H_\alpha \setminus \text{interior}\left(\overline{H}_\alpha^\varepsilon \right)
= 
\partial_\varepsilon \overline{H}_\alpha \cup  \bigcup_j N_\varepsilon(v_\alpha^j) .\]
A similar story holds for the $\varepsilon'$-truncation of $H_\beta$.
We now define the map
\[f_{\alpha \beta} :
H_\alpha \setminus \text{interior}\left(\overline{H}_\alpha^\varepsilon \right)
\to
H_\beta \setminus \text{interior}\left(\overline{H}_\beta^{\varepsilon'} \right)\]
to be the union of standard maps near vertices and piecewise linear on $\partial_\varepsilon H_\alpha \to \partial_{\varepsilon'} H_\beta$. In particular, the restriction of $f_{\alpha\beta}$ to each side of $H_\alpha$ is a standard map induced by $g$, which is Lipschitz by Lemma \ref{lem:radcontract} and our bounds from assumption (a).
The map $f_{\alpha\beta}$ is locally Lipschitz on a compact set, hence Lipschitz, and so by Theorem \ref{Thm:KV} extends to a Lipschitz map $H_\alpha \to H_\beta$ of the same Lipschitz constant.

However, this does not complete the proof, as we need a uniform bound on the Lipschitz constant of $f_{\alpha \beta}$ as $\alpha, \beta$ vary in $I$.
To estimate this, let us introduce some more constants.
For each $\alpha \in I$ and $j=1, \ldots, k$, let $\gamma_{\alpha,j}$ denote the $j$-th side of $H_\alpha$, and let $\bar\gamma_{\alpha,j}^\varepsilon$ denote the $\varepsilon$-truncation of $\gamma_{\alpha,j}$.
Set
\[ L_I := \sup_{\alpha,\beta} \max_j \frac
{\ell\left(\bar \gamma_{\beta,j}^{\varepsilon'}\right)}
{\ell\left(\bar \gamma_{\alpha,j}^\varepsilon\right)}, 
\qquad
D_I := \sup_\alpha \text{diam}( H_\alpha).\]
By assumption (a) and our choice of $\varepsilon_I$, both of these quantities are finite.
Similarly, set $W_I$ to be the infimal distance between non-consecutive sides of $H_\alpha$, and define 
\[R_{I}(\varepsilon) := \inf_\alpha R_{\alpha}(\varepsilon) = O(\varepsilon^2),\]
where $R_\alpha(\varepsilon)$ is the injectivity radius of $\overline{H}_\alpha^\varepsilon$. Assumptions (b) and (c) imply both terms are positive (for any choice of $\varepsilon$).

Let $C(\varepsilon) = O(\varepsilon^2)$ denote the length of the arc of the truncating circle $C_\varepsilon$ between the horocycles $h_{\pm 2}$; then for each $p$ in a standard neighborhood of a vertex, it is within $C(\varepsilon)$ of two different sides of $H_\alpha$.
Take $\varepsilon_I$ small enough to ensure $C(\varepsilon) < W_I/2$.
Then for any $p,q \in H_\alpha \setminus \text{interior}(\overline{H}_\alpha^\varepsilon)$ with $d(p,q) \le W_I/2$, one of four things can hold:
\begin{enumerate}
    \item $p$ and $q$ lie on the same edge of $\partial_\varepsilon H_\alpha$.
    \item $p$ and $q$ lie in a standard neighborhood $N_\varepsilon(v_\alpha^j)$ of some vertex.
    \item $p, q \in \partial_\varepsilon H_\alpha$ lie on (truncations of) sides which share a vertex.
    \item $p \in \partial_\varepsilon H_\alpha$ and $q$ is in some standard neighborhood adjoining that side.
\end{enumerate}
We now bound the Lipschitz constant of $f_{\alpha\beta}$.
Given $p \neq q \in H_\alpha \setminus \text{interior}(\overline{H}_\alpha^\varepsilon)$, set
\[ L_{p,q}=\frac{d(f_{\alpha \beta}(p),f_{\alpha\beta}(q))}{d(p,q)}.\]
If $d(p,q) \ge W_I/2$, this is bounded by $2D_I/W_I$. Thus, we can focus on the four cases enumerated above.
\begin{enumerate}
    \item The Lipschitz constant (with respect to arc length) restricted to any side of $\partial_\varepsilon \overline{H}_\alpha$ is at most $L_I$. Since the curvature of any h-cycle is bounded by 1 and $d(p,q) \le W_I/2$, we get a uniform bound on $L_{p,q}$.
    \item Lemma \ref{lem:HLip} immediately furnishes a bound on $L_{p,q}$ for $p,q$ in the same $N_\varepsilon(v_\alpha^j)$. This is uniform by assumption (b).
    \item By definition of $R_I(\varepsilon)$, we have that $d(p,q) \ge R_I(\varepsilon)$, so $L_{p,q} \le D_I/ R_I(\varepsilon)$.
    \item Finally, suppose that $p$ lies in some truncated side $\bar \gamma_\alpha^\varepsilon$ and $q$ is in some standard neighborhood $N_\varepsilon(v)$ adjoining that side.
    Set $q'$ to be the point of $N_\varepsilon(v_\alpha) \cap \gamma$ with the same radial coordinate as $q$; then $d(q,q') \le C(\varepsilon)$.
    Thus, so long as $p$ is sufficiently far ($2C(\varepsilon)$ will do) away from $N_\varepsilon(v_\alpha)$ we get a lower bound on $d(p,q)$, hence an upper bound on $L_{p,q}$.
    
    It remains to consider the case when $p$ is close to $N_{\varepsilon}(v_\alpha)$.
    In this case, $q \in N_{\tilde \varepsilon}(v_\alpha)$ for some $\tilde \varepsilon > \varepsilon$, and likewise $f_{\alpha \beta}(q) \in N_{\tilde \varepsilon'}(v_\beta)$.
    The standard maps $\gamma_\alpha \to \gamma_\beta$ induce a map $\tilde g$ on these slightly larger truncation parameters:
    \[\tilde g: [0, \tilde \varepsilon] \to [0, \tilde \varepsilon'].\]
    Lemma \ref{lem:radcontract} and our bound on $L_I$ together imply that the Lipschitz constant of $\tilde g$ is uniformly bounded.
    Shrinking $\varepsilon_I$ as necessary to ensure that $\tilde \varepsilon = \varepsilon + O(\varepsilon^2)$ and $\tilde \varepsilon'$ are small enough to apply Lemma \ref{lem:HLip}, we get the desired bound on $L_{p,q}$.
\end{enumerate} 
Taking the maximum of the bounds above gives a uniform bound on the Lipschitz constant of $f_{\alpha \beta}$, hence on the Lipschitz constant of the extension $H_\alpha \to H_\beta$.
\end{proof}

  \begin{figure}[htp!]
      \centering
      \includegraphics[width=0.7\linewidth]{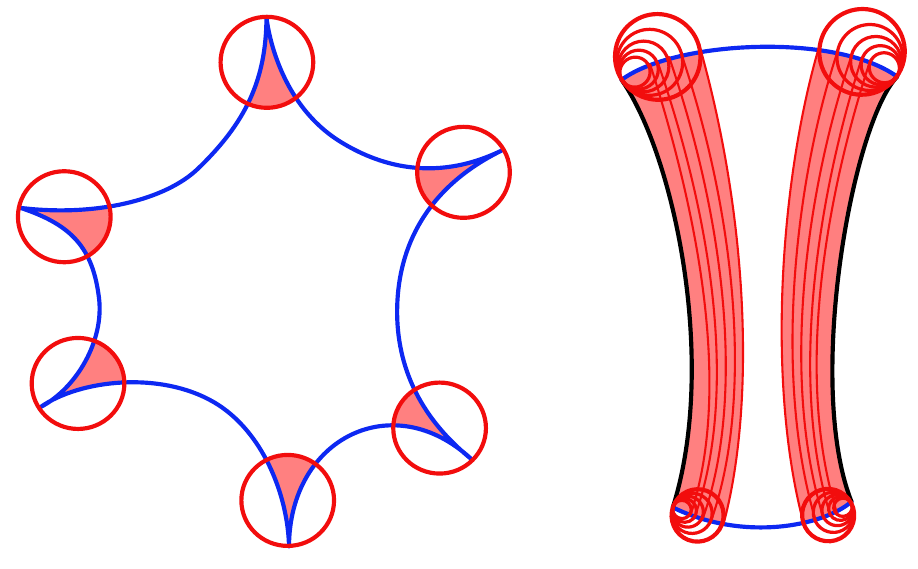}
      \caption{Truncations of an h-gon and of a rectangle. Note that the $\varepsilon$-truncation of one side of a rectangle is compatible with a truncation with different defining parameter.}
      \label{fig:truncate}
  \end{figure}

\begin{remark}
Comparing the bounds in cases (1)--(4), one expects that the final Lipschitz constant is $O(1/\varepsilon^2)$, coming from case (3).
This makes sense because the non-Lipschitz behavior of the affine map on $\partial H_\alpha$ is concentrated at the vertices (Example \ref{example:nonLip}), so shrinking the truncation parameter should make its Lipschitz constant blow up.
\end{remark}

\subsection{Maps on rectangles and spikes}\label{subsec:rectmaps}
We now build Lipschitz maps between the remaining pieces of the decomposition of Lemma \ref{Lem:HgonDecomp}. 

Recall that for us, a rectangle is a right-angled quadrilateral whose sides alternate between geodesics and hypercycles. 
Given a rectangle $R$, there is a vertical foliation of $R$ by geodesic segments of equal length, which we call the \textbf{height} of $R$.
There is also a transverse horizontal foliation of $R$ by hypercycles, with possibly one geodesic leaf. The {\bf width} of $R$ is the length of the shortest such hypercycle.
We will always suppose that each rectangle comes equipped with a choice of ``bottom'' horizontal side and maps preserve this notion. A map $f: R \to R'$ is called \textbf{standard} if $f$ maps vertical leaves to vertical leaves and horizontal leaves to horizontal leaves, and the restriction of $f$ to any vertical leaf has constant slope equal to the ratio of the heights of $R'$ and $R$. 

Let $[0,X]$ be the arc length parametrization of the bottom side of $R$ (oriented so $R$ is to the left of, i.e., above, the bottom side).
We can promote this to a parametrization of all of $R$ as follows.
The rectangle is foliated by hypercycles at a fixed distance from the bottom side of $R$, which all share the same geodesic axis.
We may therefore parametrize $R$ via
\[\Phi_R: [0,X] \times [Y_1,Y_2] \xrightarrow{\sim} R,\]
by insisting that the horizontal line at height $y$ is mapped to the hypercycle $h_y$ at distance $y$ from the geodesic axis.
Note that $y$ is signed, so $y < 0$ ($>0$) corresponds to hypercycles on the right (left) of the geodesic $h_0$, where the orientation of $h_0$ comes from that of the bottom side of $R$.
With this parametrization, we see that vertical lines are mapped to the vertical geodesic segments isometrically.

Define the following map taking the bottom leaf to the leaf at height $y$:
\[\phi_y: h_{Y_1}=\Phi_R(x, Y_1) \mapsto \Phi_R(x,y)=h_y\]
This is the map induced by ``flowing upwards'' along the vertical geodesics for time $y-Y_1$.

\begin{lemma}\label{lem:flowup affine}
For any $y \in [Y_1, Y_2]$, the map $\phi_y$ is affine (with respect to arc length) of slope
$\cosh(y) / \cosh(Y_1)$.
\end{lemma}
\begin{proof}
All horizontal leaves of $R$ are hypercycles which share the same geodesic axis. 
The arc length of the hypercycle distance $y$ away from a geodesic of length $\ell$ is $\ell\cosh(y)$ \cite[Example 1.3.2 (5)]{Buser}, which is in particular linear in $\ell$.
\end{proof}

Now suppose $R'$ is another rectangle (with a choice of bottom side), parametrized by
\[\Phi_{R'}: [0,X'] \times [Y_1', Y_2'] \xrightarrow{\sim} R'.\]
Any homeomorphism $f: [0, X] \to [0, X']$
induces a standard map $R \to R'$ which, in our coordinate systems, simply has the form
\[\Phi_R(x,y) \mapsto \Phi_{R'}\left(f(x), Ly+D\right)\]
where $L = (Y_2' - Y_1')/(Y_2-Y_1)$ is the ratio of heights and $D \in \RR$ is uniquely determined (the actual value is trivial to compute but unimportant).

Given a bound on the quality of $f$ and the geometry of the rectangles, we can bound the quality of the induced map.

\begin{proposition}[Lipschitz maps between rectangles]\label{Prop:RectLip}
For all $K\ge1$ and $A,B,w > 0$, there is an $L_0$ such that for all $L \ge L_0$, the following holds.
Suppose that $R$ and $R'$ are two rectangles parametrized by $\Phi_R$ and $\Phi_{R'}$ as above. Suppose further that:
\begin{itemize}
    \item The horizontal sides of $R$ and of $R'$ all have length in the interval $[A,B]$. 
    \item We are given a $K$-Lipschitz map $f:[0,X] \to [0,X']$.
    \item The width of $R$ is at least $w$.
    \item The height of $R'$ is $L$ times that of $R$, i.e., $Y_2'-Y_1' = L(Y_2-Y_1)$.
\end{itemize}
Then the induced standard map $R \to R'$ is $L$-Lipschitz.
\end{proposition}
\begin{proof}
To establish the proposition, we simply need to bound the partial derivatives of $f$ in the vertical and horizontal directions.
In the vertical direction, $f$ stretches distances uniformly by $L$ by construction.

In the horizontal direction, we observe that the induced map on each horizontal leaf 
\[f_y: h_y \to h_{Ly+D}'\]
can also be described by projecting to $h_{Y_1}$ via the comparison map $\phi_y^{-1}$ in $R$, applying $f$, then pushing back up to the appropriate height using the comparison map $\phi_{Ly+D}'$ in $R'$.
Thus, the total Lipschitz constant of $f_y$ is at most
\begin{equation}\label{eqn:rectbd Y1}
\Lip(f_y) \le 
\Lip(\phi_y^{-1}) \cdot K \cdot \Lip(\phi_{Ly+D}').
\end{equation}
Since we have a bound on the width of $R$, the first term is bounded by $B/w$. 
Lemma \ref{lem:flowup affine} together with our assumption on the lengths of the horizontal sides of $R'$ implies the last term is at most $B/A$.

All told, we see the horizontal direction is stretched by at most $KB^2/wA$ and the vertical is stretched by $L$, so as long as $L_0 \ge KB^2/wA$ the desired statement holds.
\end{proof}

\begin{remark}
Most of the dependencies of Proposition \ref{Prop:RectLip} are necessary, but it is conceivable that $L_0$ is independent of $w$.
Such a statement would require balancing of the growth of $D$ with how quickly the width of $R'$ shrinks.
This computation becomes quite intricate and is not needed for our application of the Proposition, hence we have contented ourselves with the statement above.
\end{remark}

It is even easier to build Lipschitz maps on spikes.
Any spike $V$ is foliated by parallel horocycles, and can be parametrized by its boundary horocycle $h$ and the distance into the spike.

\begin{lemma}[Lipschitz maps between spikes] \label{Lem:CuspLip}
Let $V$ and $V'$ be two spikes bounded by horocycles $h$ and $h'$, respectively.
Suppose $f:h \to h'$ is $K$-Lipschitz (with respect to arc length).
Then for all $L \ge \max\{K,1\}$, $f$ extends to a $L$-Lipschitz map $V \to V'$ with slope $L$ along the boundary geodesics of $V$.
\end{lemma}
\begin{proof}
Identify $V$ and $V'$ in the upper half-plane as
\[ V=\{ x+iy: 0 \le x \le X,  y \ge 1\} \quad \text{and} \quad V' = \{ x+iy : 0 \le x \le X', y \ge 1\},\]
for some constants $X, X' >0$. In these coordinates, $h$ and $h'$ become horizontal segments with imaginary part $1$, and $f$ is just a $K$-Lipschitz map $[0,X] \to [0,X']$.
Extend $f$ to a map $V \to V'$ via
\[x+iy \mapsto f(x)+i y^L.\]
One computes explicitly that $\|Df \| \le L$, ensuring that $f$ is $L$--Lipschitz.
\end{proof}

\subsection{Gluing maps on pieces}\label{subsec:glue maps}
We can now put together the building blocks from Proposition \ref{Prop:HLip}, Proposition \ref{Prop:RectLip}, and Lemma \ref{Lem:CuspLip} to build a variety of boundary-tight maps.

\begin{corollary} \label{cor:Horo to Horo}
Let $P$ and $P'$ be two $n$-sided ideal polygons with cyclically labeled sides, each of which contains an embedded horogon.
Then for all sufficiently large $L$ (depending only on the geometry of the horogons), there is a boundary-tight $L$-Lipschitz map $P \to P'$.
\end{corollary}

Observe that Corollary \ref{cor:Horo to Horo} also yields the sufficiency direction of Theorem \ref{mainthm:polygon}.(1)

\begin{proof}
Let $H \subset P$ and $H' \subset P'$ be embedded horogons and let $f: H \to H'$ be a $K$--Lipschitz map furnished by Proposition \ref{Prop:HLip}. Let $L \ge K$. Since the restriction of $f$ to any side of $H$ is $K$-Lipschitz with respect to arc length,
we can extend $f$ to a $L$-Lipschitz map to any spike of $P\setminus H$. This defines an $L$-Lipschitz map $f: P \to P'$. Since the extension has constant slope $L$ along vertical sides of spikes, $f$ is boundary-tight by construction.
\end{proof}

To glue together maps of h-gons and rectangles, we will need to make sure that the maps on their hypercyclic boundaries match.
To that end, we will first globally specify the truncating parameters used in our constructions.
Let $P$ be an ideal polygon and let $V(P)$ denote the vertices of its spine.
An {\bf h-gon truncation} is a function
\[tr: V(P) \to \RR_{>0}\]
such that for each $v \in V(P)$, the corresponding h-gon $H_v$ is $tr(v)$-truncatable.

Suppose that $v, w \in V(P)$ are adjacent over an edge $e$ of $\Sp(P)$, corresponding to a rectangle $R_e$.
We say that an h-gon truncation is {\bf compatible over $e$} if all vertical leaves through the $tr(v)$-truncation of the corresponding horizontal side of $R_e$ meet the $tr(w)$-truncation of the other side.
Equivalently, the comparison map $\phi_{Y_2}$ of Lemma \ref{lem:flowup affine} takes one truncation to the other.
An h-gon truncation is {\bf globally compatible} if it is compatible over all edges of the spine.

\begin{lemma}[Controlled truncations]\label{lem:exists_truncation}
For any $n \ge 3$, there is a $C_n$ such that the following holds. For any $\varepsilon$ sufficiently small ($\lesssim 1/3C_n$ suffices),
every ideal $n$-gon $P$ admits a globally compatible truncation with all parameters in $(\varepsilon/C_n, \varepsilon C_n)$.
\end{lemma}
\begin{proof}
By Lemmas \ref{lem:Hgon rad bd} and \ref{lem:Hgonbounds}, all boundary h-cycles of each decomposition piece have length in $[1,2r_n]$. 
Thus, by Lemma \ref{lem:flowup affine}, if $v$ and $w$ are adjacent over an edge $e$ of $\Sp(P)$ and $tr$ is compatible over $e$, we see that
\[\frac{1}{2r_n}
\le
\frac{tr(v)}{tr(w)}
\le
2r_n.\]

We now use the fact that $\Sp(P)$ is a tree. Pick a base vertex $v$ and assign $tr(v) = \varepsilon$. For each $w$ adjacent to $v$, we can then choose a compatible truncating parameter which is at most $2r_n$ times as large/small (for example, $tr(w)$ can be specified using the comparison map $\phi_{Y_2}$ coming from the rectangle corresponding to the edge connecting $v$ and $w$).
We can then iterate this procedure, propagating out from each such $w$ at the cost of another factor of $2r_n$.
Since there are at most $n-2$ vertices of $\Sp(P)$, this eventually terminates with a total multiplicative cost of $(2r_n)^{n-2} =: C_n$.
\end{proof}

We can now use this truncation to glue together maps on pieces, proving Theorem \ref{thm:scalespineup}.

\begin{proof}[Proof of Theorem \ref{thm:scalespineup}]
We begin by noting that by combining Lemma \ref{lem:Hgon rad bd} and \ref{lem:Hgonbounds}, we know that the h-gons appearing in the decomposition of any ideal $n$-gon have uniformly bounded corner divergence and injectivity radius, so we can apply Proposition \ref{Prop:HLip} with uniform quality.
The same Lemmas give us uniform bounds on the length of h-cycles bounding the rectangles of our decomposition, and we have a bound on their widths by hypothesis. This allows us to apply Proposition \ref{Prop:RectLip} with dependencies only on $n$ and $K$.

The spines of $P$ and $P_L$ have the same combinatorial structure, so throughout the proof we identify their vertex sets. 
Fix $\varepsilon$ sufficiently small such that $\varepsilon C_n$ is smaller than the constant $\varepsilon_I$ from Proposition \ref{Prop:HLip}, where $I$ is the family of h-gons inscribable in ideal $n$-gons.
Using Lemma \ref{lem:exists_truncation}, pick globally compatible h-gon truncations $tr$ and $tr_L$ for $P$ and $P_L$, respectively, with all parameters in $(\varepsilon/C_n, \varepsilon C_n)$.

Pick an arbitrary $v \in V(P)$. As in Lemma \ref{lem:exists_truncation}, we will build our map $P \to P_L$ by propagating out from $v$.
Consider the constant linear map
\[g_v: [0, tr(v)] \to [0, tr_L(v)].\]
By our choice of truncating parameters, $g_v$ has slope at most $C_n^2$.
By Proposition \ref{Prop:HLip}, $g_v$ induces a $K_0$-Lipschitz map between the h-gons $H_v$ and $H_{v,L}$ corresponding to $v$ which restricts to a standard map on each boundary h-cycle.
Note that the Lipschitz constant $K_0$ depends only on $n$ by the uniformity of Proposition \ref{Prop:HLip}.

Now consider a compact edge $e$ of $\Sp(P)$ adjacent to $v$ and the corresponding rectangles $R_e \subset P$ and $R_{e,L} \subset P_L$.
Thinking of the hypercycle common to $H_v$ and $R_e$ as the ``bottom'' edge of $R_e$, Proposition \ref{Prop:RectLip} ensures that the standard map induced by $g_v$ on that hypercycle extends to an 
$L$-Lipschitz standard map $R_e \to R_{e,L}$ so long as $L \ge L_0$, where $L_0$ depends only on $n$ and $K_0$.\footnote{To be very precise, $L_0$ depends on the Lipschitz constant with respect to arc length of the standard map induced by $g_v$, which in turn depends on $K_0$.}

Let $w$ denote the other endpoint of $e$. The restriction of the induced standard map $R_e \to R_{e,L}$ to the top edge is a standard map on that hypercycle with Lipschitz constant at most $2r_n$ times larger than on the bottom edge (Lemma \ref{lem:flowup affine}).
In particular, restricting to a neighborhood of the endpoints of this hypercycle induces a $2r_n K_0 M$-Lipschitz map
\[g_w: [0, tr(w)] \to [0, tr_L(w)]\]
between radial parametrizations, as described after Lemma \ref{lem:radcontract}.
We now repeat the argument from before, using Proposition \ref{Prop:HLip} to extend this to a $K_1$-Lipschitz map between the relevant horogons $H_w \subset P$ and $H_{w,L} \subset P_L$.
The Lipschitz constant $K_1$ now depends only on $n$ and $K_0$.
The restrictions of this map to the other h-cycles of $H_w$ are the standard maps induced by $g_w$, and so each extends to an $L$-Lipschitz map on each of the adjoining rectangles and spikes so long as $L \ge L_1$ for some $L_1$ depending on $n$ and $K_1$.

We continue in this fashion, exhausting $\Sp(P)$ by combinatorial distance to $v$, alternating between building maps on h-gons and on rectangles. 
All told, this results in 
\begin{itemize}
    \item A set of maps between the h-gons of $P$ and $P_L$, each of whose Lipschitz constants $K_i$ depends only on $n$ and $K_j$ for $j < i$, and 
    \item A set of maps between the rectangles of $P$ and $P_L$ which are each $L$-Lipschitz, so long as $L \ge L_i$, where $L_i$ depends only on $n$ and $K_j$ for $j \le i$.
\end{itemize}
These maps agree on their shared hypercycles.
Collapsing all of the dependencies, we see that the Lipschitz constants $K_i$ and the the thresholds $L_i$ all depend only on $n$.
Thus, if we take $L$ greater than the max of all of these and use Lemma \ref{Lem:CuspLip} to extend the maps between horocycles to $L$-Lipschitz maps between spikes, all of these pieces glue together to give a globally $L$-Lipschitz map $P \to P_L$ which stretches $\partial P$ affinely by construction.
\end{proof}

Reversing which direction we are scaling the spine, we also get the following:

\begin{corollary}\label{cor:scalespinedown}
For every ideal $n$-gon $P$, there is an $L_1(P)$ such that for any $L \ge L_1$, there is an $L$-Lipschitz, boundary-tight map $P_{1/L} \to P$.
\end{corollary}
\begin{proof}
As $L \to \infty$, the polygons $P_{1/L}$ converge to the regular ideal $n$-gon. In particular, the waists of $P_{1/L}$ converge to some positive number. We may therefore apply Theorem \ref{thm:scalespineup} with a uniform threshold to the polygons $P_{1/L}$ once they are sufficiently close to regular.
\end{proof}

This allows us to deduce the sufficiency part of Theorem \ref{mainthm:polygon}.(2).

\begin{proof}[Proof of sufficiency in Theorem \ref{mainthm:polygon}.(2)]
Let $P$ be any ideal $n$-gon, with residue 0 if $n$ is even, and consider the polygon $P_{1/L}$.
For any $L$ sufficiently large, this is arbitrarily close to the regular $n$-gon $P_{\reg}$, hence contains an embedded horogon.
Corollary \ref{cor:Horo to Horo} builds a boundary-tight map $P_{\reg} \to P_{1/L}$, and Corollary \ref{cor:scalespinedown} builds a boundary-tight map $P_{1/L} \to P$.
\end{proof}

\section{Prescribed itineraries}\label{sec:itin}
We now prove a more detailed version of Theorem \ref{mainthm:itin} (Theorem \ref{thm:itineraries}). As was the case for our analysis of deflations, we will need to coarsen our statements, but we can still show that the obstructions found above essentially account for all possible constraints.

In Section \ref{subsec:glue_reg}, we describe a formalism for gluing together ideal polygons to get a closed surface. We then describe the geometry of the constructed surfaces in Section \ref{subsec:shearvsspine}; this section also contains a number of lemmas we need in the sequel. Theorem \ref{thm:itineraries} is proven in Section \ref{subsec:itin} by gluing together the maps from the previous Section \ref{sec:construct horogon}.

\subsection{Gluing and spinning}\label{subsec:glue_reg}
We first specify a particular method for constructing surfaces by gluing together regular polygons.
It is particularly easy to analyze the geometry of surfaces built from this construction, a fact we will regularly exploit in the sequel.

Let us first clarify what we mean by ``gluing.''
Suppose we are given two oriented surfaces $Y_1$ and $Y_2$ with crowned boundary. Pick two bi-infinite geodesics $g_i \subset \partial Y_i$ together with points $p_i \in g_i$, and equip the $g_i$ with their induced boundary orientations.\footnote{This notion of boundary orientation is the standard one (so $Y_i$ is to the left of $g_i$), not a choice of orientation on an even-spiked crown. We apologize for the overlap in terminology.}
There is a 1-parameter family of orientation-reversing isometric identifications of $g_1$ with $g_2$, parametrized by the signed distance from $p_1$ to $p_2$, which is called the {\bf shear} (measured with respect to the $p_i$). Here, we choose the convention that the signed distance is measured with the orientation of $g_2$, so that that positive shear means that if you start in $Y_1$, travel to $p_1$, then travel along $g_1 = g_2$ to $p_2$, then you have turned rightwards (see \cite[Remark 13.3]{shshI}).

The {\bf gluing} of $Y_1$ and $Y_2$ with a given shear $s$ is the surface obtained by quotienting by the specified identification of $g_1$ with $g_2$.
One can also glue together two bi-infinite geodesics of the same surface in much the same way; in this case, the resulting surface may not be complete (for example, if the glued leaf spirals around a closed boundary component), so we take its completion.
This notion of gluing clearly depends also on our choice of points $p_i$ (with respect to which we measure shears), but moving either $p_i$ just changes the shear by the (signed) distance the point has moved.

When a surface $Y$ is built out of iterated gluings, it comes with a finite set of disjoint, bi-infinite geodesics corresponding to the gluing locus.
Restricting to a specific class of gluings allows us to keep track of this data using tools from earlier in the paper.

Given an arc $\alpha$ on a surface with (possibly crowned) boundary $\Sigma$, not isotopic into a spike, we can {\bf spin} $\alpha$ by turning left from $\alpha$ when it encounters $\partial \Sigma$ and then continuing along $\partial \Sigma$.
Pulling tight yields a simple, bi-infinite leaf $\ell_\alpha$ each of whose ends either goes out a cusp of $\Sigma$ or spirals onto a closed component of $\partial\Sigma$.
We observe that if $\arc$ is a system of disjoint arcs, then the leaves obtained this way are also all disjoint.
Compare with \cite[Proof of Proposition 6.3]{Th_stretch} or the discussion in \cite[\S3.1]{PW_envelopes}.

\begin{construction}[Gluing along spun arcs]\label{constr:gluereg}
Let $\Sigma$ be a surface with (possibly crowned) boundary and fix a weighted, filling arc system 
\[\arcwt = \sum c_\alpha \alpha \in |\mathscr{A}_{\text{fill}}(\Sigma, \partial \Sigma)|_{\RR}.\] 
Spin the underlying arcs $\arc$ of $\arcwt$ to obtain a collection of disjoint bi-infinite leaves $\ell$ on $\Sigma$; observe that $\Sigma \setminus \ell$ is a union of ideal polygons, each of which corresponds to a complementary component of $\Sigma \setminus \arc$.
We now define a hyperbolic structure $Y(\arcwt)$ on $\Sigma$ by stipulating that each of these polygons is regular and gluing them together with shear $c_\alpha$ along the leaf $\ell_\alpha$, where shears are measured with respect to the inscribed regular horogons.
See Figure \ref{fig:gluereg}.
\end{construction}

It is important to note that the surface $Y(\arcwt)$ is usually {\em not} the same as the surface whose orthogeodesic spine is dual to $\arcwt$ (Theorem \ref{thm:spinecoords}), else Propositions \ref{prop:reg_glue_res} and \ref{prop:glue_reg_deflate} below would be trivial.
Along similar lines, we observe that $Y(\arcwt)$ varies continuously as $\arcwt$ varies among weighted arc systems with the same support, but the construction is {\em not} continuous as $\arcwt$ varies throughout all of $|\mathscr{A}_{\text{fill}}(\Sigma, \partial \Sigma)|_{\RR}$.

\begin{figure}[ht]
    \centering
    \includegraphics[width=\linewidth]{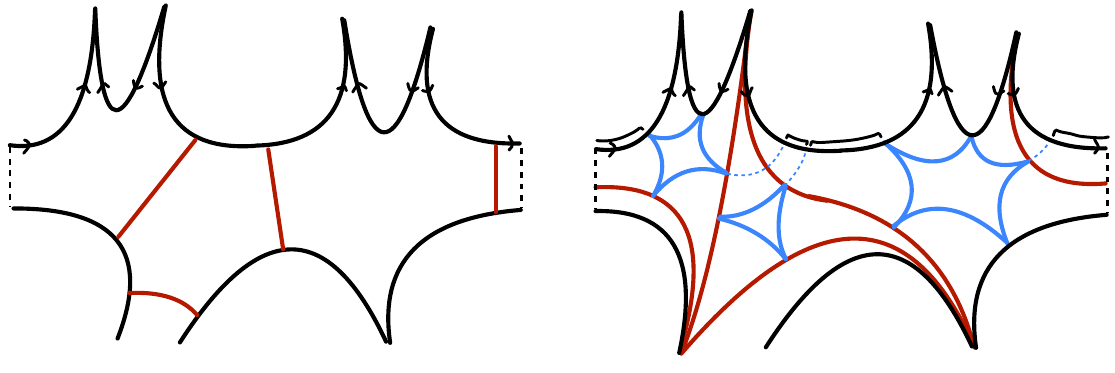}
    \caption{Building a surface out of regular polygons. On the left, a weighted arc system. On the right, the result of Construction \ref{constr:gluereg}.}
    \label{fig:gluereg}.
\end{figure}

\begin{remark}
There is further flexibility in Construction \ref{constr:gluereg}: one could spin arcs to the right (or both left and right, though one must take care to ensure simplicity), and one could also glue the polygons together with negative weights.
Versions of the geometric statements we prove below should also still hold, but one must take care in their formulation.
\end{remark}

One nice feature of Construction \ref{constr:gluereg} is that it allows us to easily compute metric residues in terms of shearing data.
Given $\arcwt \in |\mathscr{A}_{\text{fill}}(\Sigma, \partial \Sigma)|_{\RR}$, there is a combinatorial notion of residue defined as follows. 
Let $\vec C$ be either an oriented crown or closed boundary component of $\Sigma$. Then $\res_{\arcwt}(\vec C)$ is the sum of the weights of the arcs incident to edges of $\vec C$ that see $\Sigma$ on its left, minus the weights of the arcs incident to edges of $\vec C$ that see $\Sigma$ on its right (taken with multiplicity).
See \cite[\S6]{shshI}.

\begin{proposition}\label{prop:reg_glue_res}
Let $\Sigma$ be a surface with boundary and let $Y \in \cT(\Sigma)$ be the result of applying Construction \ref{constr:gluereg} to $\arcwt \in |\mathscr{A}_{\text{fill}}(\Sigma, \partial \Sigma)|_{\RR}$.
Then for any oriented crown or closed boundary component $\vec C$ of $\partial \Sigma$,
\[\res_{Y}(\vec C) = \res_{\arcwt}(\vec C).
\]
\end{proposition}

As an example, if an even-gon $Q$ is glued out of regular odd-gons $Q_1$ and $Q_2$, then $\res(Q)$ is equal, up to sign, to twice the shear between the pieces.\footnote{If an even-gon is glued out of two even-gons, then their residues add.}

\begin{proof}
The proof is essentially contained in Figures \ref{fig:gluereg} and \ref{fig:gluereg_res_closed}. To be more explicit, we describe the relevant features.

{\medskip \noindent \em Case 1: $\vec C$ is a crown.}
For each spike $V$ of $\vec C$, consider the component of $\Sigma \setminus \arc$ containing $V$. This component corresponds to one of the regular polygons $P_V$ of $Y(\arcwt) \setminus \ell$, and $V$ corresponds to a spike of $P_V$.
Observe that since $\ell$ came from spinning the arcs of $\arc$, one of the sides of $P_V$ is one of the geodesics of $V$.\footnote{If one stands in $Y$ and looks out the spike of $V$, it is the ``left'' geodesic.}
The horocyclic segment of the regular horogon in $P_V$ cutting off that spike extends to a horocyclic segment $h_V$ in $Y$ connecting the geodesics of $V$. 

Moreover, the horocyclic foliations of the polygons of $Y \setminus \ell$ glue to a global (measured) foliation $\mathcal F$ of $Y$, and we claim $h_V$ is a leaf. 
This follows because the only other polygons $Q$ of $Y \setminus \ell$ meeting $h_V$ are those which share an ideal vertex with $V$. We are always gluing polygons with positive shears, so $h_V \cap Q$ is closer to $V$ than the regular horogon of $Q$, hence is a horocyclic segment.
These segments piece together to give $h_V$.

We now use this choice of truncation to compute the metric residue of $\vec C$. For each geodesic $g$ of $\vec C$ connecting spikes $V_1$ and $V_2$, projection along the leaves of $\mathcal F$ shows that the distance between the horocycles $h_{V_1}$ and $h_{V_2}$ is exactly the sum of shears over the geodesics of $\ell$ separating $P_{V_1}$ and $P_{V_2}$.
By construction, this is the same as the total weight of the corresponding arcs of $\arc$.
Adding up with the appropriate signs yields the desired formula.

{\medskip \noindent \em Case 2: $\vec C$ is a closed geodesic.}
Now we must work in the universal cover and use $\pi_1$-equivariance. Pick a geodesic $g$ of $\widetilde Y$ covering $C$ and pick some (lift of a) polygon $P$ with a spike asymptotic to $g$. Equivalently, $P$ corresponds to some component of $\Sigma \setminus \arc$ abutting $g$.
Consider the spike of $P$ asymptotic to $g$. 
By the same argument as in the first case, the segment of the regular horogon in $P$ cutting off that spike extends on one side\footnote{Its ``right,'' if one stands in $P$ and looks out the spike.} to a horocyclic segment $h$ meeting $g$.
By the same positivity argument as above, $h$ is a partial leaf of the regular horocyclic foliation $\mathcal F$ of $Y$ with respect to $\partial Y \cup \ell$.
We note that if one tries to extend $h$ on the other side, then the horocycle defining $h$ does not match the foliation $\mathcal F$.

\begin{figure}[ht]
    \centering
    \includegraphics[width=\linewidth]{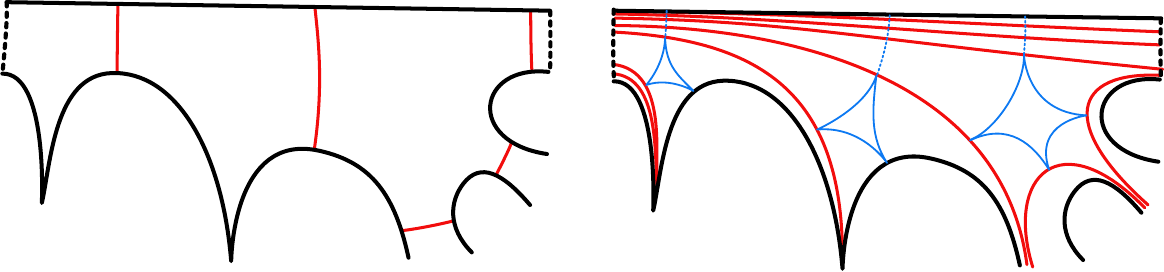}
    \caption{A filling arc system and the output of Construction \ref{constr:gluereg}. The top geodesic is the lift of a closed boundary component of $Y$.}
    \label{fig:gluereg_res_closed}
\end{figure}

The length of $C$ in $Y$ is now the distance (along $g$) between $h \cap g$ and $\gamma h \cap g$, where $\gamma$ is the stabilizer of $g$ in $\pi_1(Y)$.
By projecting along the leaves of $\mathcal F$, this is the same as the the sum of shears over geodesics separating $P$ and $\gamma P$, which by construction is the total weight of arcs incident to $C$.
\end{proof}

\subsection{Shears versus spines}\label{subsec:shearvsspine}
When one spins an arc system as described above, there is some extra topological control on the resulting leaves, a feature we have already exploited in the previous subsection. Here is another version:

\begin{lemma}[Spun arcs separation]\label{lem:spin_separation}
Let $\smash{\widetilde \Sigma}$ be a simply-connected surface with boundary and fix some system of disjoint arcs $\arc$ on $\smash{\widetilde \Sigma}$.
Spin $\arc$ to yield a set of bi-infinite leaves $\ell$.
Then for every arc $\alpha \in \arc$, the following holds.
Let $g_1, g_2 \subset \partial \smash{\widetilde \Sigma}$ denote the boundary components connected by $\alpha$.
Then the leaves of $\ell$ which separate $g_1$ from $g_2$ 
exactly correspond to those arcs with an endpoint on either $g_1$ or $g_2$.
\end{lemma}

In particular, the leaves separating $g_1$ from $g_2$ break into two packets, consisting of those asymptotic to $g_1$ and those asymptotic to $g_2$. These have only the leaf $\ell(\alpha)$ in common.

\begin{proof}
Since $\ell_\alpha$ is asymptotic to $g_1$, any leaf separating them must be asymptotic to both. The same holds for $\ell_\alpha$ and $g_2$.
But now we note that we can connect $g_1$ to $g_2$ by crossing from $g_1$ to $\ell_\alpha$, running along $\ell_\alpha$, then crossing to $g_2$.
Thus any leaf of $\ell$ separating $g_1$ from $g_2$ must be asymptotic to $\ell_\alpha$, which means it must be asymptotic to either $g_1$ or $g_2$, which means it must come from spinning an arc with an endpoint on $g_1$ or $g_2$.
\end{proof}

The above Lemma allows us to understand how to glue together deflation maps coming from the regular pieces.

\begin{proposition}\label{prop:glue_reg_deflate}
Let $\Sigma$ be a surface with crowned boundary, fix $\arcwt \in |\mathscr{A}_{\text{fill}}(\Sigma, \partial \Sigma)|_{\RR}$, and let $T(\arcwt) \in \TRG(\Sigma)$ be its dual metric ribbon graph.
Then there exists a deflation
\[Y(\arcwt) \to T(\arcwt)\]
where $Y(\arcwt)\in \cT(\Sigma)$ is the output of Construction \ref{constr:gluereg} applied to $\arcwt$.\end{proposition}
\begin{proof}
The deflation is defined by collapsing the leaves of the regular horocyclic foliation $\mathcal F$ on each polygon.
This results in a deflation $\dfl:Y \to T$ to {\em some} ribbon graph, and it remains to show that $T = T(\arcwt)$.
We will do this at the level of the universal cover ($\pi_1$ equivariance will follow because our construction will not require any choices).

Collapsing the regular horocyclic foliation on each $n$-gon of $\smash{\widetilde Y} \setminus \smash{\widetilde \ell}$ gives a map to the $n$-vertex star $\ast_n$.
If $P_1$ and $P_2$ are adjacent over a leaf $\ell_{12}$ of $\smash{\widetilde \ell}$, then the maps $f_1:P_1 \to \ast_{n_1}$ and $f_2: P_2 \to \ast_{n_2}$ can be glued together into a map
\[P_1 \cup P_2 \to (\ast_{n_1} \cup \ast_{n_2}) / (f_1(\ell_{12}) \sim f_2(\ell_{12})).\]
Since shears are measured with respect to the ideal horogons, which are sent to vertices under the deflations $f_i$, the target is a graph with two vertices, one of valence $n_1$ and the other of valence $n_2$, connected by an edge $e$ whose length is exactly the shear used to glue $P_1$ and $P_2$. See Figure \ref{fig:glue_deflate}.

\begin{figure}[ht]
    \centering
    \includegraphics[width=.7\linewidth]{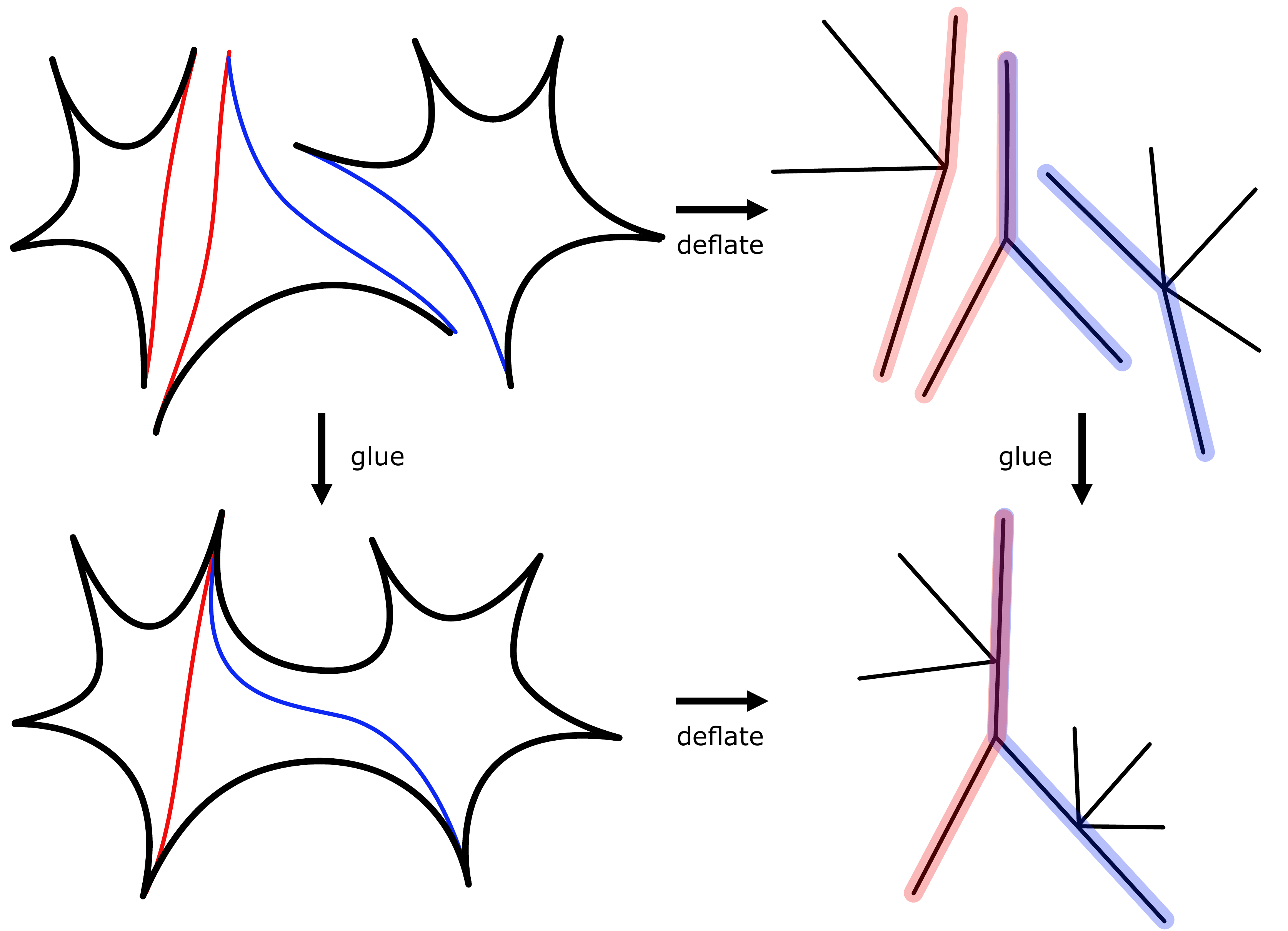}
    \caption{Gluing deflations of polygons.}
    \label{fig:glue_deflate}
\end{figure}

Globalizing, this means that (universal cover of) the metric ribbon graph $T$ resulting from collapsing the leaves of $\mathcal F$ may be obtained as 
\[\widetilde{T} = \bigsqcup_{P} \ast_{n(P)} / \sim, \]
where the index runs over all polygons $P$ of $\smash{\widetilde{Y}} \setminus \smash{\widetilde{\ell}}$, the number $n(P)$ is the number of sides of $P$, and two geodesic paths in different stars are equivalent if they are the image of a leaf of $\smash{\widetilde{\ell}}$ under two different deflation maps.

With this description, it is now easy to see that $T = T(\arcwt)$.
Indeed, suppose that $P_1$ and $P_2$ are polygons adjacent over a leaf $\ell$ as before.
We wish to understand all other points of $\bigsqcup_{P} \ast_{n(P)}$ identified with the (open) edge $e$. 
By Lemma \ref{lem:spin_separation}, together with the fact that all polygons are glued with positive shears, we see that $\dfl^{-1}(e)$ consists entirely of leaves of $\mathcal F$ isotopic to $\arc$ (compare the argument in Proposition \ref{prop:reg_glue_res}).
In particular, for any other polygon $Q$, the intersection $\dfl^{-1}(e) \cap Q$ is a union of horocycles based at a single spike of $Q$. 
Thus, we see that no other vertices are mapped to $e$ even after making all identifications, hence it persists in the final graph.
\end{proof}

In the following subsection, we will want to compare the spine of $Y(\arcwt)$ with the dual ribbon graph $T(\arcwt)$.
More generally, we will want a similar result for surfaces glued out of polygons which are almost, but not exactly, regular.

Given any hyperbolic surface $Y$ with crowned boundary, define its {\bf spine basepoints} to be the projection of the vertices of $\Sp(Y)$ to $\partial Y$; we have already had to consider these in Section \ref{subsec:generalspinedeflate}, and will need to refer to them repeatedly in what is to come.
Say that an ideal polygon $P$ is {\bf $\zeta$-nearly-regular}\footnote{Compare with the definition of ``$\zeta$-equidistant'' in \cite[Definition 6.5]{shshII}.}
if all of the finite edges of its spine have length at most $\zeta$.
In particular, there is some constant $C$ depending only on the number of sides of $P$ such that, for any geodesic side $g$ of $P$, the spine basepoints on $g$ are all within $C\zeta$ of each other.
Note that for any $\zeta$, the family of $\zeta$-nearly-regular polygons is compact.

The key estimates we will need in the sequel are the following:

\begin{lemma}[Big shears short arcs]\label{lem:bigshears_closegeos}
For every $\zeta, \varepsilon > 0$, there is an $S_0$ ($=O(\log(1/\varepsilon))$)
such that the following holds.
Let $\Sigma$ be any finite-type surface with (crowned) boundary and $\arc$ any filling arc system. Spin $\arc$ to get a union of leaves $\ell$.
Let $Y \in \T(\Sigma)$ be any surface obtained by gluing together $\zeta$-nearly-regular polygons $\{P_i\}$ along $\ell$ with shears at least $2S_0$ (as measured with respect to any spine basepoints).
Then:
\begin{enumerate}
    \item The orthogeodesic representatives of the arcs of $\arc$ each have length $O(\varepsilon)$.
    \item Each spine basepoint of $Y$ is within $O(\zeta + \varepsilon)$ of (the image of) a spine basepoint of one of the $P_i$, and vice versa.
\end{enumerate}
The implicit constants depend only on $\zeta$, the maximum number of arcs incident to any boundary geodesic of $Y$, the maximum number of the sides of polygons of $Y \setminus \ell$, and the minimum length of any closed boundary component of $\partial Y$.
\end{lemma}

Statement (1) will allow us to apply our results on deflation maps to these glued surfaces, while (2) will allow us to iteratively apply certain arguments.

\begin{proof}
Take the universal cover $\widetilde Y$ of $Y$, pick a lift of $\alpha$ (which we will give the same name), and let $g_1, g_2 \subset \partial Y$ denote the boundary components connected by $\alpha$. Let $P_1$ and $P_2$ denote the (lifts of) the polygons of $Y \setminus \ell$ adjacent over $\ell_\alpha$.
For concreteness, let us suppose that $P_1$ and $g_1$ are on the same side of $\ell_\alpha$, and similarly for $P_2$ and $g_2$.

\begin{claim}\label{claim:bigshear}
There is some $S_0 = O(\log(1/\varepsilon))$ such that if all shears are positive and the distance of $p \in \ell_\alpha$ to the spine basepoints of $P_i$ is at least $S_0$, then $p$ is within $O(\varepsilon)$ of $g_i$.
\end{claim}

Assuming the Claim, let us prove the Lemma. For (1), we see that if we glue $P_1$ and $P_2$ with shears larger than $2S_0$,
then there is a point $p$ of $\ell_\alpha$ which is $\ge S_0$ away from the spine basepoints of both $P_1$ and $P_2$, hence, $p$ is $O(\varepsilon)$ close to both $g_1$ and $g_2$.
Since the length of the orthogeodesic representative of $\alpha$ is the minimum distance between points of $g_1$ and $g_2$, we see it must also be of order $O(\varepsilon)$.

For (2), we observe that our bound on shears implies that if the distance between $p$ and any spine basepoint of $P_1$ is at most $S_0$, then $p$ must be at least $S_0$ away from the spine basepoints of $P_2$, hence is $O(\varepsilon)$ close to $g_2$. The same holds for all other polygons adjacent to $P_1$.
Taken together, this implies that for any vertex $v$ of the spine of $P_1$, the geodesic tuple $\partial P_1$ is $\varepsilon$-Hausdorff close to a tuple of geodesics of $\partial Y$ on a ball of radius $O(\log(1/\zeta))$ around $v$.
Moreover, the implicit constant depends only on the inradius of $P_1$ (which in turn depends only on $\zeta$ and the number of sides of $P_1$).
Lemma 6.2 and Corollary 6.3 of \cite{shshII} now imply that the center of any circle tangent to three geodesics of $\partial P_1$ is $O(\varepsilon)$ close to the center of the circle tangent to the corresponding geodesics of $\partial Y$, hence the collection of spine basepoints on these tuples must have diameter $O(\zeta + \varepsilon)$.
\end{proof}

\begin{proof}[Proof of Claim \ref{claim:bigshear}]
We begin by observing that an analogous statement is obvious in the non-glued setting: if one goes $\log(1/\varepsilon)$ far out from the spine basepoints in any $\zeta$-nearly-regular polygon, then any horocycle based at any spike has length at most $O(\varepsilon)$. This follows from compactness of the space of $\zeta$-nearly-regular polygons.

Consider the set of geodesics $\Lambda$ of $\ell$ separating $p$ from $g_1$; by Lemma \ref{lem:spin_separation}, these come from spinning lifts of arcs of $\arc$ incident to $g_1$. Consider the horocycle $h$ based at the common endpoint of $\ell_\alpha$ and $g_1$.
Since all shears are positive, we see that for each polygon $Q$ of $\widetilde Y \setminus \widetilde \ell$ meeting $h$, the distance from $h$ to the spine basepoints of $P_1$ is greater than the distance to the spine basepoints of $Q$ (up to an error of $O(\zeta)$).
In particular, so long as $p$ is taken far enough away from the basepoints of $P_1$, then $h \cap P_1$ and $h \cap Q$ both have length at most $\varepsilon$.

In the case where $g_1$ projects down to a geodesic in a crown, we are done. The geodesic $g_1$ is not stabilized by any element of $\pi_1 Y$, so there are finitely many leaves we need to consider separating $p$ from $g_1$. Adding up the length of the horocycles in the complementary regions yields a multiple of $\varepsilon$.

Otherwise, $g_1$ is the lift of a closed boundary component and there are infinitely many leaves in $\Lambda$, but only finitely many up to the action of the deck group. Let $\gamma$ be a generator of the stabilizer of $g_1$; then by our positivity assumption, we know the shear from $P_1$ to $\gamma P_1$ is positive.
In particular, this means that $h \cap \gamma P_1$ is deeper in the spike than $h \cap P_1$.
Lengths of horocycles decrease exponentially as one goes deeper into a cusp, so the total length of $h$ is a sum of geometric series, each of whose first terms is on the order of $\varepsilon$. Thus we still get the desired bound.

We note that the exponential decay rate of these series in the closed case is uniformly fast, as it is approximately $e^s$, where $s$ is the length of the boundary component covered by $g_1$. Since we are gluing with large positive shears, Proposition \ref{prop:reg_glue_res} ensures that $s$ is large.
\end{proof}

\subsection{All itineraries occur}\label{subsec:itin}
With these geometric preliminaries established, we can now prove that, subject to the constraints of Theorems \ref{thm:imageofcut} and \ref{thm:lifecycle}, every possible sequence of arc systems is realized by boundary-tight maps.

Let $\lambda$ be a lamination on $S$ admitting a measure of full support. Say that an unweighted arc system $\arc = \bigcup \alpha$ on $\Sigma = S \setminus \lambda$ is {\bf $(S \setminus \lambda)$-compatible} if there exists a system of nonzero weights $(c_\alpha)$ such that 
\[\sum_{\alpha \in \arc} c_\alpha \alpha \in \Base,\]
where $\Base$ is the set of weighted, filling arc systems satisfying the gluing condition of Theorem \ref{thm:imageofcut}, with residues replaced with the combinatorial notion $\res_{\arcwt}$.
Equivalently (by Theorem \ref{thm:imageofcut}), there is some $X \in \T(S)$ such that $\arc(X \setminus \lambda) = \arc.$ 

The following is a more detailed version of Theorem \ref{mainthm:itin}.

\begin{theorem}\label{thm:itineraries}
Let $\Sigma$ be a finite-type surface with crowned boundary. For any sequence of nested, filling arc systems
\[\arc_0 \subset \arc_1 \subset \ldots \subset \arc_n,\]
there exist $Y_0, Y_1, \ldots, Y_n \in \cT(\Sigma)$ together with boundary-tight maps $Y_i \to Y_{i+1}$ such that:
\begin{enumerate}
    \item $\arc_i \subset \arc(Y_i)$.
    \item All arcs of $\arc(Y_i) \setminus \arc_i$ have uniformly bounded weight.
\end{enumerate}
Moreover, if $\lambda$ is a geodesic lamination on a closed surface $S$ that admits a measure of full support and $\Sigma = S \setminus \lambda$, and if the $\arc_i$ are all $(S \setminus \lambda)$-compatible, then the $Y_i$ can all be taken in $\cT(S \setminus \lambda)$.
In this case, there exist $X_i \in \cT(S)$ such that:
\begin{enumerate}
    \item[(3)] $X_i \setminus \lambda = Y_i$.
    \item[(4)] There are tight maps $X_i \to X_{i+1}$ that affinely stretch $\lambda$ by the Lipschitz constant.
\end{enumerate}
\end{theorem}

In fact, we will prove much more, specifying not only the combinatorics of the arc system of $Y_i$ but also their weights.
To set some language for this, we recall a notion from the Introduction.
Given two weighted arc systems (filling or otherwise) $\arcwt$ and $\arcwtb$ on $\Sigma$, we say 
\[\arcwt = \arcwtb + O(1)\]
if there is a constant $C$ (depending only on the topological type of $\Sigma$) such that
\begin{enumerate}
    \item All arcs of $\arcwt$ that are not in $\arcwtb$ have weight at most $C$, and vice versa.
    \item The weights of each arc they have in common differs by at most $C$.
\end{enumerate}

\begin{proof}
We first fix a number of constants that will be relevant throughout the proof. Set $\varepsilon_0$ to be the cutoff from Lemma \ref{lem:thickthin_deflation}.
Shrink this to some $\varepsilon_1$ such that weights of arcs of length at most $\varepsilon_1$ are always at least $C_0$, the cutoff from Corollary \ref{cor:thin bands long edges} (compare the discussion after the Corollary).
Fix some small $\zeta$, let $S_0$ be the shear threshold from Lemma \ref{lem:bigshears_closegeos}, and increase the threshold to some $S_1 > S_0 + O(\zeta + \varepsilon_1)$, where the implicit constant depends on that from the Lemma and the topological type of $\Sigma$.
In particular, taking $\varepsilon_1$ and $\zeta$ sufficiently small, we may as well set $S_1 = S_0+1$.

Choose weight systems on the $\arc_i$
\[\arcwt_0, \arcwt_1, \ldots, \arcwt_n \in |\mathscr{A}_{\text{fill}}(\Sigma, \partial \Sigma)|_{\RR}\]
whose coefficients are larger than $S_1$ but uniformly bounded, say by $2S_1$.
If the $\arc_i$ are $(S \setminus \lambda)$-compatible, choose $\arcwt_i \in \Base$.

The idea of our proof is to construct the $Y_i$ inductively by gluing together (almost)-regular ideal polygons along geodesics coming from the arcs of $\arc_i$ with shears coming from the coefficients of $\arcwt_i$ (see \eqref{eqn:identarc} below). The boundary-tight maps $Y_i \to Y_{i+1}$ will be constructed by gluing together maps guaranteed by Corollary \ref{cor:scalespinedown}.
\bigskip

Let us first prove the theorem in the case when $n=1$: this will demonstrate all of the main ideas but with drastically simplified bookkeeping.
Spin $\arc_0$ and $\arc_1$ (as described before Construction \ref{constr:gluereg}) to get two collections of bi-infinite leaves $\ell_0 \subset \ell_1$ that cut $\Sigma$ into ideal polygons.
We fix all polygons of $\Sigma \setminus \ell_1$ to be regular, and then define a hyperbolic structure $Q$ on $\Sigma \setminus \ell_0$ by gluing together regular pieces along the leaves of $\ell_1 \setminus \ell_0$.
For each arc $\alpha \in \arc_1 \setminus \arc_0$, we glue along $\ell_\alpha$ with shear $c_\alpha$, measured with respect to the inscribed regular horogons.
Each side of $Q$ comes from a single polygon of $\Sigma \setminus \ell_1$, so the regular horogons in pieces pick out a family of basepoints $q^\ast \subset \partial Q$, one on each side.

By Corollary \ref{cor:scalespinedown}, for any $\zeta>0$ and any $L$ sufficiently large there is a union of $\zeta$-nearly-regular polygons $P$ with $\Sp(P) = 1/L \cdot \Sp(Q)$ together with an $L$-Lipschitz, boundary-tight map $f:P \to Q$ that takes spine basepoints to spine basepoints.
Let $p^\ast = f^{-1}(q^\ast) \in \partial P$. Glue together the nearly-regular polygons $P$ along $\ell_0$ with shears corresponding to the weights of $\arcwt_0$ (as measured with respect to $p^
\ast$) to get a hyperbolic structure $Y_0 \in \cT(\Sigma)$.
Similarly, glue $Q$ with shears given by $L \arcwt_0$ (measured with respect to $q^
\ast$) to get a structure $Y_1$. 
Since the map $f:P \to Q$ was boundary-affine, and $Q$ is glued with exactly $L$ times the shears of $P$ (measured with respect to points which are sent to each other), we see that we can also glue $f$ along $\ell_0$ to yield an $L$-Lipschitz, boundary-tight map $Y_0 \to Y_1$.

We now show that the spines of the $Y_i$ are as specified.
By tracing through the argument, we see that the surface $Y_1$ can equivalently be specified by applying Construction \ref{constr:gluereg} to the weighted arc system
\[\underline{C}_1 := L \arcwt_0 + (\arcwt_1 - \arcwt_0)
= (L - 1) \arcwt_0 + \arcwt_1.\]
Proposition \ref{prop:glue_reg_deflate} produces a deflation map $Y_1 \to T(\underline{C}_1)$. 
So long as we take $L \ge 2$, all of the coefficients of $\underline{C}_1$ are at least $S_1 > S_0$,
and so we can apply Lemma \ref{lem:bigshears_closegeos} to deduce that the boundary geodesics of $Y_1$ connected by any arc of $\arc_1$ are at most $\varepsilon_1$ far apart.
Lemma \ref{lem:thickthin_deflation} then ensures the spine of $Y_1$ contains edges dual to the arcs of $\arc_1$. 
Applying Corollary \ref{cor:thin bands long edges}, we see that that for any arc $\alpha \in \arc_1$,
\[| c_\alpha(Y_1) - c_\alpha(\arcwt_1)| \le 
\max \{C_0, 4 \rad(Y_1)\}.\]
where $c_\alpha(\cdot)$ denotes the coefficient of $\alpha$ in each weighted arc system.
Moreover, $\arcwt(Y_1)$ cannot have any other arcs of large ($\gtrsim \max \{C_0, 4 \rad(Y_1)\}$) weight, else Corollary \ref{cor:thin bands long edges} would ensure that $T$ would also have to have a corresponding edge of comparable weight.
Thus, since $\rad(Y_1)$ is uniformly bounded (by $\varepsilon$ plus the max embedding radius of $\zeta$-almost-regular polygons),
\[\arcwt(Y_1) = \underline{C}_1 + O(1).\]

We now turn to $Y_0$. By Lemma \ref{lem:composite_deflation}, there is a deflation map
\[Y_0 \to 1/L \cdot T(\underline{C}_1) = T(\underline{C}_1/ L).\]
So long as $L$ is taken sufficiently large, the uniform upper bound on the coefficients of $\arcwt_1$ ensures that 
$\underline{C}_1/ L = \arcwt_0 + O(1)$.
We now wish to run through the same argument as before, but first we must check that we can apply Lemma \ref{lem:bigshears_closegeos} to our gluing of $P$.

The issue is that we have defined our gluing with respect to the $p^\ast$, not the spine basepoints.
What saves us is the second statement of that Lemma, applied to $Q$.
In particular, since we have glued with large enough shears, Lemma \ref{lem:bigshears_closegeos} tells us that the  $q^\ast$ are $O(\zeta + \varepsilon)$ close to the spine basepoints of $Q$. 
Since our boundary-tight map sends spine basepoints to spine basepoints, we have that the $p^\ast$ are 
$O\left((\zeta + \varepsilon)/L \right)$ close to the spine basepoints of $P$, which are $O(\zeta)$ close to each other.
Hence, large shears with respect to the $p^\ast$ translate to large shears with respect to spine basepoints of $P$, allowing us to apply Lemma \ref{lem:bigshears_closegeos}.

We now apply the same battery of statements as for $Y_1$: Lemma \ref{lem:thickthin_deflation} and Corollary \ref{cor:thin bands long edges} imply the spine of $Y_0$ have edges corresponding to $\arcwt_0$ with comparable weight, and another application of Corollary \ref{cor:thin bands long edges} ensures there cannot be any other long edges.
Thus,
\[\arcwt(Y_0) = \underline{C}_1/ L + O(1) 
= \arcwt_0 + O(1).\]

Finally, suppose that $\arcwt_0$ and $\arcwt_1$ were chosen in $\Base$. Then $\underline{C}_1 \in \Base$ as well, hence Proposition \ref{prop:reg_glue_res} implies that $Y_1$ must satisfy the total residue condition of Theorem \ref{thm:imageofcut}.
Residues are multiplicative under boundary-tight maps (Proposition \ref{prop:residues}) and the total residue condition is linear, so $Y_0$ also satisfies it.
Choosing any $X_0$ with $X_0 \setminus \lambda = Y_0$, Theorem \ref{thm:glueoptimal} then specifies a unique $X_1$ together with a tight map $X_0 \to X_1$ stretching $\lambda$ affinely.
This completes our proof when $n=1$.
\bigskip

For the general case, recall that we have picked systems of weights $\arcwt_j$ on $\arc_j$. 
Set $\arcwtb_j := \arcwt_j - \arcwt_{j-1}$ and let $\arcb_j$ denote the underlying (not necessarily filling!) arc system.
Define $\ell_j$ to be the geodesics on $\Sigma$ obtained by spinning the arcs of $\underline{\smash{\beta}}_j$, and set
\[\widehat\ell_j= \ell_0 \cup \ldots \cup \ell_j.\]
Observe that since $\arc_0$ is filling, then $\Sigma \setminus \widehat\ell_j$ is a union of ideal polygons for all $i = 0, \ldots, n$.
For each $j$, let $P(j)$ denote the disjoint union of these ideal polygons (considered as a topological surface), and note that $P(j) = P(j-1) \setminus \ell_j$. 
More generally, whenever $j' \ge j$, each component of $P(j)$ is a union of components of $P(j')$. Since we will not need to refer to exactly how the polygons are assembled out of each other, we will not introduce notation.

We now inductively define a family of hyperbolic structures $Q_i(j)$ on the $P(j)$ together with boundary-tight maps between them.
We then glue these polygons together along the $\smash{\widehat \ell_j}$ to yield our desired $Y_i$ (together with the maps between them).
Reference to the schematic in Figure \ref{fig:itin_tower} will be very useful throughout this proof; the reader should climb the staircase from bottom-right to top-left and fill in columns as she goes.

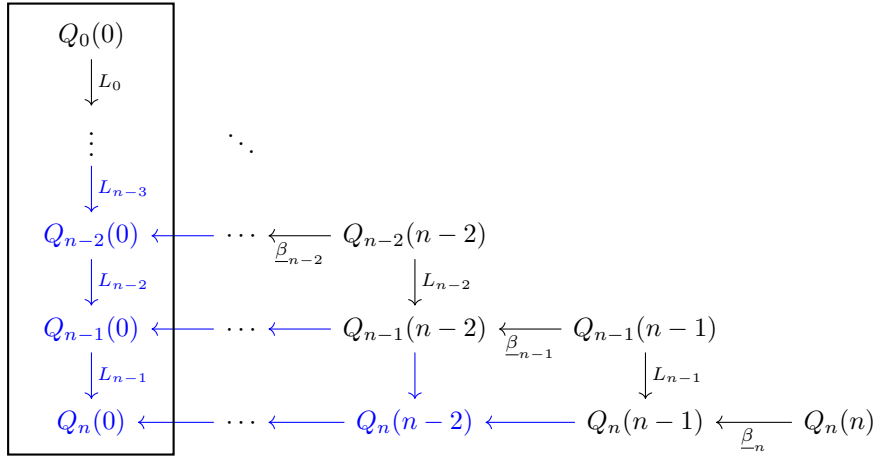
\begin{figure}[ht]
    \centering
$\begin{tikzcd}
[execute at end picture={
\draw[thick] (-3.8,-3) -- (-3.8,3) -- (-6,3) -- (-6,-3) -- cycle;
}]
Q_0(0) \ar{d}{L_0} &&&&
\\
\vdots \ar[blue]{d}{L_{n-3}} & \ddots
\\
{\color{blue} Q_{n-2}(0)}\ar[blue]{d}{L_{n-2}}
& \cdots \ar[blue]{l}
& Q_{n-2}(n-2)\arrow{l}{\underline{\beta}_{n-2}} \ar{d}{L_{n-2}}
\\
{\color{blue} Q_{n-1}(0)}\ar[blue]{d}{L_{n-1}}
& \cdots\ar[blue]{l}
& Q_{n-1}(n-2) \ar[blue]{d} \ar[blue]{l}
&Q_{n-1}(n-1) \arrow{l}{\underline{\beta}_{n-1}} \ar{d}{L_{n-1}}
\\
{\color{blue} Q_n(0)}
& \cdots \ar[blue]{l}
& {\color{blue} Q_n(n-2)} \ar[blue]{l}
& Q_n(n-1) \arrow[blue]{l}
& Q_n(n) \arrow{l}{\underline{\beta}_n}
\end{tikzcd}$
\caption{Constructing surfaces and maps for Theorem \ref{thm:itineraries}.
Each $Q_i(i)$ is a union of (nearly-)regular polygons. All horizontal arrows are gluings, labeled by the arc system corresponding to the glued leaves, while all vertical arrows are boundary-tight maps, labeled by their Lipschitz constant.}
\label{fig:itin_tower}
\end{figure}

First, we set $Q_n(n)$ to be the union of regular ideal polygons of type $P(n)$.
Glue together the polygons of $Q_n(n)$ along the leaves of $\ell_n$ with shears $\arcwtb_n$ (measured with respect to the inscribed regular horogons), yielding a hyperbolic structure $Q_n(n-1)$ on the union of polygons $P(n-1)$.
As in the $n=1$ case, let $q_n^\ast(n-1) \subset \partial Q_n(n-1)$ denote the collection of basepoints coming from the regular horogons of $Q_n(n)$.
By Corollary \ref{cor:scalespinedown}, for large enough $L_{n-1}$ there is a union of $\zeta$-nearly-regular polygons $Q_{n-1}(n-1)$ together with an $L_{n-1}$-Lipschitz boundary-tight map $Q_{n-1}(n-1) \to Q_{n}(n-1).$
Set $q_{n-1}^\ast(n-1) \subset \partial Q_{n-1}(n-1)$ to be the preimages of the basepoints $q_n^\ast(n-1)$ under this map.

Now we iterate: glue together the polygons of $Q_{n-1}(n-1)$ along $\ell_{n-1}$ using shears of size $\arcwtb_{n-1}$ (measured with respect to $q_{n-1}^\ast(n-1)$)
to define a hyperbolic structure $Q_{n-1}(n-2)$ on the polygons of $P(n-2)$. Let $q_{n-1}^\ast(n-2)$ denote the remaining basepoints on the boundary.
Corollary \ref{cor:scalespinedown} again ensures that there is a boundary-tight map from a union of nearly-regular polygons $Q_{n-2}(n-2)$ to this structure, so long as the Lipschitz constant $L_{n-2}$ is sufficiently large, and we define $q_{n-2}^\ast(n-2)$ by taking the preimage of basepoints.

Now we observe that we can also glue together the $Q_n(n-1)$ along $\ell_{n-1}$ with shears $L_{n-1} \arcwtb_{n-1}$ 
(measured with respect to $q_{n}^\ast(n-1)$)
to yield yet another hyperbolic structure $Q_{n}(n-2)$ on $P(n-2)$, and set $q_{n}^\ast(n-2)$ to be those basepoints which are still on the boundary (i.e., those which we did not use to measure shears).
Since the shears of $Q_n(n-2)$ along $\ell_{n-1}$ are exactly $L_{n-1}$ times the shears of $Q_{n-1}(n-2)$ along those same geodesics, measured with respect to points which are taken to each other, the maps $Q_{n-1}(n-1) \to Q_{n}(n-1)$ also glue to yield an $L_{n-1}$-Lipschitz boundary-tight map
\[Q_{n-1}(n-2) \to Q_{n}(n-2).\]

We continue this process, for each $i \ge j$ gluing the $Q_i(j)$ together along $\ell_j$ with shears
\begin{equation}
\label{eqn:shears}\text{shears} =\left\{
\begin{array}{cl}
\arcwtb_j & i = j \\
L_{i-1} \ldots L_{j} \cdot \arcwtb_j & i > j
\end{array}\right.
\end{equation}
(as measured with respect to basepoints $q_i^\ast(j)$).
The resulting hyperbolic structures $Q_{i}(j-1)$ on $P(j-1)$ 
come with a new collection of basepoints $q_i(j-1)$ as well as $L_{i-1}$-Lipschitz boundary-tight maps
$Q_{i-1}(j-1) \to Q_{i}(j-1)$, allowing us to move one column leftwards in Figure \ref{fig:itin_tower}.
Corollary \ref{cor:scalespinedown} allows us to move one row up, constructing an $L_{j-1}$-Lipschitz boundary-tight map from nearly-regular polygons $Q_{j-1}(j-1)$ to the glued $Q_j(j-1)$, and defining basepoints on its boundary by taking preimages.

This procedure terminates when $j=0$, yielding a sequence of hyperbolic structures $Q_i(0)$ on $P(0)$ together with boundary-tight maps
\[Q_0(0) \to Q_1(0) \to Q_2(0) \to \ldots \to Q_n(0),\]
which are boxed in Figure \ref{fig:itin_tower}.
Finally, we glue together each of the $Q_i(0)$ along the leaves of $\ell_0$ with shears as specified in \eqref{eqn:shears} (for $j=0$, measured with respect to the basepoints $q_i^\ast(0)$), resulting in a sequence of hyperbolic structures $Y_i$ on $\Sigma$ together with boundary-tight maps
\[Y_0 \to Y_1 \to Y_2 \to \ldots \to Y_n.\]
Note that the Lipschitz constant of the map $Y_i \to Y_{i+1}$ is achieved not only on the boundary, but also on all leaves of $\widehat \ell_i$.

With our construction complete, the verification of the conclusions of the Theorem is just as in the $n=1$ case.
Set $\underline{C}_0 := \arcwt_0$ and for $i \ge 1$, set
\begin{equation}\label{eqn:identarc}
 \underline{C}_i := 
\arcwtb_i
+ L_{i-1} \left( \arcwtb_{i-1}
+ L_{i-2}\left(\arcwtb_{i-2}
+ \cdots 
+ L_1 \bigg(\arcwtb_1
+ L_0 \arcwt_0 \bigg)
\cdots \right) \right).
\end{equation}
Observe that $\underline{C}_i$ is supported on $\arc_i$. Our goal is to show $\arcwt(Y_i) = \underline{C}_i + O(1).$

By design, the final surface $Y_n$ is the result of applying Construction \ref{constr:gluereg} to $\underline{C}_n$. Thus, Proposition \ref{prop:glue_reg_deflate} yields a deflation to the dual tree $Y_n \to T(\underline{C}_n)$, and applying Lemma \ref{lem:composite_deflation}, we get deflations
\[Y_i \to 1/(L_{n-1} \cdots L_i) \cdot  T( \underline{C}_n) = T(\underline{C}_n / (L_{n-1} \cdots L_i)).\]
Observe 
$\underline{C}_n / (L_{n-1} \cdots L_i) = \underline{C}_i + O(1).$

\begin{claim}
The arcs of $\arc_i$ are all $\varepsilon_1$ short on $Y_i$.
\end{claim}
\begin{proof}
Each $Y_i$ is constructed by gluing together $\zeta$-almost-regular polygons $Q_i(i)$ with large shears, so this should follow from Lemma \ref{lem:bigshears_closegeos}. 
As in the $n=1$ case, the only thing to check is that the points $q_i^\ast(i)$, which we used to define our gluing, are actually close to the spine basepoints of $Q_i(i)$, which appear in the statement of the Lemma.

For $i=n$ the points of the regular horogon are the same as the spine basepoints of $Q_n(n)$, by symmetry.
Thus, since $Q_{n-1}(n)$ is glued out of the $Q_n(n)$ with large shears, the second statement of Lemma \ref{lem:bigshears_closegeos} ensures that the points $q_{n}^\ast(n-1)$ are within $O(\zeta + \varepsilon)$ of the spine basepoints of $Q_{n}(n-1)$.
Now since the boundary-optimal map $Q_{n-1}(n-1)$ takes spine basepoints to spine basepoints, we conclude that the points $q_{n-1}^\ast(n-1)$ are even closer to the spine basepoints of $Q_{n-1}(n-1)$.

Iterating, we have that $Q_{n-1}(n-2)$ is glued out of $Q_{n-1}(n-1)$ with big shears with respect to the $q_{n-1}^\ast(n-1)$, and so by the paragraph above, we can say that the shears are still large when measured with respect to the spine basepoints.
Hence we can apply the second statement of Lemma \ref{lem:bigshears_closegeos} again to conclude that the points $q_{n-1}^\ast(n-2)$ are close to spine basepoints of $Q_{n-1}(n-2)$.
Pulling back by the Lipschitz map $Q_{n-2}(n-2) \to Q_{n-1}(n-2)$ we get that $q_{n-2}^\ast(n-2)$ are close to the spine basepoints of $Q_{n-2}(n-2)$, and the process repeats.

Thus, we conclude that for all $i$ the points we are using to define our shears are within $O(\zeta + \varepsilon)$ of the spine basepoints, hence with our choice of cutoff $S_1$ the shears when measured with respect to spine basepoints are still large enough to apply Lemma \ref{lem:bigshears_closegeos}.
\end{proof}

Since the arcs of $\arc_i$ are all short, Lemma \ref{lem:thickthin_deflation} ensures $\arc(Y_i)$ contains $\arc_i$, and Corollary \ref{cor:thin bands long edges} tells us their weight in $\arcwt(Y_i)$ is comparable to their weights in $\underline{C}_i$.
Moreover, there can be no other arcs of large weight, else they would appear in the target $T(\underline{C}_n / (L_{n-1} \cdots L_i))$. Thus, $\arcwt(Y_i) = \underline{C}_i + O(1).$

To verify the final statements of the Theorem, we note that $\underline{C}_n$ can be rewritten as a linear combination of the $\arcwt_i$. 
If each $\arcwt_i \in \Base$, so is $\underline{C}_n$, thus $Y_n$ satisfies the total residue condition of Theorem \ref{thm:imageofcut} (Proposition \ref{prop:reg_glue_res}).
We can then propagate up to all other $Y_i$ via multiplicativity of metric residues (Proposition \ref{prop:residues}).
Using Theorem \ref{thm:imageofcut}, we can choose some $X_0$ with $X_0 \setminus \lambda = Y_0$, and Theorem \ref{thm:glueoptimal} builds the desired surfaces $X_i$ together with tight maps $X_i \to X_{i+1}$ stretching $\lambda$ affinely by $L_i$. This completes the proof.
\end{proof}

\begin{remark}
The surfaces $Y_i$ cannot all be obtained via Construction \ref{constr:gluereg} applied to $\underline{C}_i$, as the polygons $Q_i(j)$ may have nonzero residues. This justifies the intricate induction used in the proof.
\end{remark}

\bibliographystyle{amsalpha}
\bibliography{references}

\end{document}